\documentclass[11pt,reqno]{article}
\usepackage[utf8]{inputenc}
\usepackage{geometry}
\usepackage{color}
\usepackage{calc}
\usepackage{bm}
\usepackage{amsmath}
\usepackage{amsthm}
\usepackage{amssymb}
\usepackage{graphicx}
\usepackage[all]{xy}
\PassOptionsToPackage{normalem}{ulem}
\usepackage{ulem}
\usepackage[unicode=true,pdfusetitle,
 bookmarks=true,bookmarksnumbered=false,bookmarksopen=false,
 breaklinks=false,pdfborder={0 0 1},backref=page,colorlinks=true]
 {hyperref}
\hypersetup{
 linkcolor=blue, citecolor=blue}

\makeatletter
\numberwithin{equation}{section}
\numberwithin{figure}{section}
\theoremstyle{plain}
\newtheorem{thm}{\protect\theoremname}[section]
\theoremstyle{plain}
\newtheorem{conjecture}[thm]{\protect\conjecturename}
\theoremstyle{remark}
\newtheorem{rem}[thm]{\protect\remarkname}
\theoremstyle{plain}
\newtheorem{lem}[thm]{\protect\lemmaname}
\theoremstyle{definition}
\newtheorem{defn}[thm]{\protect\definitionname}
\theoremstyle{definition}
\newtheorem{example}[thm]{\protect\examplename}
\theoremstyle{plain}
\newtheorem{prop}[thm]{\protect\propositionname}
\theoremstyle{remark}
\newtheorem{notation}[thm]{\protect\notationname}
\theoremstyle{plain}
\newtheorem{cor}[thm]{\protect\corollaryname}
\theoremstyle{remark}
\newtheorem*{rem*}{\protect\remarkname}

\usepackage{graphicx}
\usepackage{appendix}

\usepackage{enumitem}
\usepackage[backref=page]{hyperref}
\usepackage{multicol}
\renewcommand{\marginpar}[2][]{}
\usepackage{ytableau}

\makeatother

\providecommand{\conjecturename}{Conjecture}
\providecommand{\corollaryname}{Corollary}
\providecommand{\definitionname}{Definition}
\providecommand{\examplename}{Example}
\providecommand{\lemmaname}{Lemma}
\providecommand{\notationname}{Notation}
\providecommand{\propositionname}{Proposition}
\providecommand{\remarkname}{Remark}
\providecommand{\theoremname}{Theorem}

\begin{document}
\global\long\def\F{\mathrm{\mathbf{F}} }%
\global\long\def\Aut{\mathrm{Aut}}%
\global\long\def\C{\mathbf{C}}%
\global\long\def\H{\mathcal{H}}%
\global\long\def\U{\mathcal{U}}%
\global\long\def\ext{\mathrm{ext}}%
\global\long\def\hull{\mathrm{hull}}%
\global\long\def\triv{\mathrm{triv}}%
\global\long\def\Hom{\mathrm{Hom}}%

\global\long\def\trace{\mathrm{tr}}%
\global\long\def\End{\mathrm{End}}%

\global\long\def\L{\mathcal{L}}%
\global\long\def\W{\mathcal{W}}%
\global\long\def\E{\mathbb{E}}%
\global\long\def\SL{\mathrm{SL}}%
\global\long\def\R{\mathbf{R}}%
\global\long\def\Pairs{\mathrm{PowerPairs}}%
\global\long\def\Z{\mathbf{Z}}%
\global\long\def\rs{\to}%
\global\long\def\A{\mathcal{A}}%
\global\long\def\a{\mathbf{a}}%
\global\long\def\rsa{\rightsquigarrow}%

\global\long\def\b{\mathbf{b}}%
\global\long\def\df{\mathrm{def}}%
\global\long\def\eqdf{\stackrel{\df}{=}}%
\global\long\def\ZZ{\overline{Z}}%
\global\long\def\Tr{\mathrm{Tr}}%
\global\long\def\N{\mathbf{N}}%
\global\long\def\std{\mathrm{std}}%
\global\long\def\HS{\mathrm{H.S.}}%
\global\long\def\spec{\mathrm{spec}}%
\global\long\def\Ind{\mathrm{Ind}}%
\global\long\def\half{\frac{1}{2}}%
\global\long\def\Re{\mathrm{Re}}%
\global\long\def\Im{\mathrm{Im}}%
\global\long\def\Rect{\mathrm{Rect}}%
\global\long\def\Crit{\mathrm{Crit}}%
\global\long\def\Stab{\mathrm{Stab}}%
\global\long\def\SL{\mathrm{SL}}%
\global\long\def\Tab{\mathrm{Tab}}%
\global\long\def\Cont{\mathrm{Cont}}%
\global\long\def\I{\mathcal{I}}%
\global\long\def\J{\mathcal{J}}%
\global\long\def\short{\mathrm{short}}%
\global\long\def\Id{\mathrm{Id}}%
\global\long\def\B{\mathcal{B}}%
\global\long\def\ax{\mathrm{ax}}%
\global\long\def\cox{\mathrm{cox}}%
\global\long\def\row{\mathrm{row}}%
\global\long\def\col{\mathrm{col}}%
\global\long\def\X{\mathbb{X}}%
\global\long\def\Fat{\mathsf{Fat}}%

\global\long\def\V{\mathcal{V}}%
\global\long\def\P{\mathbb{P}}%
\global\long\def\Fill{\mathsf{Fill}}%
\global\long\def\fix{\mathsf{fix}}%
\global\long\def\reg{\mathrm{reg}}%
\global\long\def\edge{E}%
\global\long\def\id{\mathrm{id}}%
\global\long\def\emb{\mathrm{emb}}%

\global\long\def\Hom{\mathrm{Hom}}%
 
\global\long\def\F{\mathrm{\mathbf{F}} }%
  
\global\long\def\pr{\mathrm{Prob} }%
 
\global\long\def\tr{{\cal T}r }%
\global\long\def\gs{\mathsf{GS}}%
 
\global\long\def\Xcov{{\scriptscriptstyle \overset{\twoheadrightarrow}{\mathcal{\gs}}}}%
 
\global\long\def\covers{\leq_{\Xcov}}%
\global\long\def\core{\mathrm{Core}}%
\global\long\def\pcore{\mathrm{PCore}}%
\global\long\def\im{\mathrm{imm}}%
\global\long\def\br{\mathsf{BR}}%
\global\long\def\ebs{\mathsf{EBS}}%
\global\long\def\ev{\mathrm{ev}}%
\global\long\def\CC{\mathcal{C}}%
\global\long\def\sides{\mathrm{Sides}}%
\global\long\def\tp{\mathrm{top}}%
\global\long\def\lf{\mathrm{left}}%
\global\long\def\MCG{\mathrm{MCG}}%

\global\long\def\defect{\mathrm{Defect}}%
\global\long\def\M{\mathcal{M}}%
\global\long\def\O{\mathcal{O}}%
\global\long\def\EE{\mathcal{E}}%
\global\long\def\sign{\mathrm{sign}}%

\global\long\def\v{\mathfrak{v}}%
\global\long\def\e{\mathfrak{e}}%
\global\long\def\f{\mathfrak{f}}%
\global\long\def\D{\mathfrak{D}}%
\global\long\def\d{\mathfrak{d}}%
\global\long\def\he{\mathfrak{he}}%
\global\long\def\OVB{\mathsf{OvB}}%
\global\long\def\G{\mathcal{G}}%
 
\global\long\def\ts{\widetilde{\Sigma_{2}}}%

\title{A random cover of a compact hyperbolic surface has relative spectral
gap $\frac{3}{16}-\varepsilon$}
\author{Michael Magee, Frédéric Naud, Doron Puder}
\maketitle
\begin{abstract}
Let $X$ be a compact connected hyperbolic surface, that is, a closed
connected orientable smooth surface with a Riemannian metric of constant
curvature -1. For each $n\in\N$, let $X_{n}$ be a random degree-$n$
cover of $X$ sampled uniformly from all degree-$n$ Riemannian covering
spaces of $X$. An eigenvalue of $X$ or $X_{n}$ is an eigenvalue
of the associated Laplacian operator $\Delta_{X}$ or $\Delta_{X_{n}}$.
We say that an eigenvalue of $X_{n}$ is \emph{new }if it occurs with
greater multiplicity than in $X$. We prove that for any $\varepsilon>0$,
with probability tending to 1 as $n\to\infty$, there are no new eigenvalues
of $X_{n}$ below $\frac{3}{16}-\varepsilon$. We conjecture that
the same result holds with $\frac{3}{16}$ replaced by $\frac{1}{4}$.
\end{abstract}
\tableofcontents{}

\section{Introduction\label{sec:Introduction}}

Spectral gap is a fundamental concept in mathematics and related sciences
as it governs the rate at which a process converges towards its stationary
state. The question that motivates this paper is whether random objects
have large, or even optimal, spectral gaps. This will be made precise
below.

One of the simplest examples of spectral gap is the spectral gap of
a \emph{graph. }The spectrum of a graph $\mathcal{G}$ on $n$ vertices
is the collection of eigenvalues of its adjacency matrix $A_{\mathcal{G}}$.
Assuming that $\mathcal{G}$ is $d$-regular, the largest eigenvalue
occurs at $d$ and is simple if and only if $\mathcal{G}$ is connected.
This means, writing 
\[
\lambda_{0}=d\geq\lambda_{1}\geq\lambda_{2}\geq\cdots\geq\lambda_{n-1}
\]
for the eigenvalues of $A_{\mathcal{G}}$, then there is a \emph{spectral
gap} between $\lambda_{0}$ and $\lambda_{1}$ (i.e.~$\lambda_{0}>\lambda_{1}$)
if and only if $\mathcal{G}$ is connected. In fact, the Cheeger inequalities
for graphs due to Alon and Milman \cite{AlonMilman} show that the
size of the spectral gap (i.e.~$\lambda_{0}-\lambda_{1}$) quantifies
how difficult it is, roughly speaking, to separate the vertices of
$\mathcal{G}$ into two sets, each not too small, with few edges between
them. This is in tension with the fact that a $d$-regular graph is
\emph{sparse. }Sparse yet highly-connected graphs are called \emph{expander
graphs }and are relevant to many real-world examples\footnote{Following Barzdin and Kolmogorov \cite{Barzdin1993}, consider the
network of neurons in a human brain\@.}.

However, a result of Alon and Boppana \cite{Nilli} puts a sharp bound
on what one can achieve: for a sequence of $d$-regular graphs $\mathcal{G}_{n}$
on $n$ vertices, as $n\to\infty$, $\lambda_{1}(\mathcal{G}_{n})\geq2\sqrt{d-1}-o(1)$.
The trivial eigenvalues of a graph occur at $d$, and if $\mathcal{G}$
has a bipartite component, at $-d$. A connected $d$-regular graph
with all its non-trivial eigenvalues in the interval $[-2\sqrt{d-1},2\sqrt{d-1]}$
is called a \emph{Ramanujan graph }after Lubotzky, Phillips, and Sarnak
\cite{LPS}. 

\emph{In the rest of the paper, if an event depending on a parameter
$n$ holds with probability tending to 1 as $n\to\infty$, then we
say it holds asymptotically almost surely (a.a.s.).} A famous conjecture
of Alon \cite{Alon}, now a theorem due to Friedman \cite{Friedman},
states that for any $\varepsilon>0$, a.a.s.~a random $d$-regular
graph on $n$ vertices, chosen uniformly amongst such graphs, has
all its non-trivial eigenvalues bounded in absolute value by $2\sqrt{d-1}+\varepsilon$.
In other words, almost all $d$-regular graphs have almost optimal
spectral gaps. In \cite{bordenave2015new}, Bordenave has given a
shorter proof of Friedman's theorem. A first result about uniform
spectral gap for random regular graphs is due to Broder and Shamir
\cite{BroderShamir} who proved a.a.s.~$\lambda_{1}\leq3d^{3/4}$.
The approach in the current paper is similar to the direct trace method
introduced by Broder-Shamir, which was subsequently improved by Puder
and Friedman-Puder \cite{PUDER,friedman2020note} to show a.a.s.~that
$\lambda_{1}\leq2\sqrt{d-1}+\frac{2}{\sqrt{d-1}}$.

Friedman conjectured in \cite{FriedmanRelExp} that the following
extension of Alon's conjecture holds. Given any finite graph $\mathcal{G}$
there is a notion of a degree-$n$ cover\footnote{The precise definition of a cover of a graph is not important here;
only that it is analogous to a covering space of a surface.} $\mathcal{G}_{n}$ of the graph. Elements of the spectrum\footnote{We also take multiplicities into account in this statement.}
of $\mathcal{G}_{n}$ that are not elements of the spectrum of $\mathcal{G}$
are called \emph{new eigenvalues} of $\mathcal{G}_{n}$. Friedman
conjectured that for a fixed finite graph $\mathcal{G}$, for any
$\varepsilon>0$ a random degree-$n$ cover of $\mathcal{G}$ a.a.s.~has
no new eigenvalues of absolute value larger than $\rho(\mathcal{G})+\varepsilon$,
where $\rho(\mathcal{G})$ is the spectral radius of the adjacency
operator of the universal cover of $\mathcal{G}$, acting on $\ell^{2}$
functions\@. For $d$ even, the special case where $\G$ is a bouquet
of $\frac{d}{2}$ loops recovers Alon's conjecture. Friedman's conjecture
has recently been proved in a breakthrough by Bordenave and Collins
\cite{BordenaveCollins}.

\emph{The focus of this paper is the extension of Alon's and Friedman's
conjectures to compact hyperbolic surfaces. }

A \emph{hyperbolic surface} is a Riemannian surface of constant curvature
$-1$ without boundary. In this paper, all surfaces will be orientable.
By uniformization \cite[\S 9.2]{Beardon}, a connected compact hyperbolic
surface can be realized as $\Gamma\backslash\mathbb{H}$ where $\Gamma$
is a discrete subgroup of $\mathrm{PSL_{2}(\R)}$ and 
\begin{align*}
\mathbb{H} & =\{\,x+iy\,:\,x,y\in\R,\,y>0\,\}
\end{align*}
is the hyperbolic upper half plane, upon which $\mathrm{PSL_{2}(\R)}$
acts via Möbius transformations preserving the hyperbolic metric 
\[
\frac{dx^{2}+dy^{2}}{y^{2}}.
\]
Let $X=\Gamma\backslash\mathbb{H}$ be a connected compact hyperbolic
surface. Topologically, $X$ is a connected closed surface of some
genus $g\geq2$.

Since the Laplacian $\Delta_{\mathbb{H}}$ on $\mathbb{H}$ is invariant
under $\mathrm{PSL}_{2}(\R)$, it descends to a differential operator
on $C^{\infty}(X)$ and extends to a non-negative, unbounded, self-adjoint
operator $\Delta_{X}$ on $L^{2}(X)$. The spectrum of $\Delta_{X}$
consists of real eigenvalues
\[
0=\lambda_{0}(X)\leq\lambda_{1}(X)\leq\cdots\leq\lambda_{n}(X)\leq\cdots
\]
with $\lambda_{i}\to\infty$ as $i\to\infty$. The same discussion
also applies if we drop the condition that $X$ is connected\footnote{In which case $X$ is a finite union of connected compact hyperbolic
surfaces, each of which can be realized as a quotient of $\mathbb{H}$.}. We have $\lambda_{0}(X)<\lambda_{1}(X)$ if and only if $X$ is
connected, as for graphs\@. With Friedman's conjecture in mind, we
also note that the spectrum of $\Delta_{\mathbb{H}}$ is absolutely
continuous and supported on the interval $[\frac{1}{4},\infty)$ (e.g.~\cite[Thm. 4.3]{Borthwick}).
There is also an analog of the Alon-Boppana Theorem in this setting:
a result of Huber \cite{Huber} states that for any sequence of compact
hyperbolic surfaces $X_{i}$ with genera $g(X_{i})$ tending to infinity,
\[
\limsup_{i\to\infty}\lambda_{1}(X_{i})\leq\frac{1}{4}.
\]

To state an analog of the Alon/Friedman conjecture for surfaces, we
need a notion of a \emph{random cover. }Suppose $X$ is a compact
connected hyperbolic surface, and suppose $\tilde{X}$ is a degree-$n$
Riemannian cover of $X$. Fix a point $x_{0}\in X$ and label the
fiber above it by $[n]\eqdf\{1,\ldots,n\}$. There is a monodromy
map
\[
\pi_{1}(X,x_{0})\to S_{n}
\]
that describes how the fiber of $x_{0}$ is permuted when following
lifts of a closed loop from $X$ to $\tilde{X}$. Here $S_{n}$ is
the symmetric group of permutations of the set $[n]$. The cover $\tilde{X}$
is uniquely determined by the monodromy homomorphism. Let $g$ denote
the genus of $X$. We fix an isomorphism
\begin{equation}
\pi_{1}(X,x_{0})\cong\Gamma_{g}\eqdf\left\langle a_{1},b_{1},a_{2},b_{2},\ldots,a_{g},b_{g}\,\middle|\,\left[a_{1},b_{1}\right]\cdots\left[a_{g},b_{g}\right]=1\right\rangle .\label{eq:fundamental-group-isomorphism}
\end{equation}
Now, given any
\[
\phi\in\X_{g,n}\eqdf\Hom(\Gamma_{g},S_{n})
\]
we can construct a cover of $X$ whose monodromy map is $\phi$ as
follows. Using the fixed isomorphism of (\ref{eq:fundamental-group-isomorphism}),
we have a free properly discontinuous action of $\Gamma_{g}$ on $\mathbb{H}$
by isometries. Define a new action of $\Gamma_{g}$ on $\mathbb{H}\times[n]$
by
\[
\gamma(z,i)=(\gamma z,\phi[\gamma](i)).
\]
The quotient of $\mathbb{H}\times[n]$ by this action is named $X_{\phi}$
and is a hyperbolic cover of $X$ with monodromy $\phi$. This construction
establishes a one-to-one correspondence between $\phi\in\X_{g,n}$
and degree-$n$ covers with a labeled fiber $X_{\phi}$ of $X$. See
also Example \ref{exa:x_phi}.

As for graphs, any eigenvalue of $\Delta_{X}$ will also be an eigenvalue
of $\Delta_{X_{\phi}}$: every eigenfunction of $\Delta_{X}$ can
be pulled back to an eigenfunction of $\Delta_{X_{\phi}}$ with the
same eigenvalue. We say that an eigenvalue of $\Delta_{X_{\phi}}$
is \emph{new }if it is not one of $\Delta_{X},$ or more generally,
appears with greater multiplicity in $X_{\phi}$. To pick a random
cover of $X$, we simply use the uniform probability measure on the
finite set $\X_{g,n}$. Recall we say an event that pertains to any
$n$ holds a.a.s.~if it holds with probability tending to one as
$n\to\infty$. The analog of Friedman's conjecture for surfaces is
the following.
\begin{conjecture}
\label{conj:Alon-friedman-for-surfaces}Let $X$ be a compact connected
hyperbolic surface. Then for any $\varepsilon>0$, a.a.s. 
\[
\spec\left(\Delta_{X_{\phi}}\right)\cap\left[0,\frac{1}{4}-\varepsilon\right]=\spec\left(\Delta_{X}\right)\cap\left[0,\frac{1}{4}-\varepsilon\right]
\]
and the multiplicities on both sides are the same. 
\end{conjecture}

\begin{rem}
The analog of Conjecture \ref{conj:Alon-friedman-for-surfaces} for
finite area non-compact surfaces appeared previously in the work of
Golubev and Kamber \cite[Conj. 1.6(1)]{GolubevKamber}.
\end{rem}

\begin{rem}
We have explained the number $\frac{1}{4}$ in terms of the spectrum
of the Laplacian on the hyperbolic plane and as an asymptotically
optimal spectral gap in light of Huber's result \cite{Huber}. The
number $\frac{1}{4}$ also features prominently in Selberg's eigenvalue
conjecture \cite{SelbergFourier}, that states for $X=\SL_{2}(\Z)\backslash\mathbb{H},$
the (deterministic) family of congruence covers of $X$ never have
new eigenvalues below $\text{\ensuremath{\frac{1}{4}}}.$ Although
Selberg's conjecture is for a finite-area, non-compact hyperbolic
orbifold, the Jacquet-Langlands correspondence \cite{JacquetLanglands}
means that it also applies to certain arithmetic compact hyperbolic
surfaces.
\end{rem}

\begin{rem}
\label{rem:lambda1-one-quarter-remark}In \cite[Problem 10.4]{wright2020tour},
Wright asks, for random compact hyperbolic surfaces sampled according
to the Weil-Petersson volume form on the moduli space of genus $g$
closed hyperbolic surfaces, whether $\liminf_{g\to\infty}(\P(\lambda_{1}>\frac{1}{4}))>0$.
See $\S\S$\ref{subsec:Related-works} for what is known in this setting.
It is not even known \cite[Problem 10.3]{wright2020tour} whether
there is a sequence of Riemann surfaces $X_{n}$ with genus tending
to $\infty$ such that $\lambda_{1}(X_{n})\to\frac{1}{4}$. Conjecture
\ref{conj:Alon-friedman-for-surfaces} offers a new route to resolving
this problem via the probabilistic method, since it is known by work
of Jenni \cite{Jenni} that there exists a genus 2 hyperbolic surface
$X$ with $\lambda_{1}(X)>\frac{1}{4}$ and this $X$ can be taken
as the base surface in Conjecture \ref{conj:Alon-friedman-for-surfaces}.
(See $\S\S$\ref{subsec:Subsequent-results} for important developments
in this area after the current paper was written.)
\end{rem}

The main theorem of the paper, described in the title, is the following.
\begin{thm}
\label{thm:three-sixteenths}Let $X$ be a compact connected hyperbolic
surface. Then for any $\varepsilon>0$, a.a.s.
\[
\spec\left(\Delta_{X_{\phi}}\right)\cap\left[0,\frac{3}{16}-\varepsilon\right]=\spec\left(\Delta_{X}\right)\cap\left[0,\frac{3}{16}-\varepsilon\right]
\]
and the multiplicities on both sides are the same. 
\end{thm}

\begin{rem}
The appearance of the number $\frac{3}{16}$ in Theorem \ref{thm:three-sixteenths}
is essentially for the same reason that $\frac{3}{4}$ appears in
\cite{MageeNaud} (note that $\frac{3}{16}=\frac{3}{4}(1-\frac{3}{4})$,
and eigenvalues of the Laplacian are naturally parameterized as $s(1-s)$).
Ultimately, the appearance of $\frac{3}{4}$ can be traced back to
the method of Broder and Shamir \cite{BroderShamir} who proved that
a.a.s.~a random $2d$-regular graph on $n$ vertices has $\lambda_{1}\leq O\left(d^{3/4}\right)$,
using an estimate analogous to Theorem \ref{thm:effective-error}
below.
\end{rem}

\begin{rem}
More mysteriously, $\frac{3}{16}$ is also the lower bound that Selberg
obtained for the smallest new eigenvalue of a congruence cover of
the modular curve $\SL_{2}(\Z)\backslash\mathbb{H}$, in the same
paper \cite{SelbergFourier} as his eigenvalue conjecture. In this
context, the number arises ultimately from bounds on Kloosterman sums
due to Weil \cite{Weil} that follow from Weil's resolution of the
Riemann hypothesis for curves over finite fields. The state of the
art on Selberg's eigenvalue conjecture, after decades of intermediate
results \cite{GelbartJacquet,Iwaniec89,LRS,Iwaniec96,KimShahidi},
is due to Kim and Sarnak \cite{KIM} who produced a spectral gap of
size $\frac{975}{4096}$ for congruence covers of $\SL_{2}(\Z)\backslash\mathbb{H}$.
\end{rem}

It was pointed out to us by A.~Kamber that our methods also yield
the following estimate on the density of new eigenvalues of a random
cover.
\begin{thm}
\label{thm:density}Let 
\[
0\leq\lambda_{i_{1}}(X_{\phi})\leq\lambda_{i_{2}}(X_{\phi})\leq\cdots\leq\lambda_{i_{k(\phi)}}(X_{\phi})\leq\frac{1}{4}
\]
denote the collection of new eigenvalues of $\Delta_{X_{\phi}}$ of
size at most $\frac{1}{4}$, included with multiplicity. For each
of these, we write $\lambda_{i_{j}}=s_{i_{j}}(1-s_{i_{j}})$ with
$s_{i_{j}}=s_{i_{j}}(X_{\phi})\in\left[\frac{1}{2},1\right]$. For
any $\varepsilon>0$ and $\sigma\in\left(\frac{1}{2},1\right)$, a.a.s.
\begin{equation}
\#\left\{ 1\leq j\leq k(\phi)\,:\,\lambda_{i_{j}}<\sigma\left(1-\sigma\right)\right\} =\#\left\{ 1\leq j\leq k(\phi)\,:\,s_{i_{j}}>\sigma\right\} \leq n^{3-4\sigma+\varepsilon}.\label{eq:density-inequality}
\end{equation}
\end{thm}

\begin{rem}
The estimate (\ref{eq:density-inequality}) was established by Iwaniec
\cite[Thm 11.7]{Iwaniecbook} for congruence covers of $\SL_{2}(\Z)\backslash\mathbb{H}$.
Although Iwaniec's theorem has been generalized in various directions
\cite{Huxley,Sarnak,Humphries}, as far as we know, Iwaniec's result
has not been directly improved, so speaking about density of eigenvalues,
Theorem \ref{thm:density} establishes for random covers the best
result known in the arithmetic setting for eigenvalues above the Kim-Sarnak
bound $\frac{975}{4096}$ \cite{KIM}. Density estimates such as Theorem
\ref{thm:density} have applications to the cutoff phenomenon on hyperbolic
surfaces by work of Golubev and Kamber \cite{GolubevKamber}.
\end{rem}

We prove Theorems \ref{thm:three-sixteenths} and \ref{thm:density}
using Selberg's trace formula in $\S$\ref{sec:The-proof-of-spectral-gap-theorem}.
We use as a `black-box' in this method a statistical result (Theorem
\ref{thm:effective-error}) about the expected number of fixed points
of a fixed $\gamma\in\Gamma_{g}$ under a random $\phi$. 

If $\pi\in S_{n}$ then we write $\fix(\pi)$ for the number of fixed
points of the permutation $\pi$. Given an element $\gamma\in\Gamma_{g}$,
we let $\fix_{\gamma}$ be the function
\[
\fix_{\gamma}:\X_{g,n}\to\Z,\quad\fix_{\gamma}(\phi)\eqdf\fix(\phi(\gamma)).
\]
We write $\E_{g,n}[\fix_{\gamma}]$ for the expected value of $\fix_{\gamma}$
with respect to the uniform probability measure on $\X_{g,n}$. In
\cite{MPasympcover}, the first and third named authors proved the
following theorem.
\begin{thm}
\noindent \label{thm:asymptotic-non-effective}Let $g\geq2$ and $1\ne\gamma\in\Gamma_{g}$.
If $q\in\N$ is maximal such that $\gamma=\gamma_{0}^{~q}$ for some
$\gamma_{0}\in\Gamma_{g}$, then, as $n\to\infty$,
\begin{align*}
\E_{g,n}[\fix_{\gamma}] & =d(q)+O_{\gamma}\left(n^{-1}\right),
\end{align*}
where $d(q)$ is the number of divisors of $q$.
\end{thm}

In the current paper, we need an effective version of Theorem \ref{thm:asymptotic-non-effective}
that controls the dependence of the error term on $\gamma$. We need
this estimate only for $\gamma$ that are not a proper power. For
$\gamma\in\Gamma_{g}$, we write $\ell_{w}(\gamma)$ \label{ell_w}
for the \emph{cyclic-word-length }of $\gamma$, namely, for the length
of a shortest word in the generators $a_{1},b_{1},\ldots,a_{g},b_{g}$
of $\Gamma_{g}$ that represents an element in the conjugacy class
of $\gamma$ in $\Gamma_{g}$. The effective version of Theorem \ref{thm:asymptotic-non-effective}
that we prove here is the following.
\begin{thm}
\label{thm:effective-error}For each genus $g\geq2$, there is a constant
$A=A(g)$ such that for any $c>0$, if $1\neq\gamma\in\Gamma_{g}$
is not a proper power of another element in $\Gamma_{g}$ and $\ell_{w}(\gamma)\leq c\log n$
then
\[
\E_{g,n}[\fix_{\gamma}]=1+O_{c,g}\left(\frac{(\log n)^{A}}{n}\right).
\]
The implied constant in the big-O depends only on $c$ and $g$.
\end{thm}

\begin{rem}
In the rest of the paper, just to avoid complications in notation
and formulas that would obfuscate our arguments, we give the proof
of Theorem \ref{thm:effective-error} when $g=2$. The extension to
arbitrary genus is for the most part obvious: if it is not at some
point, we will point out the necessary changes.
\end{rem}

The proof of Theorem \ref{thm:effective-error} takes up the bulk
of the paper, spanning $\S$\ref{sec:rep-theory-symmetric}-$\S$\ref{sec:Proof-of-Effective-error-theorem}.
The proof of Theorem \ref{thm:effective-error} involves delving into
the proof of Theorem \ref{thm:asymptotic-non-effective} and refining
the estimates, as well as introducing some completely new ideas.

\subsection{Related works\label{subsec:Related-works}}

\paragraph*{The Brooks-Makover model}

The first study of spectral gap for random surfaces in the literature
is due to Brooks and Makover \cite{BrooksMakover} who form a model
of a random compact surface as follows. Firstly, for a parameter $n$,
they glue together $n$ copies of an ideal hyperbolic triangle where
the gluing scheme is given by a random trivalent ribbon graph. Their
model for this random ribbon graph is a modification of the Bollobás
bin model from \cite{Bollobas}. This yields a random finite-area,
non compact hyperbolic surface. Then they perform a compactification
procedure to obtain a random compact hyperbolic surface $X_{\mathrm{BM}}(n)$.
The genus of this surface is not deterministic, however. Brooks and
Makover prove that for this random model, there is a non-explicit
constant $C>0$ such that a.a.s.~(as $n\to\infty$)
\[
\lambda_{1}(X_{\mathrm{BM}}(n))\geq C.
\]
Theorem \ref{thm:three-sixteenths} concerns a different random model,
but improves on the Brooks-Makover result in two important ways: the
bound on new eigenvalues is explicit, and this bound is independent
of the compact hyperbolic surface $X$ with which we begin.

It is also worth mentioning a recent result of Budzinski, Curien,
and Petri \cite[Thm.$\:$1]{budzinski2021diameter} who prove that
the ratios 
\[
\frac{\mathrm{diameter}(X_{\mathrm{BM}}(n))}{\log n}
\]
converge to $2$ in probability as $n\to\infty$; they also observe
that this is not the optimal value by a factor of $2$. 

\paragraph*{The Weil-Petersson model }

Another reasonable model of random surfaces comes from the Weil-Petersson
volume form on the moduli space $\M_{g}$ of compact hyperbolic surfaces
of genus $g$. Let $X_{\mathrm{WP}}(g)$ denote a random surface in
$\M_{g}$ sampled according to the (normalized) Weil-Petersson volume
form. Mirzakhani proved in \cite[\S\S\S 1.2.I]{MirzakhaniRandom}
that with probability tending to 1 as $g\to\infty$,
\[
\lambda_{1}(X_{\mathrm{WP}}(g))\geq\frac{1}{4}\left(\frac{\log2}{2\pi+\log2}\right)^{2}\approx0.00247.
\]
We also note recent work of Monk \cite{Monk} who gives estimates
on the density of eigenvalues below $\frac{1}{4}$ of the Laplacian
on $X_{\mathrm{WP}}(g)$.

\paragraph*{Prior work of the authors}

In some sense, the closest result to Theorem \ref{thm:three-sixteenths}
in the literature is due to the first and second named authors of
the paper \cite{MageeNaud}, but it does not apply to compact surfaces,
rather to infinite area convex co-compact hyperbolic surfaces. Because
these surfaces have infinite area, their spectral theory is more involved.
We will focus on one result of \cite{MageeNaud} to illustrate the
comparison with this paper. 

Suppose $X$ is a connected non-elementary, non-compact, convex co-compact
hyperbolic surface. The spectral theory of $X$ is driven by a critical
parameter $\delta=\delta(X)\in(0,1)$. This parameter is both the
critical exponent of a Poincaré series and the Hausdorff dimension
of the limit set of $X$. If $\delta>\frac{1}{2}$ then results of
Patterson \cite{Patterson} and Lax-Phillips \cite{LP} say that the
bottom of the spectrum of $X$ is a simple eigenvalue at $\delta(1-\delta)$
and there are finitely many eigenvalues in the range $[\delta(1-\delta),\frac{1}{4})$.
In \cite{MageeNaud}, a model of a random degree-$n$ cover of $X$
was introduced that is completely analogous to the one used here;
the only difference in the construction is that the fundamental group
of $X$ is a free group $\F_{r}$ and hence one uses random $\phi\in\Hom(\F_{r},S_{n})$
to construct the random surface $X_{\phi}$. The following theorem
was obtained in \cite[Thm. 1.3.]{MageeNaud}.
\begin{thm}
\label{thm:L2-gap}Assume that $\delta=\delta(X)>\frac{1}{2}$. Then
for any $\sigma_{0}\in\left(\frac{3}{4}\delta,\delta\right)$, a.a.s.
\begin{equation}
\spec\left(\Delta_{X_{\phi}}\right)\cap\left[\delta\left(1-\delta\right),\sigma_{0}(1-\sigma_{0})\right]=\spec\left(\Delta_{X}\right)\cap\left[\delta\left(1-\delta\right),\sigma_{0}(1-\sigma_{0})\right]\label{eq:gap-1}
\end{equation}
and the multiplicities on both sides are the same.
\end{thm}

Although Theorem \ref{thm:L2-gap} is analogous to Theorem \ref{thm:three-sixteenths}
(for compact $X$, $\delta(X)=1$), the methods used in \cite{MageeNaud}
have almost no overlap with the methods used here. For infinite area
$X$, the fundamental group is free, so the replacement of Theorem
\ref{thm:effective-error} was already known by results of Broder-Shamir
\cite{BroderShamir} and the third named author \cite{PUDER}. The
challenge in \cite{MageeNaud} was to develop bespoke analytic machinery
to access these estimates.

Conversely, in the current paper, the needed analytic machinery already
exists (Selberg's trace formula) and rather, it is the establishment
of Theorem \ref{thm:effective-error} that is the main challenge here,
stemming from the non-free fundamental group $\Gamma_{g}$.

\subsection{Subsequent results \label{subsec:Subsequent-results}}

Since the preprint version of the current paper appeared in March
2020, several important results have been obtained in the area of
spectral gap of random surfaces. Independently of each other, Wu and
Xue \cite{wu2021random} and Lipnowski and Wright \cite{lipnowski2021optimal}
proved that for any $\varepsilon>0$, a Weil-Petersson random compact
hyperbolic surface of genus $g$ has spectral gap of size at least
$\frac{3}{16}-\varepsilon$ with probability tending to one as $g\to\infty$.
This result has been extended to the case of Weil-Petersson random
surfaces with not too many cusps by Hide in \cite{HideWP}.

In \cite{MageeNaud2}, the results of \cite{MageeNaud} have been
strengthened by the first and second named author to an (essentially
optimal) analog of Friedman's theorem for bounded frequency resonances
on infinite area Schottky surfaces.

Hide and the first named author have recently proved in \cite{HideMagee}
that the analog of Conjecture \ref{conj:Alon-friedman-for-surfaces}
for finite area non-compact hyperbolic surfaces holds true, and by
combining this result with a cusp removal argument of Buser, Burger,
and Dodziuk \cite{BBD}, in \cite{HideMagee} it is also proved that
there exist compact hyperbolic surfaces with genera tending to infinity
and $\lambda_{1}\to\frac{1}{4}$. (We have chosen to preserve Remark
\ref{rem:lambda1-one-quarter-remark} as originally written here for
posterity.)

\subsection{Structure of the proofs and the issues that arise\label{subsec:Overview-of-the-paper}}

\subsection*{Proof of Theorem \ref{thm:three-sixteenths} given Theorem \ref{thm:effective-error}}

First, we explain the outline of the proof of Theorem \ref{thm:three-sixteenths}
from Theorem \ref{thm:effective-error}. Theorem \ref{thm:density}
also follows from Theorem \ref{thm:effective-error} using the same
ideas. Both proofs are presented in full in $\S\ref{sec:The-proof-of-spectral-gap-theorem}$.

Our method of proving Theorem \ref{thm:three-sixteenths} is analogous
to the method of Broder and Shamir \cite{BroderShamir} for proving
that a random $2d$-regular graph has a large spectral gap. For us,
the Selberg trace formula replaces a more elementary formula for the
trace of a power of the adjacency operator of a graph in terms of
closed paths in the graph.

Let $\Gamma$ denote the fundamental group of $X$. By taking the
difference of the Selberg trace formula for $X_{\phi}$ and that for
$X$ we obtain a formula of the form
\begin{equation}
\sum_{\text{new eigenvalues \ensuremath{\lambda} of \ensuremath{X_{\phi}}}}F(\lambda)=\sum_{[\gamma]\in C(\Gamma)}G(\gamma)\left(\fix_{\gamma}(\phi)-1\right),\label{eq:trace-formula-simplicstic}
\end{equation}
where $C(\Gamma)$ is the collection of conjugacy classes in $\Gamma$,
and $F$ and $G$ are interdependent functions that depend on $n$.
The way we choose $F$ and $G$ together is to ensure
\begin{itemize}
\item $F(\lambda)$ is non-negative for any possible $\lambda$, and large
if $\lambda$ is an eigenvalue we want to forbid, and 
\item $G(\gamma)$ localizes to $\gamma$ with $\ell_{w}(\gamma)\leq c\log n$
for some $c=c(X)$.
\end{itemize}
By taking expectations of (\ref{eq:trace-formula-simplicstic}) we
obtain
\begin{equation}
\E\left[\sum_{\text{new eigenvalues \ensuremath{\lambda} of \ensuremath{X_{\phi}}}}F(\lambda)\right]=\sum_{[\gamma]\in C(\Gamma)}G(\gamma)\E\left[\fix_{\gamma}(\phi)-1\right].\label{eq:expected-trace-formula}
\end{equation}

The proof will conclude by bounding the right hand side and applying
Markov's inequality to conclude that there are no new eigenvalues
in the desired forbidden region. Since $G$ is well-controlled in
our proof, it remains to estimate each term $\E\left[\fix_{\gamma}(\phi)-1\right]$.
To do this, we echo Broder-Shamir \cite{BroderShamir} and partition
the summation on the right-hand side of (\ref{eq:expected-trace-formula})
according to three groups.
\begin{itemize}
\item If $\gamma$ is the identity, then $G(1)$ is easily analyzed, and
$\E\left[\fix_{\gamma}(\phi)-1\right]=n-1$.
\item If $\gamma$ is a proper power of a non-trivial element of $\Gamma$,
then we use a trivial bound $\E\left[\fix_{\gamma}(\phi)-1\right]\leq n-1$,
so we get no gain from the expectation. On the other hand, the contribution
to 
\[
\sum_{[\gamma]\in C(\Gamma)}G(\gamma)
\]
from these elements is negligible. Intuitively, this is because the
number of elements of $\Gamma$ with $\ell_{w}(\gamma)\leq L$ and
that are proper powers is (exponentially) negligible compared to the
total number of elements.
\item If $\gamma$ is not a proper power and not the identity, then we use
Theorem \ref{thm:effective-error} to obtain\\
 $\E\left[\fix_{\gamma}(\phi)-1\right]=O_{X}\left(\frac{(\log n)^{A}}{n}\right)$.
Thus for `most' summands in the right-hand side of (\ref{eq:expected-trace-formula})
we obtain a significant gain from the expectation.
\end{itemize}
Assembling all these estimates together gives a sufficiently upper
strong bound on (\ref{eq:expected-trace-formula}) to obtain Theorem
\ref{thm:three-sixteenths} via Markov's inequality.

\subsection*{Proof of Theorem \ref{thm:effective-error} }

To understand the proof of Theorem \ref{thm:effective-error}, we
suggest that the reader first read the overview below, then $\S$\ref{sec:Proof-of-Effective-error-theorem}
where all the components of the proof are brought together, and then
$\S$\ref{sec:tiled-surfaces}-$\S$\ref{sec:Estimates-for-the-probabilities-of-tiled-surfaces}
where the technical ingredients are proved. As throughout the paper,
we assume $g=2$ in this overview and we will forgo precision to give
a bird's-eye view of the proof.

Fixing an octagonal fundamental domain for $X$, any $X_{\phi}$ is
tiled by octagons; this tiling comes with some extra labelings of
edges corresponding to the generators of $\Gamma$. Any labeled 2-dimensional
CW-complex that can occur as a subcomplex of some $X_{\phi}$ is called
a tiled surface. For any tiled surface $Y$, we write $\E_{n}^{\emb}\left(Y\right)$
for the expected number, when $\phi$ is chosen uniformly at random
in $\Hom(\Gamma,S_{n})$, of embedded copies of $Y$ in $X_{\phi}$. 

In the previous paper \cite{MPasympcover}, we axiomatized certain
collections $\mathcal{R}$ of tiled surfaces, depending on $\gamma$,
that have the property that 
\begin{equation}
\E_{2,n}[\fix_{\gamma}]=\sum_{Y\in\mathcal{\mathcal{R}}}\E_{n}^{\emb}(Y).\label{eq:resolution-expectation}
\end{equation}
These collections are called \emph{resolutions. }Here we have oversimplified
the definitions to give an overview of the main ideas.

In \cite{MPasympcover}, we chose a resolution, depending on $\gamma$,
that consisted of two special types of tiled surfaces: those that
are boundary reduced or strongly boundary reduced. The motivation
for these definitions is that they make our methods for estimating
$\E_{n}^{\emb}(Y)$ more accurate. To give an example, if $Y$ is
strongly boundary reduced then we prove that for $Y$ fixed and $n\to\infty$,
we obtain\footnote{Some of the notation we use here is detailed in $\S\S$ \ref{subsec:Notation}.}
\begin{equation}
\E_{n}^{\emb}(Y)=n^{\chi(Y)}\left(1+O_{Y}\left(n^{-1}\right)\right).\label{eq:sbr-asymptotic}
\end{equation}
However, the implied constant depends on $Y$, and in the current
paper we have to control uniformly all $\gamma$ with $\ell_{w}(\gamma)\leq c\log n$.
The methods of \cite{MPasympcover} are not good enough for this goal.
To deal with this, we introduce in Definition \ref{def:e-adapted}
a new type of tiled surface called `$\varepsilon$-adapted' (for some
$\varepsilon\geq0$) that directly generalizes, and quantifies, the
concept of being strongly boundary reduced. We will explain the benefits
of this definition momentarily. We also introduce a new algorithm
called the \emph{octagons-vs-boundary} \emph{algorithm} that given
$\gamma$, produces a finite resolution $\mathcal{R}$ as in (\ref{eq:resolution-expectation})
such that every $Y\in\mathcal{R}$ is either
\begin{itemize}
\item $\varepsilon$-adapted for some $\varepsilon>0$, or
\item boundary reduced, with the additional condition that $\d\left(Y\right)<\f\left(Y\right)<-\chi(Y)$,
where $\d(Y)$ is the length of the boundary of $Y$ and $\f(Y)$
is the number of octagons in $Y$.
\end{itemize}
Any $Y\in\mathcal{R}$ has $\d(Y)\leq c'(\log n)$ and $\f(Y)\leq c'(\log n)^{2}$
given that $\ell_{w}(\gamma)\leq c\log n$ (Corollary \ref{cor:facts-about-the-resolution}).
The fact that we maintain control on these quantities during the algorithm
is essential. However, a defect of this algorithm is that we lose
control of how many $\varepsilon$-adapted $Y\in\mathcal{R}$ there
are of a given Euler characteristic. In contrast, in the algorithm
of \cite{MPasympcover} we control, at least, the number of elements
in the resolution of Euler characteristic zero. We later have to work
to get around this.

We run the octagons-vs-boundary algorithm for a fixed $\varepsilon=\frac{1}{32}$
to obtain a resolution $\mathcal{R}.$ Let us explain the benefits
of this resolution we have constructed. The $\varepsilon$-adapted
$Y\in\mathcal{R}$ contribute the main contributions to (\ref{eq:resolution-expectation}),
and the merely boundary reduced $Y$ contribute something negligible. 

Indeed, we prove for any boundary reduced $Y\in\mathcal{R}$ in the
regime of parameters we care about, that
\begin{equation}
\E_{n}^{\emb}(Y)\ll(A_{0}\f(Y))^{A_{0}\f(Y)}n^{\chi(Y)},\label{eq:br-bound-overview}
\end{equation}
where $A_{0}>0$. This bound (\ref{eq:br-bound-overview}) appears
in (\ref{eq:for-overview}) as the result of combining Corollary \ref{cor:size-of-X_n},
Theorem \ref{thm:E_n-emb-exact-expression}, Proposition \ref{prop:Xi-bound-final-BR}
and Lemma \ref{lem:D-vs-d}; the proof is by carefully effectivizing
the arguments of \cite{MPasympcover}.

While the bound (\ref{eq:br-bound-overview}) is quite bad (for example,
using it on all terms in (\ref{eq:resolution-expectation}) would
not even recover the results of \cite{MPasympcover}), the control
of the dependence on $\d(Y)$ is enough so that when combined with
$\d\left(Y\right)<\f\left(Y\right)<-\chi(Y)$ we obtain
\[
\E_{n}^{\emb}(Y)\ll(A_{0}\f(Y))^{A_{0}\f(Y)}n^{-\f(Y)}\ll\left(\frac{\left(c'(\log n)^{2}\right)^{A_{0}}}{n}\right)^{\f(Y)}.
\]
This is good enough that it can simply be combined with counting \emph{all
possible $Y$ }with $\d(Y)\leq c'(\log n)$ and $\f(Y)\leq c'(\log n)^{2}$
to obtain that the non-$\varepsilon$-adapted surfaces in $\mathcal{R}$
contribute $\ll\frac{(\log n)^{A}}{n}$ to (\ref{eq:resolution-expectation})
for $A>0$. This is Proposition \ref{prop:contributions-from-non-adapted}.

So from now on \uline{assume \mbox{$Y\in\mathcal{R}$} is \mbox{$\varepsilon$}-adapted}
and we explain how to control the contributions to (\ref{eq:resolution-expectation})
from these remaining $Y$. We first prove that there is a rational
function $Q_{Y}$ such that 
\begin{equation}
\E_{n}^{\emb}(Y)=n^{\chi(Y)}\left(Q_{Y}(n)+O\left(\frac{1}{n}\right)\right)\left(1+O\left(\frac{(\log n)^{2}}{n}\right)\right),\label{eq:E_n-to-Q}
\end{equation}
where the implied constants hold for any $\varepsilon$-adapted $Y\in{\cal R}$
as long as $\ell_{w}\left(\gamma\right)\le c\log n$ (Theorem \ref{thm:E_n-emb-exact-expression},
Proposition \ref{prop:Xi-bound-final-e-adapted} and Corollary \ref{cor:rational-form-Xi-dnu}).
In fact, this expression remains approximately valid for the same
$Y$ if $n$ is replaced throughout by $m$ with $m\approx(\log n)^{B}$
for some $B>0$; this will become relevant momentarily. 

The rational function $Q_{Y}$ is new to this paper; it appears through
Corollary \ref{cor:size-xn*} and Lemma \ref{lem:x_n^*-rational}
and results from refining the representation-theoretic arguments in
\cite{MPasympcover}. The description of $Q_{Y}$ is in terms of Stallings
core graphs \cite{stallings1983topology}, and related to the theory
of expected number of fixed points of words in the free group. In
the notation of the rest of the paper, 
\begin{equation}
Q_{Y}(n)=\frac{\left(n\right)_{\v(Y)}}{n^{\chi(Y)}}\sum_{H\in{\cal Q}\left(Y\right)}\frac{\left(n\right)_{\v(H)}}{\prod_{f\in\left\{ a,b,c,d\right\} }(n)_{\e_{f}(H)}},\label{eq:Q_Y-expression}
\end{equation}
where $\mathcal{Q}(Y)$ is a collection of core graphs obtained by
adding handles to the one-skeleton of $Y$, performing `folding' operations,
and taking quotients in a particular way (see $\S\S\ref{subsec:UnderstandingX*_n}$
for details). 

The argument leading to (\ref{eq:E_n-to-Q}) involves isolating some
of the terms that contribute to $\E_{n}^{\emb}(Y)$, and reinterpreting
these as related to the size of a set $\X_{n}^{*}(Y,\J)$ of maps
$\F_{4}\to S_{n}$ that contain, in an appropriate sense, an embedded
copy of $Y$ but only satisfy the relation of $\Gamma$ modulo $S_{n-\v(Y)}$
rather than absolutely (Proposition \ref{prop:Xi_n-refined}). Then
by topological arguments the set $\X_{n}^{*}(Y,\J)$ is counted in
terms of core graphs leading to Lemma \ref{lem:x_n^*-rational} that
gives (\ref{eq:Q_Y-expression}) here.

One unusual thing is that our combinatorial description of $Q_{Y}$
does not immediately tell us the order of growth of $Q_{Y}(n)$, because
we do not know much about $\mathcal{Q}(Y)$. On the other hand, we
know enough about $Q_{Y}$ (for example, for what range of parameters
it is positive) so that we can `black-box' results from \cite{MPasympcover}
to learn that if $Y$ is fixed and $n\to\infty$, $Q_{Y}(n)\to1$.
(We also learn from this argument the interesting topological fact
that there is exactly one element of $\mathcal{Q}(Y)$ of maximal
Euler characteristic.)

This algebraic properties of $Q_{Y}$, together with a priori facts
about $Q_{Y}$, allow us to use (\ref{eq:E_n-to-Q}) to establish
the two following important inequalities: 
\begin{align}
\E_{n}^{\emb}(Y) & =n^{\chi(Y)}\left(1+O_{c}\left(\frac{(\log n)^{4}}{n}\right)+O\left(\frac{m}{n}\frac{\E_{m}^{\emb}(Y)}{m^{\chi(Y)}}\right)\right)\label{eq:n-to-m2}\\
n^{\chi(Y)} & \ll\frac{m}{n}\E_{m}^{\emb}(Y),\quad\text{if \ensuremath{\chi(Y)<0}}\label{eq:n-to-m1}
\end{align}
where $m\approx(\log n)^{B}$ is \uline{much smaller} than $n$.
These inequalities are provided by Proposition \ref{prop:main-term-asymp}
and Corollary \ref{cor:Em-lower-bound} (see also Remark \ref{rem:extra-remark}).
While (\ref{eq:n-to-m1}) may look surprising, its purpose is for
running our argument in reverse with decreased parameters as explained
below.

Let us now explain precisely the purpose of (\ref{eq:n-to-m1}) and
(\ref{eq:n-to-m2}). By black-boxing the results of \cite{MPasympcover}
one more time, we learn that there is exactly one $\varepsilon$-adapted
$Y\in\mathcal{R}$ with $\chi(Y)=0$, and none with $\chi(Y)>0$.
This single $Y$ with $\chi(Y)=0$ contributes the main term of Theorems
\ref{thm:asymptotic-non-effective} and \ref{thm:effective-error}
through (\ref{eq:n-to-m2}). Any other term coming from $\varepsilon$-adapted
$Y$ can be controlled in terms of $\E_{m}^{\emb}(Y)$ using (\ref{eq:n-to-m2})
and (\ref{eq:n-to-m1}). These errors could accumulate, but we can
control them \uline{all at once} by using (\ref{eq:resolution-expectation})
in reverse with $n$ replaced by $m$ to obtain

\[
\sum_{Y\in\mathcal{R}}\E_{m}^{\emb}(Y)=\E_{2,m}[\fix_{\gamma}]\leq m\approx(\log n)^{B}.
\]
Putting the previous arguments together proves Theorem \ref{thm:effective-error}.

\subsection{Notation\label{subsec:Notation}}

The commutator of two group elements is $\left[a,b\right]\eqdf aba^{-1}b^{-1}.$
For $m,n\in\N$, $m\leq n$, we use the notation $[m,n]$ for the
set $\{m,m+1,\ldots,n\}$ and $[n]$ for the set $\{1,\ldots,n\}$.
For $q,n\in\N$ with $q\leq n$ we use the Pochammer symbol 
\[
(n)_{q}\eqdf n(n-1)\cdots(n-q+1).
\]
For real-valued functions $f,g$ that depend on a parameter $n$ we
write $f=O(g)$ to mean there exist constants $C,N>0$ such that for
$n>N$, $|f(n)|\leq Cg(n)$. We write $f\ll g$ if there are $C,N>0$
such that $f(n)\leq Cg(n)$ for $n>N$. We add constants as a subscript
to the big $O$ or the $\ll$ sign to mean that the constants $C$
and $N$ depend on these other constants, for example, $f=O_{\varepsilon}(g)$
means that both $C=C(\varepsilon)$ and $N=N(\varepsilon)$ may depend
on $\varepsilon$. If there are no subscripts, it means the implied
constants depend only on the genus $g$, which is fixed throughout
most of the paper. We use the notation $f\asymp g$ to mean $f\ll g$
and $g\ll f$; the use of subscripts is the same as before.

\subsection*{Acknowledgments}

We thank Amitay Kamber for the observation that our methods prove
Theorem 1.7. We thank the anonymous referees for their careful readings
and insightful comments.

Frédéric Naud is supported by Institut Universitaire de France. Doron
Puder was supported by the Israel Science Foundation: ISF grant 1071/16.
This project has received funding from the European Research Council
(ERC) under the European Union’s Horizon 2020 research and innovation
programme (grant agreement No 850956 and grant agreement No 949143).

\section{The proof of Theorem \ref{thm:three-sixteenths} given Theorem \ref{thm:effective-error}
\label{sec:The-proof-of-spectral-gap-theorem}}

\subsection{Selberg's trace formula and counting closed geodesics}

Here we describe the main tool of this $\S$\ref{sec:The-proof-of-spectral-gap-theorem}:
Selberg's trace formula for compact hyperbolic surfaces. Let $C_{c}^{\infty}(\R)$
denote the infinitely differentiable real functions on $\R$ with
compact support. Given an \textit{\emph{even}} function $\varphi\in C_{c}^{\infty}(\R)$,
its Fourier transform is defined by
\[
\widehat{\varphi}(\xi)\eqdf\int_{-\infty}^{\infty}\varphi(x)e^{-ix\xi}dx
\]
for any $\xi\in\C$. As $\varphi\in C_{c}^{\infty}(\R)$, the integral
above converges for all $\xi\in\C$ to an entire function. 

Given a compact hyperbolic surface $X$, we write $\mathcal{L}(X)$
for the set of closed oriented geodesics in $X$. A geodesic is called
\emph{primitive} if it is not the result of repeating another geodesic
$q$ times for $q\geq2$. Let $\mathcal{P}(X)$ denote the set of
closed oriented primitive\emph{ }geodesics on $X$. Every closed geodesic
$\gamma$ has a length $\ell(\gamma)$ according to the hyperbolic
metric on $X$. Every closed oriented geodesic $\gamma\in\L(X)$ determines
a conjugacy class $[\tilde{\gamma}]$ in $\pi_{1}(X,x_{0})$ for any
basepoint $x_{0}$. Clearly, a closed oriented geodesic in $X$ is
primitive if and only if the elements of the corresponding conjugacy
class are not proper powers in $\pi_{1}(X,x_{0})$. For $\gamma\in\L(X)$
we write $\Lambda(\gamma)=\ell(\gamma_{0})$ where $\gamma_{0}$ is
the unique primitive closed oriented geodesic such that $\gamma=\gamma_{0}^{q}$
for some $q\geq1$.

We now give Selberg's trace formula for a compact hyperbolic surface
in the form of \cite[Thm. 9.5.3]{Buser} (see Selberg \cite{Selberg}
for the original appearance of this formula and Hejhal \cite{Hejhal1,Hejhal2}
for an encyclopedic treatment). 
\begin{thm}[Selberg's trace formula]
\label{thm:Selberg-trace-formual}Let $X$ be a compact hyperbolic
surface and let 
\[
0=\lambda_{0}(X)\leq\lambda_{1}(X)\leq\cdots\leq\lambda_{n}(X)\leq\cdots
\]
denote the spectrum of the Laplacian on $X$. For $i\in\N\cup\{0\}$
let
\[
r_{i}(X)\eqdf\begin{cases}
\sqrt{\lambda_{i}(X)-\frac{1}{4}} & if\ \lambda_{i}(X)>1/4\\
i\sqrt{\frac{1}{4}-\lambda_{i}(X)} & if\ \lambda_{i}(X)\leq1/4
\end{cases}.
\]
Then for any even $\varphi\in C_{c}^{\infty}(\R)$
\[
\sum_{i=0}^{\infty}\widehat{\varphi}(r_{i}(X))=\frac{\mathrm{\mathrm{area}}(X)}{4\pi}\int_{-\infty}^{\infty}r\widehat{\varphi}(r)\tanh(\pi r)dr+\sum_{\gamma\in\L(X)}\frac{\Lambda(\gamma)}{2\sinh\left(\frac{\ell(\gamma)}{2}\right)}\varphi(\ell(\gamma)).
\]
(Both sides of the formula are absolutely convergent.)
\end{thm}

We will also need a bound on the number of closed oriented geodesics
with length $\ell(\gamma)\leq T$. In fact we only need the following
very soft bound from e.g.~\cite[Lem.~9.2.7]{Buser}.
\begin{lem}
\label{lem:coarse-counting-bound}For a compact hyperbolic surface
$X$, there is a constant $C=C(X)$ such that
\[
\left|\left\{ \gamma\in\L(X)\,:\,\ell(\gamma)\leq T\right\} \right|\leq Ce^{T}.
\]
\end{lem}

Much sharper versions of this estimate are known, but Lemma \ref{lem:coarse-counting-bound}
suffices for our purposes. 

Suppose that $X$ is a connected compact hyperbolic surface. We fix
a basepoint $x_{0}\in X$ and an isomorphism $\pi_{1}(X,x_{0})\cong\Gamma_{g}$
as in (\ref{eq:fundamental-group-isomorphism}) where $g\geq2$ is
the genus of $X$. If $\gamma$ is a closed oriented geodesic, by
abuse of notation we let $\ell_{w}\left(\gamma\right)$ denote the
minimal word-length of an element in the conjugacy class in $\Gamma_{g}$
specified by $\gamma$ (on page \pageref{ell_w} we used the same
notation for an element of $\Gamma_{g}$). We want to compare $\ell(\gamma)$
and $\ell_{w}(\gamma)$. We will use the following simple consequence
of the \u{S}varc-Milnor lemma \cite[Prop. 8.19]{BridsonHaefliger}.
\begin{lem}
\label{lem:length-comparison}With notations as above, there exist
constants $K_{1},K_{2}\geq0$ depending on $X$ such that 
\[
\ell_{w}(\gamma)\leq K_{1}\ell(\gamma)+K_{2}.
\]
\end{lem}

\subsection{Choice of function for use in Selberg's trace formula}

We now fix a function $\varphi_{0}\in C_{c}^{\infty}(\mathbb{\R})$
which has the following key properties:
\begin{enumerate}
\item $\varphi_{0}$ is non-negative and even.
\item $\mathrm{Supp}(\varphi_{0})=(-1,1)$.
\item The Fourier transform $\widehat{\varphi_{0}}$ satisfies $\widehat{\varphi_{0}}(\xi)\geq0$
for all $\xi\in\mathbb{\R}\cup i\mathbb{\R}$.
\end{enumerate}
\begin{proof}[Proof that such a function exists]
Let $\psi_{0}$ be a $C^{\infty}$, even, real-valued non-negative
function whose support is exactly $(-\frac{1}{2},\frac{1}{2})$. Let
$\varphi_{0}\eqdf\psi_{0}\star\psi_{0}$ where
\[
\psi_{0}\star\psi_{0}(x)\eqdf\int_{\mathbb{\R}}\psi_{0}(x-t)\psi_{0}(t)dt.
\]
 Then $\varphi_{0}$ has the desired properties. 
\end{proof}
We now fix a function $\varphi_{0}$ as above and for any $T>0$ define
\[
\varphi_{T}(x)\eqdf\varphi_{0}\left(\frac{x}{T}\right).
\]

\begin{lem}
\label{lem:lower-bound-on-fourier}For all $\varepsilon>0$, there
exists $C_{\varepsilon}>0$ such that for all $t\in\mathbb{\R}_{\geq0}$
and for all $T>0$
\[
\widehat{\varphi_{T}}(it)\geq C_{\varepsilon}Te^{T(1-\varepsilon)t}.
\]
\end{lem}

\begin{proof}
First observe that 
\[
\widehat{\varphi_{T}}(it)=T\widehat{\varphi_{0}}(Tit)=T\int_{\mathbb{\R}}\varphi_{0}(x)e^{Txt}dx.
\]
Using $t\geq0$ and $\mathrm{Supp}(\varphi_{0})=(-1,1)$ with $\varphi_{0}$
non-negative, we have for some $C_{\varepsilon}>0$
\[
\widehat{\varphi_{T}}(it)\geq T\int_{1-\varepsilon}^{1}\varphi_{0}(x)e^{Txt}dx\geq TC_{\varepsilon}e^{T(1-\varepsilon)t}.
\]
\end{proof}

\subsection{Proof of Theorem \ref{thm:three-sixteenths}}

Let $X$ be a genus $g$ compact hyperbolic surface and let $X_{\phi}$
be the cover of $X$ corresponding to $\phi\in\Hom(\Gamma_{g},S_{n})$
constructed in the introduction. In what follows we let 
\[
T=4\log n.
\]
For every $\gamma\in\L(X)$, we pick $\tilde{\gamma}\in\Gamma_{g}$
in the conjugacy class in $\Gamma_{g}$ corresponding to $\gamma$
(so in particular $\ell_{w}\left(\tilde{\gamma}\right)=\ell_{w}\left(\gamma\right)$).
Every closed oriented geodesic $\delta$ in $X_{\phi}$ covers, via
the Riemannian covering map $X_{\phi}\to X,$ a unique closed oriented
geodesic in $X$ that we will call $\pi(\delta)$. This gives a map
\[
\pi:\L(X_{\phi})\to\L(X).
\]
Note that $\ell(\delta)=\ell(\pi(\delta)).$ We claim that $|\pi^{-1}(\gamma)|=\fix_{\tilde{\gamma}}(\phi)$,
recalling that $\fix_{\tilde{\gamma}}(\phi)$ is the number of fixed
points of $\phi(\tilde{\gamma})$. Indeed, by its very definition,
$X_{\phi}$ is a fiber bundle over $X$ with fiber $[n]$. If $\gamma\in\mathcal{P}(X)$,
and we fix some regular point $o\in\gamma$ (not a self-intersection
point), then in $X_{\phi}$, the fiber of $o$ can be identified with
$[n]$. The oriented geodesic path $\gamma\backslash\{o\}$ lifts
to $n$ oriented geodesic paths with start and end points equal to
$[n]$. The permutation of $[n]$ obtained by following these from
start to end is (up to conjugation) $\phi(\tilde{\gamma})$ and hence,
the $\delta$'s with $\pi(\delta)=\gamma$ are precisely the paths
that close up, or in other words, the $\delta$'s with $\pi(\delta)=\gamma$
correspond to fixed points of $\phi(\tilde{\gamma})$. For general
$\gamma\in\L\left(X\right)$, assume that $\gamma=\gamma_{0}^{~q}$
with $q\ge1$ and $\gamma_{0}\in\mathcal{P}\left(X\right)$. A similar
argument shows that the elements in $\pi^{-1}\left(\gamma\right)$
are in bijection with fixed points of $\tilde{\gamma_{0}}^{q}$ which
we may take as our $\tilde{\gamma}$.

We also have $\mathrm{\mathrm{area}}(X_{\phi})=n\cdot\mathrm{\mathrm{area}}(X)$.
Now applying Theorem \ref{thm:Selberg-trace-formual} to $X_{\phi}$
with the function $\varphi_{T}$ gives
\begin{align*}
\sum_{i=0}^{\infty}\widehat{\varphi_{T}}(r_{i}(X_{\phi})) & =\frac{\mathrm{area}(X_{\phi})}{4\pi}\int_{-\infty}^{\infty}r\widehat{\varphi_{T}}(r)\tanh(\pi r)dr+\sum_{\delta\in\L(X_{\phi})}\frac{\Lambda(\delta)}{2\sinh\left(\frac{\ell(\delta)}{2}\right)}\varphi_{T}(\ell(\delta))\\
 & =\frac{n\cdot\mathrm{area}(X)}{4\pi}\int_{-\infty}^{\infty}r\widehat{\varphi_{T}}(r)\tanh(\pi r)dr+\sum_{\gamma\in\L(X)}\sum_{\delta\in\pi^{-1}(\gamma)}\frac{\Lambda(\delta)}{2\sinh\left(\frac{\ell(\gamma)}{2}\right)}\varphi_{T}(\ell(\gamma))\\
 & =\frac{n\cdot\mathrm{area}(X)}{4\pi}\int_{-\infty}^{\infty}r\widehat{\varphi_{T}}(r)\tanh(\pi r)dr+\sum_{\gamma\in\mathcal{P}(X)}\frac{\fix_{\tilde{\gamma}}(\phi)\ell(\gamma)}{2\sinh\left(\frac{\ell(\gamma)}{2}\right)}\varphi_{T}(\ell(\gamma))\\
 & ~~~+\sum_{\gamma\in\mathcal{\L}(X)-\mathcal{P}(X)}\sum_{\delta\in\pi^{-1}(\gamma)}\frac{\Lambda(\delta)}{2\sinh\left(\frac{\ell(\gamma)}{2}\right)}\varphi_{T}(\ell(\gamma)),
\end{align*}
where in the second equality we used the fact that for $\delta\in\L\left(X_{\phi}\right)$,
$\ell\left(\delta\right)=\ell\left(\pi\left(\delta\right)\right)$,
and in the third equality we used that if $\gamma\in\mathcal{P}(X)$,
then $\delta\in\mathcal{P}(X_{\phi})$ for all $\delta\in\pi^{-1}(\gamma)$,
so $\Lambda(\delta)=\Lambda(\gamma)=\ell(\gamma)$. Let $i_{1},i_{2},i_{3},\ldots$
be a subsequence of $1,2,3,\ldots$ such that 
\[
0\leq\lambda_{i_{1}}(X_{\phi})\leq\lambda_{i_{2}}(X_{\phi})\leq\cdots
\]
are the \emph{new }eigenvalues of $X_{\phi}$. Thus $\lambda_{i_{1}}\left(X_{\phi}\right)$
is the smallest new eigenvalue of $X_{\phi}$. Taking the difference
of the above formula with the trace formula for $X$ (with the same
function $\varphi_{T}$) gives 
\begin{align}
\sum_{j=1}^{\infty}\widehat{\varphi_{T}}(r_{i_{j}}(X_{\phi})) & =\frac{(n-1)\cdot\mathrm{area}(X)}{4\pi}\int_{-\infty}^{\infty}r\widehat{\varphi_{T}}(r)\tanh(\pi r)dr+\sum_{\gamma\in\mathcal{P}(X)}\frac{(\fix_{\tilde{\gamma}}(\phi)-1)\ell(\gamma)}{2\sinh\left(\frac{\ell(\gamma)}{2}\right)}\varphi_{T}(\ell(\gamma))\nonumber \\
 & ~~~+\sum_{\gamma\in\mathcal{\L}(X)-\mathcal{P}(X)}\frac{\varphi_{T}\left(\ell(\gamma)\right)}{2\sinh\left(\frac{\ell(\gamma)}{2}\right)}\left(\left(\sum_{\delta\in\pi^{-1}(\gamma)}\Lambda(\delta)\right)-\Lambda(\gamma)\right).\label{eq:trace-formula-difference}
\end{align}
Since $\varphi_{T}$ is non-negative and for any $\gamma\in\L(X)$,
$|\pi^{-1}(\gamma)|\leq n$, and $\Lambda(\delta)\leq\ell\left(\delta\right)=\ell(\gamma)$
for all $\delta\in\pi^{-1}(\gamma)$, the sum on the bottom line of
(\ref{eq:trace-formula-difference}) is bounded from above by
\begin{eqnarray}
n\sum_{\gamma\in\mathcal{\L}(X)-\mathcal{P}(X)}\frac{\varphi_{T}\left(\ell(\gamma)\right)}{2\sinh\left(\frac{\ell(\gamma)}{2}\right)}\cdot\ell(\gamma) & = & n\sum_{\gamma\in\mathcal{P}(X)}\sum_{k=2}^{\infty}\frac{\varphi_{T}\left(k\ell(\gamma)\right)}{2\sinh\left(\frac{k\ell(\gamma)}{2}\right)}k\ell(\gamma).\label{eq:error-in-tf-diff}
\end{eqnarray}
We have
\begin{equation}
\sum_{k=2}^{\infty}\frac{\varphi_{T}\left(k\ell(\gamma)\right)}{2\sinh\left(\frac{k\ell(\gamma)}{2}\right)}k\ell(\gamma)\stackrel{\left(*\right)}{\ll}_{X}\ell(\gamma)\sum_{k=2}^{\infty}ke^{-\frac{k\ell(\gamma)}{2}}\stackrel{\left(**\right)}{\ll}_{X}\ell(\gamma)e^{-\ell(\gamma)},\label{eq:sum-over-k-est}
\end{equation}
where in $\left(*\right)$ we relied on that $\varphi_{T}$ is bounded,
and in both $\left(*\right)$ and $\left(**\right)$ on that there
is a positive lower bound on the lengths of closed geodesics in $X$.
As $\varphi_{T}$ is supported on $\left(-T,T\right)$, the left hand
side of (\ref{eq:sum-over-k-est}) vanishes whenever $\ell\left(\gamma\right)\ge T/2$.
Using Lemma \ref{lem:coarse-counting-bound} we thus get
\begin{align}
n\sum_{\gamma\in\mathcal{P}(X)}\sum_{k=2}^{\infty}\frac{\varphi_{T}\left(k\ell(\gamma)\right)}{2\sinh\left(\frac{k\ell(\gamma)}{2}\right)}k\ell(\gamma) & \ll_{X}n\sum_{\gamma\in\mathcal{P}(X):\ell(\gamma)\leq T}\ell(\gamma)e^{-\ell(\gamma)}\nonumber \\
 & \leq n\sum_{m=0}^{T}\sum_{\gamma\in\L(X)\,:\,m\leq\ell(\gamma)<m+1}\left(m+1\right)e^{-m}\nonumber \\
 & \ll_{X}n\sum_{m=0}^{T}(m+1)e^{m+1}e^{-m}\ll nT^{2}.\label{eq:error-in-tf-diff-est-2}
\end{align}
We also have
\begin{align}
\int_{-\infty}^{\infty}r\widehat{\varphi_{T}}(r)\tanh(\pi r)dr & =T\int_{-\infty}^{\infty}r\widehat{\varphi_{0}}(Tr)\tanh(\pi r)dr\nonumber \\
 & =\frac{1}{T}\int_{-\infty}^{\infty}r'\widehat{\varphi_{0}}(r')\tanh(\pi\frac{r'}{T})dr'\nonumber \\
 & \leq\frac{2}{T}\int_{0}^{\infty}|r'||\widehat{\varphi_{0}}(r')|dr'\ll\frac{1}{T}.\label{eq:integrak-term}
\end{align}
The final estimate uses that, since $\varphi_{0}$ is compactly supported,
$\widehat{\varphi_{0}}$ is a Schwartz function and decays faster
than any inverse of a polynomial. Combining (\ref{eq:trace-formula-difference}),
(\ref{eq:error-in-tf-diff}), (\ref{eq:error-in-tf-diff-est-2}) and
(\ref{eq:integrak-term}) gives
\begin{align}
\sum_{j=1}^{\infty}\widehat{\varphi_{T}}(r_{i_{j}}(X_{\phi})) & =O\left(\frac{(n-1)\cdot\mathrm{area}(X)}{4\pi}\cdot\frac{1}{T}\right)+\sum_{\gamma\in\mathcal{P}(X)}\frac{\left(\fix_{\tilde{\gamma}}(\phi)-1\right)\ell(\gamma)}{2\sinh\left(\frac{\ell(\gamma)}{2}\right)}\varphi_{T}\left(\ell(\gamma)\right)+O_{X}\left(T^{2}n\right)\nonumber \\
 & =\sum_{\gamma\in\mathcal{P}(X)}\frac{\left(\fix_{\tilde{\gamma}}(\phi)-1\right)\ell(\gamma)}{2\sinh\left(\frac{\ell(\gamma)}{2}\right)}\varphi_{T}\left(\ell(\gamma)\right)+O_{X}\left(T^{2}n\right),\label{eq:pre-taking-expectations}
\end{align}
where in the last equality we used $T>1$. 

We are now in a position to use Theorem \ref{thm:effective-error}.
The contributions to the sum above come from $\gamma$ with $\ell(\gamma)\leq T$.
By Lemma \ref{lem:length-comparison}, this entails $\ell_{w}(\tilde{\gamma})=\ell_{w}\left(\gamma\right)\leq K_{1}T+K_{2}\leq c\log n$
for some $c=c(X)>0$ and $n$ sufficiently large. Moreover, if $\gamma\in\mathcal{P}(X)$,
then $\tilde{\gamma}$ is not a proper power in $\Gamma_{g}$. Thus
for each $\gamma$ appearing in (\ref{eq:pre-taking-expectations}),
Theorem \ref{thm:effective-error} applies to give
\[
\E_{g,n}\left[\fix_{\tilde{\gamma}}(\phi)-1\right]\ll_{X}\frac{\left(\log n\right)^{A}}{n}
\]
where $A=A(g)>0$ and the implied constant depends on $X$. Now using
that $\text{\ensuremath{\widehat{\varphi_{T}}}}$ is non-negative
on $\ensuremath{\R\cup i\R}$, we take expectations of (\ref{eq:pre-taking-expectations})
with respect to the uniform measure on $\X_{g,n}$ to obtain
\begin{eqnarray}
\E_{g,n}\left[\widehat{\varphi_{T}}\left(r_{i_{1}}(X_{\phi})\right)\right] & \leq & \sum_{\gamma\in\mathcal{P}(X)}\frac{\E_{g,n}\left[\fix_{\tilde{\gamma}}(\phi)-1\right]\ell(\gamma)}{2\sinh\left(\frac{\ell(\gamma)}{2}\right)}\varphi_{T}(\ell(\gamma))+O_{X}\left(T^{2}n\right)\nonumber \\
 & \stackrel{\text{Theorem \ref{thm:effective-error}}}{\ll_{X}} & \frac{\left(\log n\right)^{A}}{n}\sum_{\gamma\in\mathcal{P}(X)}\frac{\ell(\gamma)}{2\sinh\left(\frac{\ell(\gamma)}{2}\right)}\varphi_{T}(\ell(\gamma))+T^{2}n\nonumber \\
 & \ll_{X} & \frac{\left(\log n\right)^{A}}{n}\sum_{\gamma\in\mathcal{P}(X)\,:\,\ell(\gamma)\leq T}\ell(\gamma)e^{-\ell(\gamma)/2}+T^{2}n\nonumber \\
 & \le & \frac{\left(\log n\right)^{A}}{n}\sum_{m=0}^{\left\lceil T-1\right\rceil }\sum_{\gamma\in\L(X)\,:\,m\leq\ell(\gamma)<m+1}(m+1)e^{-m/2}+T^{2}n\nonumber \\
 & \stackrel{\text{Lemma \ref{lem:coarse-counting-bound}}}{\ll_{X}} & \frac{\left(\log n\right)^{A}}{n}\sum_{m=0}^{\left\lceil T-1\right\rceil }(m+1)e^{-m/2}e^{m+1}+T^{2}n\nonumber \\
 & \ll & \frac{\left(\log n\right)^{A}}{n}Te^{T/2}+T^{2}n\nonumber \\
 & \stackrel{T=4\log n}{\ll_{\varepsilon}} & n^{1+\varepsilon/3},\label{eq:expectation-bound}
\end{eqnarray}
where $\varepsilon$ is the parameter given in Theorem \ref{thm:three-sixteenths}.
The third inequality above used that on a compact hyperbolic surface,
the lengths of closed geodesics are bounded below away from zero (by
the Collar Lemma \cite[Thm. 4.1.1]{Buser}), together with the fact
that $\varphi_{T}$ is supported in $[-T,T]$. So $\E_{g,n}\left[\widehat{\varphi_{T}}\left(r_{i_{1}}(X_{\phi})\right)\right]\le n^{1+\varepsilon/2}$
for large enough $n$, and for these values of $n$, by Markov's inequality
\begin{equation}
\P\left[\widehat{\varphi_{T}}(r_{i_{1}}(X_{\phi}))>n^{1+\varepsilon}\right]\leq n^{-\varepsilon/2}.\label{eq:markov}
\end{equation}
Lemma \ref{lem:lower-bound-on-fourier} implies that if $\lambda_{i_{1}}(X_{\phi})\leq\frac{3}{16}-\varepsilon$,
in which case $r_{i_{1}}(X_{\phi})=it_{\phi}$ with $t_{\phi}\in\R$
and $t_{\phi}\geq\sqrt{\frac{1}{16}+\varepsilon}\geq\frac{1}{4}+\varepsilon$
for $\varepsilon$ sufficiently small, then 
\begin{equation}
\widehat{\varphi_{T}}(r_{i_{1}}(X_{\phi}))\geq C_{\varepsilon}Te^{T(1-\varepsilon)t_{\phi}}\geq C_{\varepsilon}n^{4(1-\varepsilon)(1/4+\varepsilon)}\geq C_{\varepsilon}n^{1+2\varepsilon}>n^{1+\varepsilon},\label{eq:amplifier}
\end{equation}
by decreasing $\varepsilon$ if necessary, and then assuming $n$
is sufficiently large. Combining (\ref{eq:markov}) and (\ref{eq:amplifier})
gives 
\[
\P\left[\text{\ensuremath{X_{\phi}} has a new eigenvalue \ensuremath{\leq\frac{3}{16}-\varepsilon}}\right]\leq\P\left[\widehat{\varphi_{T}}(r_{i_{1}}(X_{\phi}))>n^{1+\varepsilon}\right]\leq n^{-\varepsilon/2}
\]
completing the proof of Theorem \ref{thm:three-sixteenths}, under
the assumption of Theorem \ref{thm:effective-error}. $\square$

\subsection{Proof of Theorem \ref{thm:density}}

We continue using the same notation as in the previous section, including
the choice of $T=4\log n$. We let 
\[
0\leq\lambda_{i_{1}}(X_{\phi})\leq\lambda_{i_{2}}(X_{\phi})\leq\cdots\leq\lambda_{i_{k(\phi)}}\leq\frac{1}{4}
\]
denote the collection of new eigenvalues of $X_{\phi}$ of size at
most $\frac{1}{4}$, with multiplicities included. For each such eigenvalue
we write $\lambda_{i_{j}}=s_{i_{j}}(1-s_{i_{j}})$ with $s_{i_{j}}\in\left[\frac{1}{2},1\right]$;
this has the result that $r_{i_{j}}=i(s_{i_{j}}-\frac{1}{2})$.

Again taking expectations of (\ref{eq:pre-taking-expectations}) with
respect to the uniform measure on $\X_{g,n}$, but this time, keeping
more terms, gives

\begin{align*}
\E_{g,n}\left[\sum_{j=1}^{k(\phi)}\widehat{\varphi_{T}}\left(r_{i_{j}}(X_{\phi})\right)\right] & \leq\sum_{\gamma\in\mathcal{P}(X)}\frac{\E_{g,n}\left[\fix_{\tilde{\gamma}}(\phi)-1\right]\ell(\gamma)}{2\sinh\left(\frac{\ell(\gamma)}{2}\right)}\varphi_{T}(\ell(\gamma))+O_{X}\left(T^{2}n\right)\\
 & \ll_{X,\varepsilon}n^{1+\varepsilon/3}
\end{align*}
by (\ref{eq:expectation-bound}). On the other hand, Lemma \ref{lem:lower-bound-on-fourier}
implies that for $\varepsilon\in(0,1)$
\[
\sum_{j=1}^{k(\phi)}\widehat{\varphi_{T}}\left(r_{i_{j}}\left(X_{\phi}\right)\right)\gg_{\varepsilon}\sum_{j=1}^{k(\phi)}Te^{T(1-\varepsilon)\left(s_{i_{j}}\left(X_{\phi}\right)-\frac{1}{2}\right)}\gg\sum_{j=1}^{k(\phi)}n^{4(1-\varepsilon)\left(s_{i_{j}}\left(X_{\phi}\right)-\frac{1}{2}\right)}.
\]
Therefore
\[
\E_{g,n}\left[\sum_{j=1}^{k(\phi)}n^{4(1-\varepsilon)\left(s_{i_{j}}\left(X_{\phi}\right)-\frac{1}{2}\right)}\right]\leq n^{1+\varepsilon/2}
\]
for $n$ sufficiently large. Markov's inequality therefore gives
\[
\P\left[\sum_{j=1}^{k(\phi)}n^{4(1-\varepsilon)\left(s_{i_{j}}\left(X_{\phi}\right)-\frac{1}{2}\right)}\geq n^{1+\varepsilon}\right]\leq n^{-\varepsilon/2}
\]
so a.a.s.~$\sum_{j=1}^{k(\phi)}n^{4(1-\varepsilon)\left(s_{i_{j}}\left(X_{\phi}\right)-\frac{1}{2}\right)}<n^{1+\varepsilon}$.
This gives that for any $\sigma\in\left(\frac{1}{2},1\right)$, a.a.s.
\[
\#\left\{ 1\leq j\leq k(\phi)\,:\,s_{i_{j}}>\sigma\right\} \leq n^{1+\varepsilon-4(1-\varepsilon)(\sigma-\frac{1}{2})}\leq n^{3-4\sigma+3\varepsilon}.
\]
This finishes the proof of Theorem \ref{thm:density} assuming Theorem
\ref{thm:effective-error}. $\square$

\section{Tiled surfaces\label{sec:tiled-surfaces}}

\subsection{Tiled surfaces}

Here we assume $g=2$, and let $\Gamma\eqdf\Gamma_{2}$. We write
$\X_{n}\eqdf\X_{2,n}$ throughout the rest of the paper. Consider
the construction of the surface $\Sigma_{2}$ from an octagon by identifying
its edges in pairs according to the pattern $aba^{-1}b^{-1}cdc^{-1}d^{-1}$.
This gives rise to a CW-structure on $\Sigma_{2}$ consisting of one
vertex (denoted $o$), four oriented $1-$cells (labeled by $a,b,c,d$)
and one $2$-cell which is the octagon glued along eight $1$-cells\footnote{We use the terms vertices, edges and octagons interchangeably with
$0$-cells, $1$-cells and $2$-cells, respectively.}. See Figure \ref{fig:Sigma_2}. We identify $\Gamma_{2}$ with $\pi_{1}\left(\Sigma_{2},o\right)$,
so that in the presentation (\ref{eq:fundamental-group-isomorphism}),
words in the generators $a,b,c,d$ correspond to the homotopy classes
of the corresponding closed paths based at $o$ along the $1$-skeleton
of $\Sigma_{2}$. 

\begin{figure}
\begin{centering}
\includegraphics[viewport=0bp 0bp 360.324bp 215.9145bp,scale=0.7]{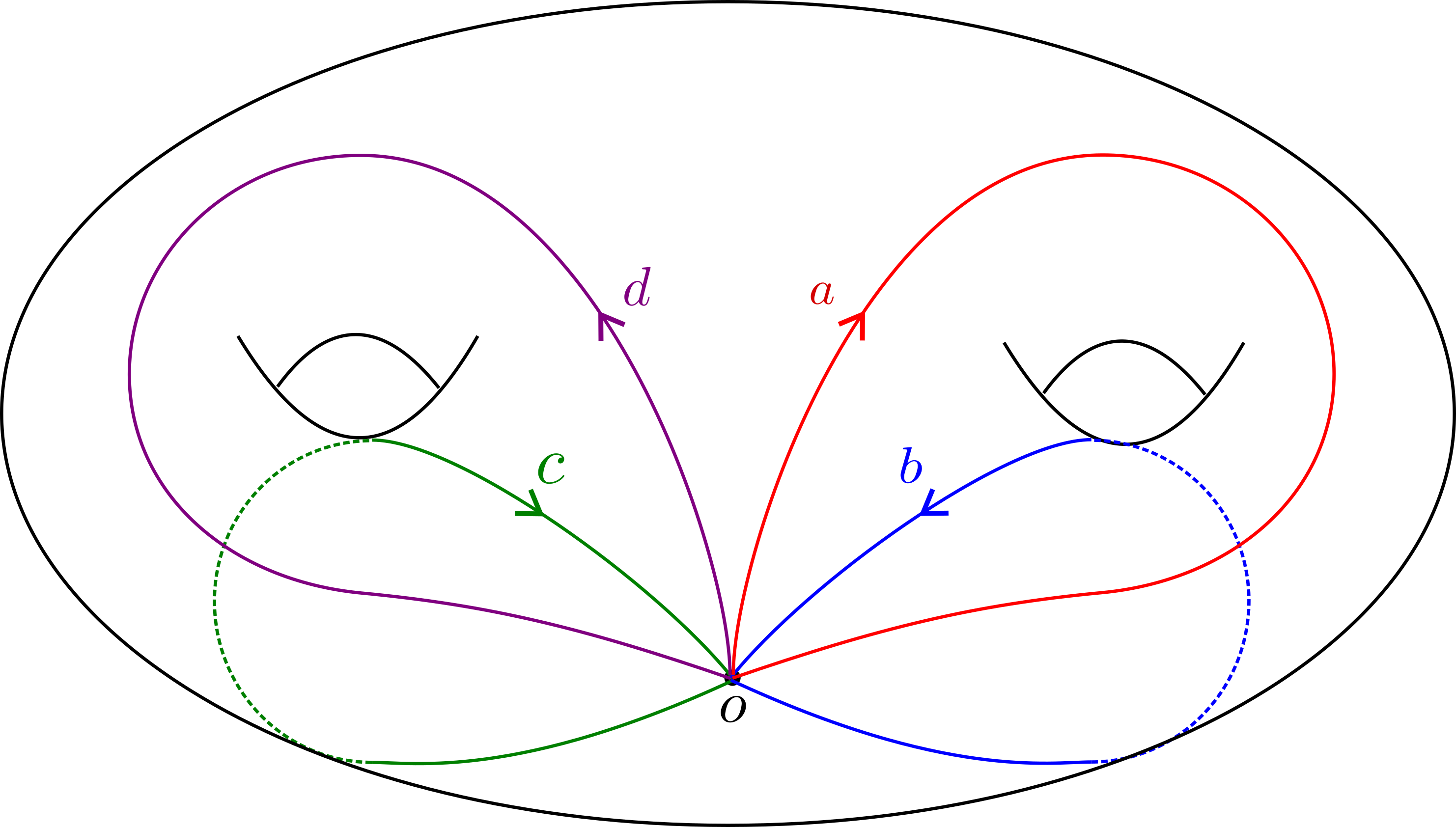}
\par\end{centering}
\caption{\label{fig:Sigma_2}The CW-structure we give to the surface $\Sigma_{2}$
with fundamental group $\Gamma=\Gamma_{2}=\left\langle a,b,c,d\,\middle|\,\left[a,b\right]\left[c,d\right]\right\rangle $:
it consists of a single vertex ($0$-cell), four edges ($1$-cells)
and one octagon (a $2$-cell).}
\end{figure}

Note that every covering space $p\colon\Upsilon\to\Sigma_{2}$ inherits
a CW-structure from $\Sigma_{2}$: the vertices are the pre-images
of $o$, and the open $1$-cells (2-cells) are the connected components
of the pre-images of the open 1-cells (2-cells, respectively) in $\Sigma_{2}$.
In particular, this is true for the universal covering space $\ts$
of $\Sigma_{2}$, which we can now think of as a CW-complex. A sub-complex
of a CW-complex is a subspace consisting of cells such that if some
cell belongs to the subcomplex, then so are the cells of smaller dimension
at its boundary. 
\begin{defn}[Tiled surface]
\label{def:tiled-surface}\cite[Def. 3.1]{MPcore} A \emph{tiled
surface} $Y$ is a sub-complex of a (not-necessarily-connected) covering
space of $\Sigma_{2}$. In particular, a tiled surface is equipped
with the restricted covering map $p\colon Y\to\Sigma_{2}$ which is
an immersion. We denote by $Y^{\left(0\right)}$ the set of vertices
and by $Y^{\left(1\right)}$ the $1$-skeleton of $Y$. If $Y$ is
compact, we write $\v\left(Y\right)$ for the number of vertices of
$Y$, $\e\left(Y\right)$ for the number of edges and $\f\left(Y\right)$\marginpar{$\text{\ensuremath{{\scriptstyle \protect\v\left(Y\right),\protect\e\left(Y\right),\protect\f\left(Y\right)}}}$}
for the number of octagons. 
\end{defn}

Alternatively, instead of considering a tiled surface $Y$ to be a
complex equipped with a restricted covering map, one may consider
$Y$ to be a complex as above with directed and labeled edges: the
directions and labels ($a,b,c,d$) are pulled back from $\Sigma_{2}$
via $p$. These labels uniquely determine $p$ as a combinatorial
map between complexes.

Note that a tiled surface is not always a surface: for example, it
may also contain vertices or edges with no $2$-cells incident to
them. However, as $Y$ is a sub-complex of a covering space of $\Sigma_{2}$,
namely, of a surface, any neighborhood of $Y$ inside the cover is
a surface, and it is sometimes beneficial to think of $Y$ as such. 
\begin{defn}[Thick version of a tiled surface]
\label{def:thick-version}\cite[Def. 3.2]{MPcore} Given a tiled
surface $Y$ which is a subcomplex of the covering space $\Upsilon$
of $\Sigma_{2}$, consider a small, closed, regular neighborhood of
$Y$ in $\Upsilon$. This neighborhood is a closed surface, possibly
with boundary, which is referred to as the \emph{thick version of
$Y$}.

We let $\partial Y$ denote the boundary of the thick version of $Y$
and $\d\left(Y\right)$\marginpar{$\partial Y,\protect\d\left(Y\right)$}
denote the number of edges along $\partial Y$ (so if an edge of $Y$
does not border any octagon, it is counted twice). 
\end{defn}

We stress that we do not think of $Y$ as a sub-complex, but rather
as a complex for its own sake, which happens to have the capacity
to be realized as a subcomplex of a covering space of $\Sigma_{2}$.
In particular, if $Y$ is compact, it is a combinatorial object given
by a finite amount of data. See \cite[\S 3]{MPcore} for a more detailed
discussion. 
\begin{defn}[Morphisms of tiled surfaces]
 Let $p_{i}\colon Y_{i}\to\Sigma_{2}$ be tiled surfaces for $i=1,2$.
A map $f\colon Y_{1}\to Y_{2}$ is a \emph{morphism }of tiled surfaces
if it is a combinatorial map of CW-complexes that commutes with the
restricted covering maps.
\[
\xymatrix{Y_{1}\ar[rr]^{f}\ar[rd]_{p_{1}} &  & Y_{2}\ar[ld]^{p_{2}}\\
 & \Sigma_{2}
}
\]
\end{defn}

In other words, a morphism of tiled surfaces is a combinatorial map
of CW-complexes sending $i$-cells to $i$-cells and which respects
the directions and labels of edges. 
\begin{example}
\label{exa:x_phi}The fibered product construction gives a one-to-one
correspondence between\linebreak{}
$\Hom(\Gamma,S_{n})$ and topological degree-$n$ covers of $\Sigma_{2}$
with a labeled fiber over the basepoint $o$. Explicitly, for $\phi\in\Hom(\Gamma,S_{n})$,
we can consider the quotient
\[
X_{\phi}\eqdf\Gamma\backslash\left(\ts\times[n]\right)
\]
where $\ts$ is the universal cover of $\Sigma_{2}$ (an open disc)
and $\Gamma$ acts on $\ts\times[n]$ diagonally, by the usual action
of $\Gamma$ on $\ts$ on the first factor, and via $\phi$ on the
second factor. The covering map $X_{\phi}\to\Sigma_{2}$ is induced
by the projection $\ts\times[n]\to\ts$. 

Being a covering space of $\Sigma_{2}$, each $X_{\phi}$ is automatically
also a tiled surface. The fiber of $o\in\Sigma_{2}$ is the collection
of vertices of $X_{\phi}$. We fix throughout the rest of the paper
a vertex $u\in\ts$ lying over $o\in\Sigma_{2}$. This identifies
the fiber over $o$ in $X_{\phi}$ with $\{u\}\times[n]$ and hence
gives a fixed bijection between the vertices of $X_{\phi}$ and the
numbers in $[n]$. The map $\phi\mapsto X_{\phi}$ is the desired
one-to-one correspondence between $\Hom(\Gamma,S_{n})$ and topological
degree-$n$ covers of $\Sigma_{2}$ with the fiber over $o$ labeled
bijectively by $\left[n\right]$. 
\end{example}

\begin{example}
\label{exa:the-loop-of-a-word}For any $1\ne\gamma\in\Gamma$, pick
a word $\tilde{\gamma}$ of minimal length in the letters $a,b,c,d$
and their inverses that represents an element in the conjugacy class
of $\gamma$ in $\Gamma$. In particular, $\tilde{\gamma}$ is cyclically
reduced. Now take a circle and divide it into $\{a,b,c,d\}$-labeled
and directed edges separated by vertices, such that following around
the circle from some vertex and in some orientation, and reading off
the labels and directions, spells out $\tilde{\gamma}$. Call the
resulting complex $\CC_{\gamma}$. That $\CC_{\gamma}$ is a tiled
surface follows from \cite{MPcore} (in particular, it is embedded
in the core surface $\core\left(\left\langle \gamma\right\rangle \right)$
which is itself a tiled surface, by \cite[Thm. 5.10]{MPcore}). Note
that generally $\CC_{\gamma}$ is not uniquely determined by $\gamma$
(e.g., \cite[Fig.~1.2 and \S 4, \S 5]{MPcore}), and we choose one
of the options arbitrarily. We have $\v(\CC_{\gamma})=\e\left(\CC_{\gamma}\right)=\ell_{w}(\gamma)$
and $\f\left(\CC_{\gamma}\right)=0$.
\end{example}

If $Y$ is a compact tiled surface, there are some simple relations
between the quantities $\v(Y)$, $\e(Y)$, $\f(Y)$, $\d(Y)$, and
$\chi(Y),$ the topological Euler characteristic of $Y$. We note
the following relations, which are straightforward or standard. For
example, $\e\left(Y\right)\le4\v\left(Y\right)$ as each vertex is
incident to at most 8 half-edges\footnote{For general $g\ge2$, (\ref{eq:d-e-f}) is $\d\left(Y\right)=2\e\left(Y\right)-4g\f\left(Y\right)$,
and (\ref{eq:v-e-f-inequality}) is $2g\f\left(Y\right)\le\e\left(Y\right)\le2g\f\left(Y\right)$.}. 
\begin{eqnarray}
\d\left(Y\right) & = & 2\e\left(Y\right)-8\f\left(Y\right).\label{eq:d-e-f}\\
4\f\left(Y\right) & \le & \e\left(Y\right)~\le~4\v\left(Y\right).\label{eq:v-e-f-inequality}
\end{eqnarray}
The following lemma will be useful later.
\begin{lem}
\label{lem:D-vs-d}Let $Y$ be a compact tiled surface without isolated
vertices. Then 
\[
\v\left(Y\right)\leq\f\left(Y\right)+\d(Y).
\]
\end{lem}

\begin{proof}
Let $\mathfrak{i}$ denote the number of \emph{internal }vertices
of $Y$, namely, vertices adjacent to $8$ octagons, and let $\mathfrak{p}$
denote the number of the remaining, \emph{peripheral} vertices. As
there are no isolated vertices, $\mathfrak{p}\le\d\left(Y\right)$
(when going through the boundary cycles, one edge at a time, one passes
at every step exactly one peripheral vertex, and each peripheral vertex
is traversed at least once, although possibly more than once). We
have 
\begin{align*}
8\f(Y) & =\sum_{O\text{ an octagon of \ensuremath{Y}}}\#\text{\{corners of \ensuremath{O}\}}\\
 & =\sum_{\text{\ensuremath{v} a vertex of \ensuremath{Y}}}\#\{\text{corners of octagons at \ensuremath{v}\}}\\
 & \geq8\mathfrak{i}=8\v(Y)-8\mathfrak{p}\ge8\v\left(Y\right)-8\d\left(Y\right).
\end{align*}
\end{proof}
The Euler characteristic $\chi(Y)$ is also controlled by $\f(Y)$
and $\d(Y)$.
\begin{lem}
\label{lem:Euler-char-bound}Let $Y$ be a compact tiled surface without
isolated vertices. Then\footnote{For arbitrary $g\ge2$, the bound is $\chi\left(Y\right)\le\frac{\d\left(Y\right)}{2}-\left(2g-2\right)\f\left(Y\right)$.}
\[
\chi(Y)\leq\frac{\d(Y)}{2}-2\f(Y).
\]
\end{lem}

\begin{proof}
We have
\begin{eqnarray*}
\chi\left(Y\right) & = & \v\left(Y\right)-\e\left(Y\right)+\f\left(Y\right)\stackrel{\eqref{eq:d-e-f}}{=}\v\left(Y\right)-3\f\left(Y\right)-\frac{\d\left(Y\right)}{2}\stackrel{\text{Lemma \ref{lem:D-vs-d}}}{\leq}\frac{\d\left(Y\right)}{2}-2\f\left(Y\right).
\end{eqnarray*}
\end{proof}

\subsection{Blocks and chains\label{subsec:Blocks-and-chains}}

Here we introduce language that was used in \cite{MPcore,MPasympcover},
based on terminology of Birman and Series from \cite{BirmanSeries}.
Let $Y$ denote a tiled surface throughout this $\S\S$\ref{subsec:Blocks-and-chains}.
When we refer to directed edges of $Y$, they are not necessarily
directed according to the definition of $Y$.

First of all, we augment $Y$ by adding \textbf{half-edges}, which
should be thought of as copies of $[0,\frac{1}{2})$. Of course, every
edge of $Y^{(1)}$ is thought of as containing two half edges, each
of which inherits a label in $\{a,b,c,d\}$ and a direction from their
ambient edge. We add to $Y$ $\{a,b,c,d\}$-labeled and directed half-edges
to form $Y_{+}$\marginpar{$Y_{+}$} so that every vertex of $Y_{+}$
has exactly 8 emanating half-edges, with labels and directions given
by `$a$-outgoing, $b$-incoming, $a$-incoming, $b$-outgoing, $c$-outgoing,
$d$-incoming, $c$-incoming, $d$-outgoing'. The cyclic order we
have written here induces a fixed cyclic ordering on the half-edges
at each vertex of $Y_{+}$. If a half-edge of $Y_{+}$ does not belong
to an edge of $Y$ (hence was added to $Y_{+}$), we call it a \textbf{hanging
half-edge}\emph{.} We may think of $Y_{+}$ as a surface too, by considering
the thick version of $Y$ and attaching a thin rectangle for every
hanging half-edge. We call the resulting surface the \textbf{thick
version of}\textbf{\emph{ $\bm{Y_{+}}$}}, and mark its boundary by
$\bm{\partial Y_{+}}$\marginpar{$\partial Y_{+}$}. See Figure \ref{fig:blocks orientation and cyclic order}
for the cyclic ordering of half-edges around every vertex and Figure
\ref{fig:piece} for a piece of $\partial Y_{+}$.

\begin{figure}
\begin{centering}
\includegraphics{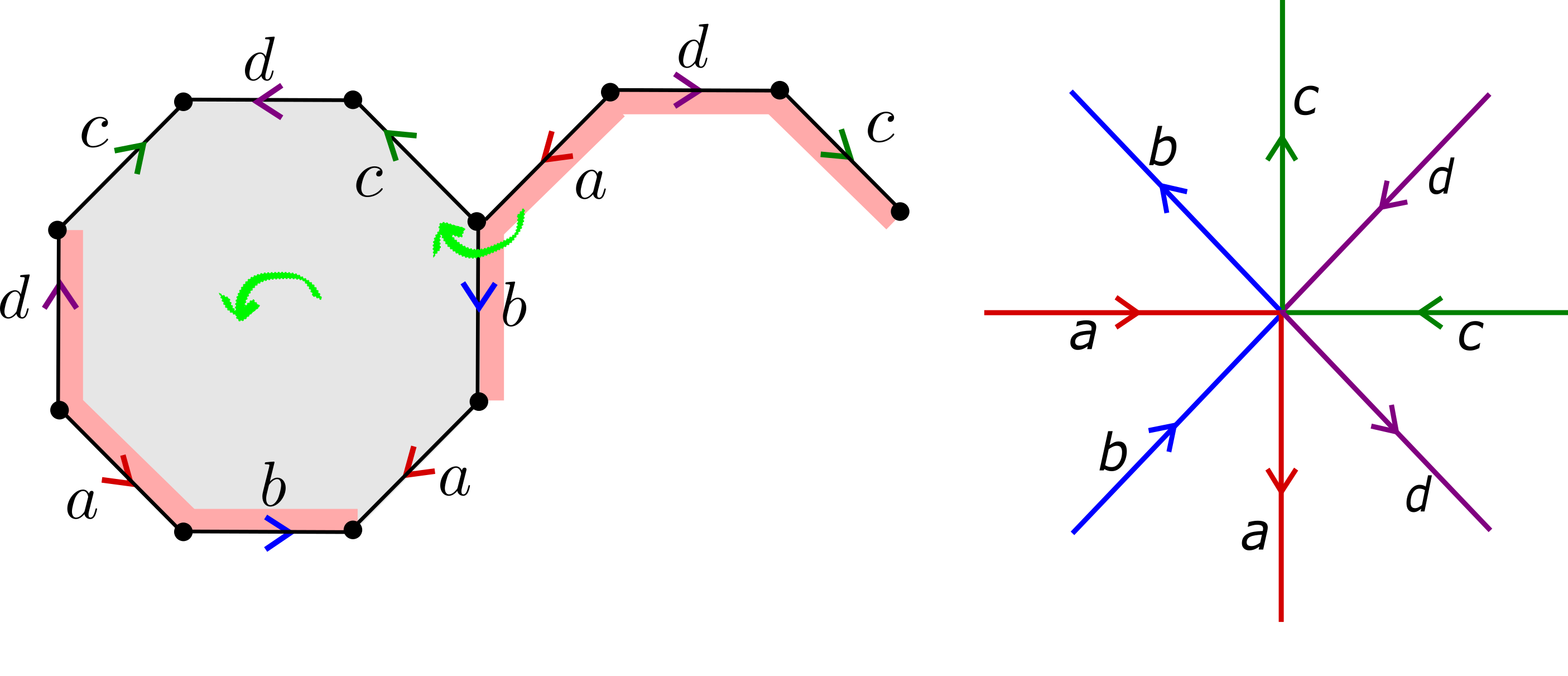}
\par\end{centering}
\caption{\label{fig:blocks orientation and cyclic order}The right figure shows
a vertex with 8 half-edges around it, ordered (clockwise) according
to the fixed cyclic order induced from the CW-structure on $\Sigma_{2}$.
On the left is a tiled surface with 11 vertices, 11 edges and one
octagon. The orientation on the octagon is counter-clockwise, while
around any vertex it is clockwise. The pink stripes describe blocks:
a half-block spelling $c^{-1}d^{-1}ab$ and a block of length $3$
spelling $d^{-1}ab$. The latter one can be extended in both ends. }
\end{figure}

For two directed edges $\vec{e_{1}}$ and $\vec{e}_{2}$ of $Y$ with
the terminal vertex $v$ of $\text{\ensuremath{\vec{e}_{1}} equal to the source of \ensuremath{\vec{e}_{2}}}$,
the \emph{half-edges}\textbf{\emph{ }}\emph{between $\vec{e}_{1}$
and $\vec{e}_{2}$ }are by definition the half edges of $Y_{+}$ at
$v$ that are strictly between $\text{\ensuremath{\vec{e}_{1}} and \ensuremath{\vec{e}_{2}}}$
in the given cyclic ordering. There are $m$ of these where $0\leq m\leq7$.

A \textbf{path}\emph{ }in $Y$ is a sequence ${\cal P}\text{\ensuremath{=}}(\vec{e_{1}},\ldots,\vec{e}_{k})$
of directed edges in $Y^{(1)}$, such that for each $1\leq i\leq k-1$
the terminal vertex of $\vec{e}_{i}$ is the initial vertex of $\vec{e}_{i+1}$.
A \textbf{cycle}\emph{ }in $Y$ is a cyclic sequence $\text{\ensuremath{\mathcal{C}=}}(\vec{e_{1}},\ldots,\vec{e}_{k})$
which is a path with the terminal vertex of $\vec{e}_{k}$ identical
to the initial vertex of $\vec{e}_{1}$. A \textbf{boundary cycle}\emph{
}of $Y$ is a cycle corresponding to a boundary component of the thick
version of $Y$. A boundary cycle is always oriented so that if $Y$
is embedded in the full cover $Z$, the boundary reads successive
segments of the boundaries of the neighboring octagons (in $Z-Y$)
with the orientation of each octagon coming from $\left[a,b\right]\left[c,d\right]$
(and not from the inverse word). For example, the unique boundary
cycle of the tiled surface in the left side of Figure \ref{fig:blocks orientation and cyclic order},
starting at the rightmost vertex, spells the cyclic word $c^{-1}d^{-1}abab^{-1}a^{-1}dcd^{-1}c^{-1}a^{-1}dc$.

If ${\cal P}$ is a path in $Y^{\left(1\right)}$, a \textbf{block}\emph{
}in ${\cal P}$ is a non-empty (possibly cyclic) subsequence of successive
edges, each successive pair of edges having no half-edges between
them (this means that a block reads necessarily a subword of the cyclic
word $\left[a,b\right]\left[c,d\right]$). A \textbf{half-block} is
a block of length $4$ (in general, $2g$) and a \textbf{long block}
is a block of length at least $5$ (in general, $2g+1$). See Figure
\ref{fig:blocks orientation and cyclic order}.

Two blocks $(\vec{e}_{i},\ldots,\vec{e}_{j})$ and $(\vec{e}_{k},\ldots,\vec{e}_{\ell})$
in a path ${\cal P}$ are called \emph{consecutive }if $(\vec{e}_{i},\ldots,\vec{e}_{j},\vec{e}_{k},\ldots,\vec{e}_{\ell})$
is a (possibly cyclic) subsequence of ${\cal P}$ and there is \uline{precisely}
one half-edge between $\vec{e}_{j}$ and $\vec{e}_{k}$. A \textbf{chain}
is a (possibly cyclic) sequence of consecutive blocks. Note that in
a chain, an $f$-edge with some $f\in\left\{ a^{+1},\ldots,d^{\pm1}\right\} $
is followed by an edge labeled by the letter $f'$ that follows $f$
in the cyclic word $\left[a,b\right]\left[c,d\right]$, or by the
letter that follows the inverse of $f'$. For example, a $b^{-1}$-edge
is always followed in a chain by either a $c$-edge or a $d^{-1}$-edge.
A \emph{cyclic chain} is a chain whose blocks pave an entire cycle
(with exactly one half-edge between the last block and the first blocks).
A \textbf{long chain} is a chain consisting of consecutive blocks
of lengths
\[
4,3,3,\ldots,3,4
\]
(in general, $2g,2g-1,2g-1,\ldots,2g-1,2g$). See Figure \ref{fig:long chain}.
A \textbf{half-chain}\emph{ }is a cyclic chain consisting of consecutive
blocks of length $3$ (in general, $2g-1$) each. 

\begin{figure}
\begin{centering}
\includegraphics[scale=0.7]{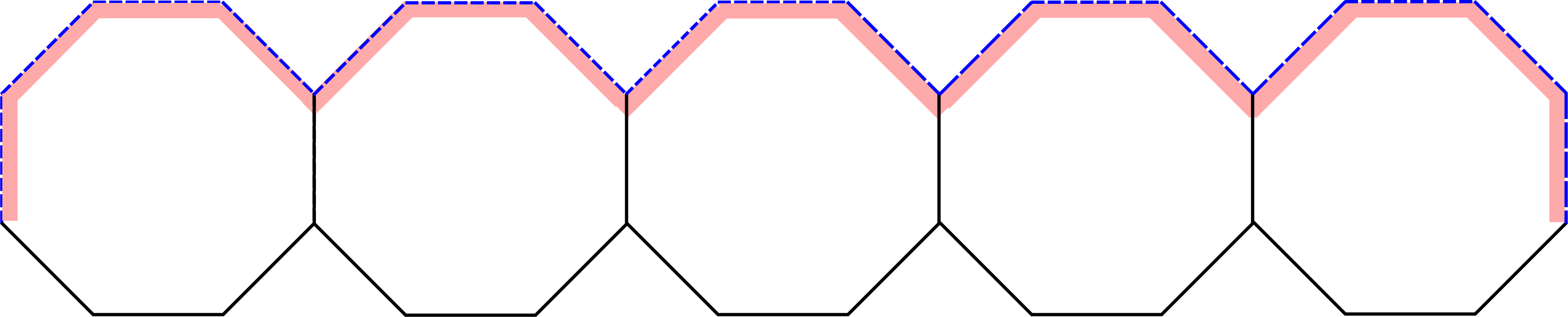}
\par\end{centering}
\caption{\label{fig:long chain}A long chain (the pink stripe) consisting of
five consecutive blocks of lengths $4,3,3,3,4$}
\end{figure}

\subsection{Boundary reduced and strongly boundary reduced tiled surfaces}

We recall the following definitions from \cite[Def. 4.1, 4.2]{MPcore}.
\begin{defn}[Boundary reduced]
A tiled surface $Y$ is \emph{boundary reduced} if no boundary cycle
of $Y$ contains a long block or a long chain.
\end{defn}

\begin{defn}[Strongly boundary reduced]
A tiled surface $Y$ is \emph{strongly boundary reduced} if no boundary
cycle of $Y$ contains a half-block or is a half-chain.
\end{defn}

Given a tiled surface $Y$ embedded in a boundary reduced tiled surface
$Z$, the $\br$-closure of $Y$ in $Z$, denoted $\br\left(Y\hookrightarrow Z\right)$
and introduced in \cite[Def.~4.4]{MPcore}, is defined as the intersection
of all boundary reduced sub-tiled surfaces of $Z$ containing $Y$.\textbf{
}We compile some properties of the $\br$-closure into the following
proposition. 
\begin{prop}
\label{prop:properties-of-BR-alg}Let $Y\hookrightarrow Z$ be an
embedding of a compact tiled surface $Y$ into a boundary reduced
tiled surface $Z$, and denote $Y'\eqdf\br\left(Y\hookrightarrow Z\right)$.
\begin{enumerate}
\item \cite[Prop.~4.5]{MPcore} $Y'$ is boundary reduced.
\item \cite[proof of Prop.~4.6]{MPcore} $Y'$ is compact, and $\d\left(Y'\right)\le\d\left(Y\right)$,
with equality if and only if $Y'=Y$.
\item \label{enu:BR-algo}\cite[proof of Prop.~4.6]{MPcore} $Y'$ can be
obtained from $Y$ by initializting $Y'=Y$ and then repreatedly either
$\left(i\right)$ annexing an octagon of $Z\setminus Y'$ which borders
a long block along $\partial Y'$, or $\left(ii\right)$ annexing
the octagons of $Z\setminus Y'$ bordering some long chain along $\partial Y'$,
until $Y'$ is boundary reduced.
\item \label{enu:octagon-growth} We have\footnote{For larger values of the genus $g$, we could get a tighter bound,
but the stated bound holds and is good enough.} 
\[
\f(Y')\leq\f(Y)+\frac{\d(Y)^{2}}{6}.
\]
\end{enumerate}
\end{prop}

\begin{proof}[Proof of item \ref{enu:octagon-growth}]
 Assume that $Y'$ is obtained from $Y$ by the procedure described
in item \ref{enu:BR-algo}. In each such step, $\d\left(Y'\right)$
decreases by at least two, so there are at most $\frac{\d\left(Y\right)}{2}$
steps where octagons are added. We will be done by showing that at
each step at most $\frac{\d\left(Y\right)}{3}$ octagons are added.
And indeed, in option $\left(i\right)$ exactly one octagon is added
(and $1\le\frac{\d\left(Y\right)}{3}$ or otherwise $Y$ is boundary
reduced). In option $\left(ii\right)$, if the long chain consists
of $\ell$ blocks, it is of length $3\ell+2\le\d\left(Y'\right)$,
and at most $\ell\le\frac{\d\left(Y'\right)-2}{3}<\frac{\d\left(Y\right)}{3}$
new octagons are added. 
\end{proof}

\subsection{Pieces and $\varepsilon$-adapted tiled surfaces\label{subsec:epsilon-adapted-tiled-surfaces}}

For the proof of Theorem \ref{thm:effective-error} we will need to
quantify (strongly) boundary reduced tiled surfaces. This is captured
by the notion of $\varepsilon$-adapted tiled surface we introduce
in this $\S\S$\ref{subsec:epsilon-adapted-tiled-surfaces}. The following
concepts of a piece and its defect play a crucial role here.
\begin{defn}[Piece, defect]
\label{def:piece} A \emph{piece }$P$ of $\partial Y_{+}$ is a
(possibly cyclic) path along $\partial Y_{+}$, consisting of whole
directed edges and/or whole hanging half-edges. We write $\e(P)$
for the number of full directed edges in $P$, $\he(P)$ for the number
of hanging half-edges in $P$, and $|P|\eqdf\e(P)+\he(P)$. We let
\[
\defect(P)\eqdf\e(P)-3\mathrm{\he}(P).
\]
(In general, $\defect(P)\eqdf\e(P)-\left(2g-1\right)\mathrm{\he}(P).$)
See Figure \ref{fig:piece} for an illustration of a piece.
\end{defn}

\begin{figure}
\begin{centering}
\includegraphics[scale=2]{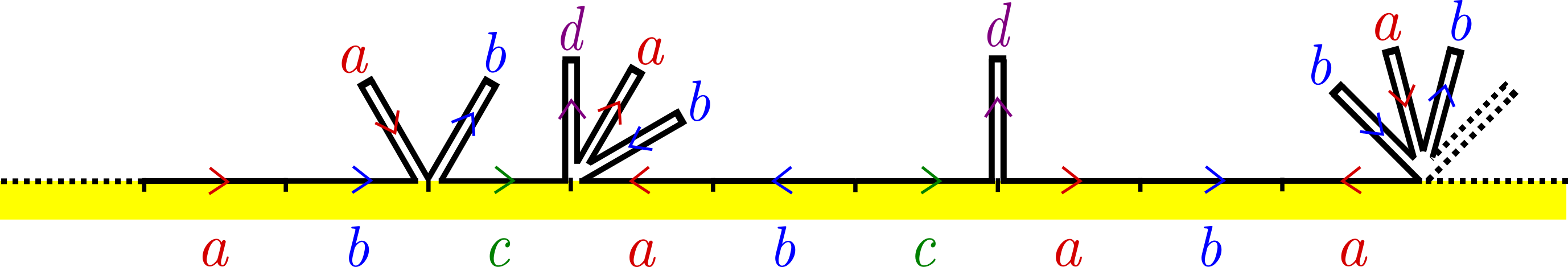}
\par\end{centering}
\caption{\label{fig:piece}A piece $P$ of $\partial Y_{+}$ is shown in black
line. The broken black line marks parts of $\partial Y_{+}$ adjacent
to but not part of $P$ and the yellow stripe marks the side of the
internal side of $Y$. This piece consists of 9 full directed edges
and 9 hanging half-edges, so $\protect\defect\left(P\right)=-18$.}
\end{figure}

\begin{defn}[$\varepsilon$-adapted]
\label{def:e-adapted} Let $\varepsilon\ge0$ and let $Y$ be a tiled
surface. A piece $P$ of $\partial Y_{+}$ is \emph{$\varepsilon$-adapted}
if it satisfies\footnote{In general, if $\defect(P)\leq2g\cdot\chi(P)-\varepsilon|P|.$ }
\begin{equation}
\defect(P)\leq4\chi(P)-\varepsilon|P|.\label{eq:defect-ineq}
\end{equation}
We have $\chi(P)=0$ if $P$ is a whole boundary component and $\chi(P)=1$
otherwise. We say that a piece $P$ is \emph{$\varepsilon$-bad} if
(\ref{eq:defect-ineq}) does not hold, i.e., if $\defect(P)>4\chi(P)-\varepsilon|P|$.
We say that $Y$ is $\varepsilon$-adapted if every piece of $Y$
is $\varepsilon$-adapted. 
\end{defn}

The following lemma shows that this notion indeed quantifies the notion
of strongly boundary reduced tiled surfaces.
\begin{lem}
\label{lem:BR and SBR are 0 and ge 0 adapted} Let $Y$ be a tiled
surface.
\begin{enumerate}
\item $Y$ is boundary reduced if and only if it is $0$-adapted.
\item $Y$ is strongly boundary reduced if and only if every piece of $\partial Y$
is $\varepsilon$-adapted for some $\varepsilon>0$.\\
If $Y$ is compact, this is equivalent to that $Y$ is $\varepsilon$-adapted
for some $\varepsilon>0$.
\end{enumerate}
\end{lem}

\begin{proof}
A block at $\partial Y$ is a piece $P$ with $\he\left(P\right)=0$.
Assume that $Y$ is $0$-adapted. If $P$ is a block at $\partial Y$,
then $\e\left(P\right)=\defect\left(P\right)\le4\chi\left(P\right)\le4$,
so $P$ cannot be a long block. If $P$ is a long chain at $\partial Y$
consisting of $k$ blocks ($k-2$ of length $3$ and two of length
$4$) and the $k-1$ hanging edges between them, then $\defect\left(P\right)=\left(3k+2\right)-3\left(k-1\right)=5>4=4\chi\left(P\right)$,
which is a contradiction. Similarly, if $P$ is a half-block or a
half-chain, then $\defect\left(P\right)=4\chi\left(P\right)$, and
so $P$ is $\varepsilon$-bad for any $\varepsilon>0$. The converse
implications are not hard and can be found in \cite[proof of Lem.~5.18]{MPasympcover}.
\end{proof}
We need the following lemma in the analysis of the next subsection.
\begin{lem}
\label{lem:lenght-of-bad-piece}If $0\le\varepsilon<3$ and $P$ is
an $\varepsilon$-bad piece of a compact tiled surface $Y$, then\footnote{For arbitrary $g\ge2$, the bound is $\left|P\right|<\frac{2g\cdot\d\left(Y\right)}{2g-1-\varepsilon}$.}
\begin{equation}
\left|P\right|<\frac{4\d\left(Y\right)}{3-\varepsilon}.\label{eq:upper bound for |P|}
\end{equation}
\end{lem}

\begin{proof}
If $P$ is $\varepsilon$-bad, then by definition $\e\left(P\right)-3\cdot\he\left(P\right)>4\chi\left(P\right)-\varepsilon\left|P\right|$.
So
\[
\left(3-\varepsilon\right)\left|P\right|<3\left(\e\left(P\right)+\he\left(P\right)\right)+\left(\e\left(P\right)-3\cdot\he\left(P\right)-4\chi\left(P\right)\right)\le4\cdot\e\left(P\right)\le4\d\left(Y\right).
\]
\end{proof}

\subsection{The octagons-vs-boundary algorithm\label{subsec:The-octagons-vs-boundary-algorit}}

In this $\S\S$\ref{subsec:The-octagons-vs-boundary-algorit} we describe
an algorithm whose purpose is to grow a given tiled surface in such
a way that either 
\begin{itemize}
\item the output $Y'$ is $\varepsilon$-adapted for some fixed $\varepsilon>0$,
or alternatively,
\item the number of octagons of $Y'$ is larger than the length of the boundary
of $Y'$.
\end{itemize}
If $Y'$ is $\varepsilon$-adapted for a suitable $\varepsilon$,
it is very well adapted to our methods, so that we can give an estimate
for $\E_{n}^{\emb}(Y')$ with an effective error term (e.g., Proposition
\ref{prop:main-term-asymp}). If, on the other hand, $\f(Y')>\d(Y')$,
then the Euler characteristic of $Y'$ can be linearly comparable
to the number of octagons in $Y'$ by Lemma \ref{lem:Euler-char-bound},
and see $\S\S$\ref{subsec:Part-I:non-e-adapted} where it is used.

The algorithm depends on a positive constant $\varepsilon>0$; we
shall see below that fixing $\varepsilon=\frac{1}{32}$ works fine
for our needs (for arbitrary $g\ge2$ we shall fix $\varepsilon=\frac{1}{16g}$.)
To force the algorithm to be deterministic, we a priori make some
choices: 
\begin{notation}
For every compact tiled surface $Y$ which is boundary reduced but
not $\varepsilon$-adapted, we pick an $\varepsilon$-bad piece $P(Y)$
of $\partial Y$.
\end{notation}

With the ambient parameter $\varepsilon$ fixed as well as the choices
of $\varepsilon$-bad pieces, the octagons-vs-boundary (OvB) algorithm
is as follows.
\begin{center}
\textbf{}%
\noindent\fbox{\begin{minipage}[t]{1\columnwidth - 2\fboxsep - 2\fboxrule}%
\textbf{Input. }An embedding of tiled surfaces $Y\hookrightarrow Z$
where $Y$ is compact and $Z$ has no boundary. %
\end{minipage}}
\par\end{center}

\begin{center}
\textbf{}%
\noindent\fbox{\begin{minipage}[t]{1\columnwidth - 2\fboxsep - 2\fboxrule}%
\textbf{Output. }A compact tiled surface $Y'$ and a factorization
of the input embedding $Y\hookrightarrow Z$ by $Y\hookrightarrow Y'\hookrightarrow Z$
where both maps are embeddings.%
\end{minipage}}
\par\end{center}

\begin{center}
\textbf{}%
\noindent\fbox{\begin{minipage}[t]{1\columnwidth - 2\fboxsep - 2\fboxrule}%
\textbf{Algorithm. }Let $Y'=Y$.

\textbf{(a)} Let $Y'=\br(Y'\hookrightarrow Z)$. If
\begin{equation}
\theta(Y')\eqdf\f(Y')-\d(Y')>0\label{eq:terminating-inequality}
\end{equation}
terminate the algorithm and return $Y'$. 

\textbf{(b) }If $Y'$ is not $\varepsilon$-adapted, add all the octagons
of $Z$ meeting\footnote{An octagon $O$ in $Z$ is said to meet $P\left(Y'\right)$ if some
directed edge or hanging-half-edge of $P\left(Y'\right)$ lies at
$\partial O$.} $P(Y')$ to $Y'$, and go to \textbf{(a)}.

Return $Y'$.%
\end{minipage}}
\par\end{center}

Note that the output $Y'$ of the algorithm is always boundary reduced.
Of course, we would like to know when/if this algorithm terminates.

In step \textbf{(a)}, if $\br\left(Y'\hookrightarrow Z\right)\ne Y'$
then $\d(Y')$ decreases by at least two, and $\f(Y')$ increases
by at least one. So $\theta(Y')$ increases by at least three.

In step \textbf{(b)}, if $Y'$ changes, the following lemma shows
that $\theta(Y')$ increases by at least one provided that $\varepsilon\le\frac{1}{16}$.

\begin{lem}
\label{lem:step-c-theta-increase}With notation as above, if $Y'$
is modified in step \textbf{(b)}, then
\begin{enumerate}
\item $\d(Y')$ increases by less than $2\varepsilon\left|P\left(Y'\right)\right|$.
\item $\theta(Y')$ increases by more than $\left(\frac{1}{8}-2\varepsilon\right)\left|P\left(Y'\right)\right|$,
so the increase is positive when\footnote{For arbitrary $g\ge2$, $\theta\left(Y'\right)$ increases by more
than $\left(\frac{1}{4g}-2\varepsilon\right)\left|P\left(Y'\right)\right|$,
so we need $\varepsilon\le\frac{1}{8g}$.} $\varepsilon\le\frac{1}{16}$.
\end{enumerate}
\end{lem}

Note that $\theta\left(Y'\right)$ is an integer, so any positive
increase is an increase by at least one.
\begin{proof}
Suppose that in step \textbf{(b)} $Y'$ is modified. Let $Y''$ denote
the result of this modification and let $P=P(Y')$. Let $k$ denote
the number of new octagons added. First assume that $P$ is a non-closed
path, so $\chi\left(P\right)=1$. We have $k\le\he\left(P\right)+1$
because every hanging half-edge along $P$ marks the passing from
one new octagon to the next one. Every new octagon borders $8$ edges
in $Z$. For most new octagons, two of these edges contain hanging
half-edges of $P$ and are internal edges in $Y''$, so if $j$ of
the edges belong to $P$, the net contribution of the octagon to $\d\left(Y''\right)-\d\left(Y'\right)$
is at most $6-2j$. The exceptions are the two extreme octagons, which
possibly meet only one hanging half-edge of $P$, and contribute a
net of at most $7-2j$. The sum of the parameter $j$ over all new
octagons is exactly $\e\left(P\right)$. In total, we obtain:
\begin{eqnarray*}
\d\left(Y''\right)-\d\left(Y'\right) & \le & 6k+2-2\cdot\e\left(P\right)\\
 & \le & 6\left(\he\left(P\right)+1\right)+2-2\cdot\e\left(P\right)\\
 & = & 2\left(3\cdot\he\left(P\right)-\e\left(P\right)\right)+8\\
 & < & 2\left(\varepsilon\left|P\right|-4\chi\left(P\right)\right)+8=2\cdot\varepsilon\left|P\right|,
\end{eqnarray*}
where the last inequality comes from the definition of an $\varepsilon$-bad
piece. If $P$ is a whole boundary cycle of $Y'_{+}$, we have $k\le\he\left(P\right)$
and all octagons contribute at most $6-2j$ to $\d\left(Y''\right)-\d\left(Y'\right)$,
so 
\begin{eqnarray*}
\d\left(Y''\right)-\d\left(Y'\right) & \le & 6k-2\cdot\e\left(P\right)\le6\cdot\he\left(P\right)-2\cdot\e\left(P\right)<2\left(\varepsilon\left|P\right|-4\chi\left(P\right)\right)=2\varepsilon\left|P\right|.
\end{eqnarray*}
This proves Part 1.

There is a total of $8k$ directed edges at the boundaries of the
new octagons. Of these, $\e\left(P\right)$ are edges of $P$. Each
of the remaining $8k-\e\left(P\right)$ can `host' two hanging half-edges
of $P$, and each hanging half-edge appears in exactly $2$ directed
edges of new octagons. This gives
\[
2\he\left(P\right)\le2\left(8k-\e\left(P\right)\right),
\]
so $8k\ge\he\left(P\right)+\e\left(P\right)=\left|P\right|$. Hence
\[
\theta\left(Y''\right)-\theta\left(Y'\right)=k-\left(\d\left(Y''\right)-\d\left(Y'\right)\right)>\frac{1}{8}\left|P\right|-2\varepsilon\left|P\right|=\left(\frac{1}{8}-2\varepsilon\right)\left|P\right|.
\]
\end{proof}
The upshot of the previous observations and Lemma \ref{lem:step-c-theta-increase}
is that, provided $\varepsilon\le\frac{1}{16}$, every time step \textbf{(a)}
of the algorithm is reached, except for the first time, $Y'$ has
changed in step \textbf{(b)}, so $\text{\ensuremath{\theta(Y')}}$
has increased by at least one. Since 
\[
\theta(Y)=\f(Y)-\d(Y)\geq-\d(Y),
\]
and the algorithm halts at the latest after the first time that $\theta\left(Y\right)$
is positive, we deduce the following lemma:
\begin{lem}
If $\varepsilon\le\frac{1}{16}$, then during the octagons-vs-boundary
algorithm, step \textbf{(a) }is reached at most $\d(Y)+2$ times.
In particular, the algorithm always terminates.
\end{lem}

Now that we know the algorithm always terminates (assuming $\varepsilon\le\frac{1}{16}$),
and it clearly has deterministic output due to our a priori choices,
if $Y\hookrightarrow Z$ is an embedding of a compact tiled surface
$Y$ into a tiled surface $Z$ without boundary we write $\OVB_{\varepsilon}(Y\hookrightarrow Z)$\marginpar{${\scriptstyle \protect\OVB_{\varepsilon}(Y\protect\hookrightarrow Z)}$}
for the output of the OvB algorithm with parameter $\varepsilon$
applied to $Y\hookrightarrow Z$. Thus $\OVB_{\varepsilon}(Y\hookrightarrow Z)$
is a tiled surface $Y'$ with an attached embedding $Y\hookrightarrow Y'$.
We can now make the following easy observation.
\begin{lem}
\label{lem:two-options-for-OVB-alg}Let $\varepsilon\le\frac{1}{16}$,
let $Y\hookrightarrow Z$ be an embedding of a compact tiled surface
$Y$ into a tiled surface $Z$ without boundary, and let $Y'=\OVB_{\varepsilon}(Y\hookrightarrow Z)$.
Then at least one of the following holds:
\begin{itemize}
\item $Y'$ is $\varepsilon$-adapted.
\item $Y'$ is boundary reduced and $\f(Y')>\d(Y')$.
\end{itemize}
\end{lem}

We also want an upper bound on how $\d(Y')$ and $\f\left(Y'\right)$
increase during the OvB algorithm. 
\begin{lem}
\label{lem:face-and-boundary-change-during-OvB}Assume\footnote{For arbitrary $g\ge2$, we pick $\varepsilon\le\frac{1}{16g}$. The
statement of the lemma holds as is. } $\varepsilon\le\frac{1}{32}$. Let $Y$ be a compact tiled surface,
$Z$ be a boundary-less tiled surface and denote $\overline{Y}=\OVB_{\varepsilon}\left(Y\hookrightarrow Z\right)$.
Then
\begin{align}
\d(\overline{Y}) & \leq3\d(Y),\label{eq:boundary change during OvB}\\
\f(\overline{Y}) & \leq\f(Y)+4\d\left(Y\right)+\d(Y)^{2}.\label{eq:octagon change during OvB}
\end{align}
\end{lem}

\begin{proof}
If step \textbf{(a) }is only reached once, then the result of the
algorithm, $\overline{Y}$, is equal to $\br(Y\hookrightarrow Z)$.
In this case we have $\d\left(\overline{Y}\right)\le\d\left(Y\right)$
and $\f\left(\overline{Y}\right)\le\f\left(Y\right)+\frac{\d\left(Y\right)^{2}}{6}$
by Proposition \ref{prop:properties-of-BR-alg} part \ref{enu:octagon-growth},
so the statement of the lemma holds. So from now on suppose step \textbf{(a)}
is reached more than once. 

Let $Y_{1}=Y'$ at the penultimate time that step \textbf{(a)} \uline{is
completed}. Between the penultimate time that step \textbf{(a)} is
completed and the algorithm terminates, step \textbf{(b)} takes place
to form $Y_{2}=Y'$, and then step \textbf{(a)} takes place one more
time to form $Y_{3}=\overline{Y}$ which is the output of the algorithm.

First we prove the bound on $\d\left(Y_{3}\right)$. We have $\theta(Y_{1})\leq0$,
so 
\[
\theta\left(Y_{1}\right)-\theta\left(Y\right)\le0-\left(\f\left(Y\right)-\d\left(Y\right)\right)\le\d\left(Y\right).
\]
We claim that in every step of the OvB algorithm, the increase in
$\theta$ is larger then the increase in $\d$. Indeed, this is obviously
true in step \textbf{(a)}, where $\theta$ does not decrease and $\d$
does not increase. It is also true in step \textbf{(b)} by Lemma \ref{lem:step-c-theta-increase}
and our assumption that $\varepsilon\le\frac{1}{32}$. Therefore,
\[
\d\left(Y_{1}\right)-\d\left(Y\right)\le\theta\left(Y_{1}\right)-\theta\left(Y\right)\le\d\left(Y\right),
\]
and we conclude that $\d\left(Y_{1}\right)\le2\d\left(Y\right)$.

Let $P=P\left(Y_{1}\right)$. By Lemma \ref{lem:step-c-theta-increase},
\begin{eqnarray*}
\d\left(Y_{2}\right) & \le & \d\left(Y_{1}\right)+2\varepsilon\left|P\right|\stackrel{\eqref{eq:upper bound for |P|}}{\le}\d\left(Y_{1}\right)+2\varepsilon\cdot\frac{4\d\left(Y_{1}\right)}{3-\varepsilon}\\
 & = & \d\left(Y_{1}\right)\left[1+\frac{8\varepsilon}{3-\varepsilon}\right]\le1.1\cdot\d\left(Y_{1}\right)\le2.2\cdot\d\left(Y\right),
\end{eqnarray*}
where the penultimate inequality is based on that $\varepsilon\le\frac{1}{32}$.
Finally, $\d\left(Y_{3}\right)\le\d\left(Y_{2}\right)$, so (\ref{eq:boundary change during OvB})
is proven.

For the number of octagons, note first that 
\[
\f\left(Y_{1}\right)=\theta\left(Y_{1}\right)+\d\left(Y_{1}\right)\le\d\left(Y_{1}\right)\le2\d\left(Y\right).
\]
Let $k$ denote the number of new octagons added in step \textbf{(b)
}to form $Y_{2}$ from $Y_{1}$. As noted in the proof of Lemma \ref{lem:step-c-theta-increase},
$k\le\he\left(P\right)+1$. As $P=P\left(Y_{1}\right)$ is $\varepsilon$-bad,
we have
\[
\he\left(P\right)\le\frac{1}{3}\left(\e\left(P\right)+\varepsilon\left|P\right|\right)\stackrel{\eqref{eq:upper bound for |P|}}{\le}\frac{1}{3}\d\left(Y_{1}\right)\left(1+\frac{4\varepsilon}{3-\varepsilon}\right)<\d\left(Y_{1}\right)\le2\d\left(Y\right),
\]
the penultimate inequality is based again on that $\varepsilon\le\frac{1}{32}$.
Thus $\f\left(Y_{2}\right)-\f\left(Y_{1}\right)\le\he\left(P\right)+1\le2\d\left(Y\right)$.

Finally, by Proposition \ref{prop:properties-of-BR-alg} part \ref{enu:octagon-growth},
$\f\left(Y_{3}\right)-\f\left(Y_{2}\right)\le\frac{\d\left(Y_{2}\right)^{2}}{6}\le\d\left(Y\right)^{2}$,
and we conclude 
\begin{eqnarray*}
\f\left(Y_{3}\right) & = & \f\left(Y_{1}\right)+\left[\f\left(Y_{2}\right)-\f\left(Y_{1}\right)\right]+\left[\f\left(Y_{3}\right)-\f\left(Y_{2}\right)\right]\\
 & \le & 2\d\left(Y\right)+2\d\left(Y\right)+\d\left(Y\right)^{2}=4\d\left(Y\right)+\d\left(Y\right)^{2},
\end{eqnarray*}
which proves (\ref{eq:octagon change during OvB}) in this case as
well.
\end{proof}

\subsection{Resolutions from the octagons-vs-boundary algorithm}

Recall the definition of the tiled surface $X_{\phi}$ from $\S$\ref{sec:Introduction}
and Example \ref{exa:x_phi}. Given a tiled surface $Y$, we define
\[
\E_{n}(Y)\eqdf\E_{\phi\in\X_{n}}[\#\text{morphisms \ensuremath{Y\to X_{\phi}].}}
\]
This is the expected number of morphisms from $Y$ to $X_{\phi}$.
Recall that we use the uniform probability measure on $\X_{n}$. We
have the following result that relates this concept to Theorem \ref{thm:effective-error}.
\begin{lem}
\label{lem:fix-gamma-fix-cycle}Given $1\ne\gamma\in\Gamma$, let
$\mathcal{C}_{\gamma}$ be as in Example \ref{exa:the-loop-of-a-word}.
Then
\begin{equation}
\E_{n}[\fix_{\gamma}]=\E_{n}(\CC_{\gamma}).\label{eq:fix_gamma_by_cycles}
\end{equation}
\end{lem}

\begin{proof}
This is not hard to check but also follows from \cite[Lem.~2.7]{MPasympcover}.
\end{proof}
We need to work not only with $\E_{n}(Y)$ for various tiled surfaces,
but also with the expected number of times that $Y$ \emph{embeds
}into $X_{\phi}$. For a tiled surface $Y$, this is given by 
\[
\E_{n}^{\emb}(Y)\eqdf\E_{\phi\in\X_{n}}[\#\text{embeddings \ensuremath{Y\hookrightarrow X_{\phi}].}}
\]
We recall the following definition from \cite[Def.~2.8]{MPasympcover}.
\begin{defn}[Resolutions]
A resolution $\mathcal{\mathcal{R}}$ of a tiled surface $Y$ is
a collection of morphisms of tiled surfaces 
\[
\mathcal{R}=\left\{ f\colon Y\to W_{f}\right\} ,
\]
such that every morphism $h\colon Y\to Z$ of $Y$ into a tiled surface
$Z$ with no boundary decomposes \uline{uniquely} as $Y\stackrel{f}{\to}W_{f}\stackrel{\overline{h}}{\hookrightarrow}Z$,
where $f\in\mathcal{R}$ and $\overline{h}$ is an \uline{embedding}.
\end{defn}

The point of this definition is the following lemma also recorded
in \cite[Lem.~2.9]{MPasympcover}.
\begin{lem}
\label{lem:resolution-sum-of-expectations}If $Y$ is a compact tiled
surface and $\mathcal{R}$ is a finite resolution of $Y$, then
\begin{equation}
\mathbb{E}_{n}\left(Y\right)=\sum_{f\in\mathcal{R}}\mathbb{E}_{n}^{\emb}\left(W_{f}\right).\label{eq:resolution}
\end{equation}
\end{lem}

The type of resolution we wish to use in this paper is the following.
\begin{defn}[$\mathcal{R}_{\varepsilon}(Y)$]
\label{def:epsilon-resolution} For a compact tiled surface $Y$,
let $\mathcal{R}_{\varepsilon}(Y)$ denote the collection of all morphisms
$Y\xrightarrow{f}W_{f}$ obtained as follows:
\begin{itemize}
\item $F:Y\to Z$ is a morphism of $Y$ into a boundary-less tiled surface
$Z$.
\item $U_{F}$ is the image of $F$ in $Z$. Hence there is a given embedding
$\iota_{F}:U_{F}\hookrightarrow Z$.
\item $W_{f}$ is given by $W_{f}=\OVB_{\varepsilon}(U_{f}\hookrightarrow Z)$
and $f=\iota_{F}\circ F\colon Y\to W_{f}$.
\end{itemize}
\end{defn}

\begin{thm}
\label{thm:certification of resolution}Given a compact tiled surface
$Y$ amd $\varepsilon\le\frac{1}{32}$ (or $\varepsilon\le\frac{1}{16g}$
for arbitrary $g\ge2$), the collection $\mathcal{R}_{\varepsilon}(Y)$
defined in Definition \ref{def:epsilon-resolution} is a finite resolution
of $Y$.
\end{thm}

\begin{proof}
To see that $\mathcal{R}_{\varepsilon}(Y)$ is finite, note that there
are finitely many options for $U_{F}$ (this is a quotient of the
compact complex $Y$). For any such $U_{F}$ we have $\f(U_{F})\leq\f(Y)$
and $\d(U_{F})\leq\d(Y)$, and hence by Lemma \ref{lem:face-and-boundary-change-during-OvB}
there is a bound on $\f(W_{f})$ depending only on $Y$. As we add
a bounded number of octagons to obtain $W_{f}$, there is a bound
also on $\v\left(W_{f}\right)$ and on $\e\left(W_{f}\right)$. This
means that $W_{f}$ is one of only finitely many tiled surfaces, and
there are finitely many morphisms of $Y$ to one of these.

Now we explain why $\mathcal{R}_{\varepsilon}(Y)$ is a resolution
-- this is essentially the same as \cite[proof of Thm.~2.14]{MPasympcover}.
Let $F\colon Y\to Z$ be a morphism with $\partial Z=\emptyset$.
By the definition of $\mathcal{R}_{\varepsilon}(Y)$, it is clear
that $F$ decomposes as $Y\stackrel{f}{\to}W_{f}\hookrightarrow Z$
for the $f\in\mathcal{R_{\varepsilon}}(Y)$ that originates in $F$.
To show uniqueness, assume that $F$ decomposes in an additional way
\[
Y\stackrel{f'}{\to}W_{f'}\hookrightarrow Z
\]
where $W_{f'}$ is the result of the $OvB$ algorithm for some $F'\colon Y\to Z'$
with $\partial Z'=\emptyset$. We claim that both decompositions are
precisely the same decomposition of $F$ (namely $W_{f'}=W_{f}$ and
$f'=f$). First, $U_{F'}=F'\left(Y\right)\hookrightarrow W_{f'}\hookrightarrow Z$,
so $U_{F'}=F'\left(Y\right)=F\left(Y\right)=U_{F}$. The OvB algorithm
with input $F'\left(Y\right)\hookrightarrow Z'$ takes place entirely
inside $W_{f'}$, and does not depend on the structure of $Z'\backslash W_{f'}$:
the choices are made depending only on the structure of the boundary
of $Y'$ in step \textbf{(b)} of the OvB algorithm, as well as in
every step of the procedure described in Proposition \ref{prop:properties-of-BR-alg}(\ref{enu:BR-algo})
to obtain $\br\left(Y'\hookrightarrow Z\right)$ in step \textbf{(a)}.
Moreover, the result of these steps depends only on the octagons of
$Z$ immediately adjacent to the boundary of $Y'$. But $W_{f'}$
is embedded in $Z$, and so it must be identical to $W_{f}$ and $f'$
identical to $f$.
\end{proof}
It is the following corollary of the previous results, applied to
a tiled surface $\CC_{\gamma}$ as in Example \ref{exa:the-loop-of-a-word},
that will be used in the rest of the paper. Recall that for $\gamma\in\Gamma$,
$\ell_{w}\left(\gamma\right)$ denotes the word-length, with respect
to the generators $\left\{ a,b,c,d\right\} $, of a shortest representative
of the conjugacy class of $\gamma$ in $\Gamma$.
\begin{cor}
\label{cor:facts-about-the-resolution}Let $1\ne\gamma\in\Gamma$
and\footnote{For arbitrary $g\ge2$, take $\varepsilon\le\frac{1}{16g}$. The same
result holds.} $\varepsilon\le\frac{1}{32}$. For any $f:\CC_{\gamma}\to W_{f}$
in $\mathcal{R}_{\varepsilon}(\CC_{\gamma})$, either
\begin{enumerate}
\item $W_{f}$ is boundary reduced, and $\chi(W_{f})<-\f(W_{f})<-\d(W_{f})$,
or
\item $W_{f}$ is $\varepsilon$-adapted.
\end{enumerate}
Moreover, in either case,
\begin{align}
\d(W_{f}) & \leq6\ell_{w}\left(\gamma\right)\label{eq:res-delta-control}\\
\f(W_{f}) & \leq8\ell_{w}\left(\gamma\right)+4\left(\ell_{w}\left(\gamma\right)\right)^{2}.\label{eq:res-F-control}
\end{align}
\end{cor}

\begin{proof}
The inequalities (\ref{eq:res-delta-control}) and (\ref{eq:res-F-control})
are from Lemma \ref{lem:face-and-boundary-change-during-OvB} and
the fact that $\d\left(\CC_{\gamma}\right)=2\ell_{w}\left(\gamma\right)$
and $\f\left(\CC_{\gamma}\right)=0$. It follows from the construction
of $\mathcal{R}_{\varepsilon}(\CC_{\gamma})$ using the OvB algorithm
that if $f\in\mathcal{R}_{\varepsilon}(Y)$ with $f:Y\to W_{f}$,
and $W_{f}$ is not $\varepsilon$-adapted, then $W_{f}$ is boundary
reduced and $\d(W_{f})<\f(W_{f})$. Combined with Lemma \ref{lem:Euler-char-bound}
this gives
\begin{align*}
\chi(W_{f}) & \leq-2\f(W_{f})+\frac{1}{2}\d(W_{f})<-2\f(W_{f})+\frac{1}{2}\f(W_{f})\leq-\f(W_{f}).
\end{align*}
\end{proof}

\section{Representation theory of symmetric groups\label{sec:rep-theory-symmetric}}

\subsection{Background\label{subsec:rep-theory-background}}

We write $S_{n}$ for the symmetric group of permutations of the set
$[n]$. By convention $S_{0}$ is the trivial group with one element.
If $m\le n$, we always let $S_{m}\leq S_{n}$ be the subgroup of
permutations fixing $[m+1,n]$ element-wise. For $k\leq n$, we will
let $S'_{k}\leq S_{n}$ be our notation for the subgroup of permutations
fixing $[n-k]$ element-wise. We write $\C[S_{n}]$ for the group
algebra of $S_{n}$ with complex coefficients.

\subsubsection*{Young diagrams}

A Young diagram (YD) of size $n$ is a collection of $n$ boxes, arranged
in left-aligned rows in the plane, such that the number of boxes in
each row is non-increasing from top to bottom. A Young diagram is
uniquely specified by the sequence $\lambda_{1},\lambda_{2},\ldots,\lambda_{r}$
where $\lambda_{i}$ is the number of boxes in the $i$th row (and
there are $r$ rows). We have $\lambda_{1}\geq\lambda_{2}\geq\cdots\geq\lambda_{r}>0$;
such a sequence of integers is called a \emph{partition. }We view
YDs and partitions interchangeably in this paper. If $\sum_{i}\lambda_{i}=n$
we write $\lambda\vdash n$. Two important examples of partitions
are $(n)$, with all boxes of the corresponding YD in the first row,
and $(1)^{n}\eqdf(\underbrace{1,\ldots,1}_{n})$, with all boxes of
the corresponding YD in the first column. If $\mu,\lambda$ are YDs,
we write $\mu\subset\lambda$ if all boxes of $\mu$ are contained
in $\lambda$ (when both are aligned to the same top-left borders).
We say $\mu\subset_{k}\lambda$ if $\mu\subset\lambda$ and there
are $k$ boxes of $\lambda$ that are not in $\mu$. We write $\emptyset$
for the empty YD with no boxes. If $\lambda$ is a YD, $\check{\lambda}$
is the\emph{ }conjugate YD obtained by reflecting $\lambda$ in the
diagonal (switching rows and columns). 

A skew Young diagram (SYD) is a pair of Young diagrams $\mu$ and
$\lambda$ with $\mu\subset\lambda$. This pair is denoted $\lambda/\mu$
and thought of as the collection of boxes of $\lambda$ that are not
in $\mu$. We identify a YD $\lambda$ with the SYD $\lambda/\emptyset$
so that YDs are special cases of SYDs. The size of a SYD $\lambda/\mu$
is the number of boxes it contains; i.e.~the number of boxes of $\lambda$
that are not in $\mu$. The size is denoted by $|\lambda/\mu|$, or
if $\lambda$ is a YD, $|\lambda|$.

\subsubsection*{Young tableaux}

Let $\lambda/\mu$ be a SYD, with $\lambda\vdash n$ and $\mu\vdash k$.
A standard Young tableau of shape $\lambda/\mu$ is a filling of the
boxes of $\lambda/\mu$ with the numbers $[k+1,n]$ such that each
number appears in exactly one box and the numbers in each row (resp.~column)
are strictly increasing from left to right (resp.~top to bottom).
We refer to standard Young tableaux just as \emph{tableaux} in this
paper. We write $\Tab(\lambda/\mu)$ for the collection of tableaux
of shape $\lambda/\mu$. Given a tableau $T$, we denote by $T\lvert_{\leq m}$
(resp.~$T\lvert_{>m})$ the tableau formed by the numbers-in-boxes
of $T$ with numbers in the set $[m]$ (resp.~$[m+1,n]$). The shape
of $T\lvert_{\le m}$ and of $T\lvert_{>m}$ is a SYD in general.
If $T$ is a tableau and the shape of $T$ is a YD we let $\mu_{m}(T)$
be the YD that is the shape of $T\lvert_{\leq m}$. If $\nu\subset\mu\subset\lambda$,
$T\in\Tab(\mu/\nu)$ and $R\in\Tab(\lambda/\mu)$, then we write $T\sqcup R$
for the tableau in $\Tab(\lambda/\nu)$ obtained by adjoining $R$
to $T$ in the obvious way.

\subsubsection*{Irreducible representations}

The equivalence classes of irreducible unitary representations of
$S_{n}$ are in one-to-one correspondence with Young diagrams of size
$n$. Given a YD $\lambda\vdash n$, we write $V^{\lambda}$ for the
corresponding irreducible representation of $S_{n}$; each $V^{\lambda}$
is a finite dimensional Hermitian complex vector space with an action
of $S_{n}$ by unitary linear automorphisms. Hence $V^{\lambda}$
can also be thought of as a module for $\C[S_{n}]$. We write $d_{\lambda}\eqdf\dim V^{\lambda}$.
It is well-known, and also follows from the discussion of the next
paragraphs, that $d_{\lambda}=|\Tab(\lambda)|$. Note that $d_{\lambda}=d_{\check{\lambda}}$
since reflection in the diagonal gives a bijection between $\Tab(\lambda)$
and $\Tab(\check{\lambda})$.

We now give an account of the Vershik-Okounkov approach to the representation
theory of symmetric groups from \cite{VershikOkounkov}. According
to the usual ordering of $[n]$ there is a filtration of subgroups
\[
S_{0}\leq S_{1}\leq S_{2}\leq\cdots\leq S_{n}.
\]
If $W$ is any unitary representation of $S_{n}$, $m\in[n]$ and
$\mu\vdash m$, we write $W_{\mu}$ for the span of vectors in copies
of $V^{\mu}$ in the restriction of $W$ to $S_{m}$; we call $W_{\mu}$
the $\mu$-isotypic subspace of $W$.

It follows from the branching law for restriction of representations
between $S_{m}$ and $S_{m-1}$ that for $\lambda\vdash n$ and $T\in\Tab(\lambda)$
the intersection
\[
\left(V^{\lambda}\right)_{\mu_{1}(T)}\cap\left(V^{\lambda}\right)_{\mu_{2}(T)}\cap\cdots\cap\left(V^{\lambda}\right)_{\mu_{n-1}(T)}
\]
is one-dimensional. Vershik-Okounkov specify a unit vector $v_{T}$
in this intersection. The collection 
\[
\{\,v_{T}\,:\,T\in\Tab(\lambda)\,\}
\]
 is an orthonormal basis for $V^{\lambda}$ called a Gelfand-Tsetlin
basis.

\subsubsection*{Modules from SYDs}

If $m,n\in\N$, $\lambda\vdash n$, $\mu\vdash m$ and $\mu\subset\lambda$,
then 
\[
V^{\lambda/\mu}\eqdf\Hom_{S_{m}}(V^{\mu},V^{\lambda})
\]
is a unitary representation of $S'_{n-m}$ as $S'_{n-m}$ is in the
centralizer of $S_{m}$ in $S_{n}$. We write $d_{\lambda/\mu}$ for
the dimension of this representation. There is also an analogous Gelfand-Tsetlin
orthonormal basis of $V^{\lambda/\mu}$ indexed by $T\in\Tab(\lambda/\mu$);
the basis element corresponding to a skew tableau $T$ will be denoted
$w_{T}$. It follows that $d_{\lambda/\mu}=|\Tab(\lambda/\mu)|$.
Note that when $\mu=\lambda$, $\Tab\left(\lambda/\mu\right)=\left\{ \emptyset\right\} $
($\emptyset$ the empty tableau), and the representation $V^{\lambda/\mu}$
is one-dimensional with basis $w_{\emptyset}$. 

One has the following consequence of Frobenius reciprocity (cf. e.g.~\cite[Lem.~3.1]{MPasympcover}).
\begin{lem}
\label{lem:induced-rep-dimension}Let $n\in\N$, $m\in\left[n\right]$
and $\mu\vdash m$. Then
\[
\sum_{\lambda\vdash n\colon\mu\subset\lambda}d_{\lambda/\mu}d_{\lambda}=\frac{n!}{m!}d_{\mu}.
\]
\end{lem}

\subsection{Effective bounds for dimensions\label{subsec:Effective-bounds-for-dimensions}}

Throughout the paper, we will write $b_{\lambda}$\marginpar{$b_{\lambda},\check{b}_{\lambda}$}
for the number of boxes outside the first row of a YD $\lambda$,
and write $\check{b}_{\lambda}$ for the number of boxes outside the
first column of $\lambda$. More generally, we write $b_{\lambda/\nu}$\marginpar{$b_{\lambda/\nu},\check{b}_{\lambda/\nu}$}
(resp.~$\check{b}_{\lambda/\nu}$) for the number of boxes outside
the first row (resp.~column) of the SYD $\lambda/\nu$, so $b_{\lambda/\nu}=b_{\lambda}-b_{\nu}$
and $\check{b}_{\lambda/\nu}=\check{b}_{\lambda}-\check{b}_{\nu}$.
We need the following bounds on dimensions of representations.
\begin{lem}
\label{lem:dim-ratio-bound}\cite[Lem.~4.3]{MPasympcover} If $n\in\N$,
$m\in\left[n\right]$, $\lambda\vdash n$, $\nu\vdash m$, $\nu\subset\lambda$
and $m\geq2b_{\lambda}$, then
\begin{equation}
\frac{(n-b_{\lambda})^{b_{\lambda}}}{b_{\lambda}^{~b_{\lambda}}m^{b_{\nu}}}\leq\frac{d_{\lambda}}{d_{\nu}}\leq\frac{b_{\nu}^{~b_{\nu}}n^{b_{\lambda}}}{(m-b_{\nu})^{b_{\nu}}}.\label{eq:dim-ratio-bound}
\end{equation}
\end{lem}

The condition $m\geq2b_{\lambda}$ ensures that both $\nu$ and $\lambda$
have most boxes in their first row. This is an important and recurring
theme of the paper (see e.g.~Proposition \ref{prop:effective-Liebeck-Shalev}).
\begin{lem}
\label{lem:dimension-bound-of-skew-module}Let $\lambda/\nu$ be a
skew Young diagram of size $n$. Then
\[
d_{\lambda/\nu}\leq(n)_{b_{\lambda/\nu}}~~~~~~~\mathrm{and~~~~~~~}d_{\lambda/\nu}\leq(n)_{\check{b}_{\lambda/\nu}}.
\]
\end{lem}

\begin{proof}
There are at most $\binom{n}{b_{\lambda/\nu}}$ options for the set
of  $b_{\lambda/\nu}$ elements outside the first row. Given these,
there are at most $b_{\lambda/\nu}!$ choices for how to place them
outside the first row. The proof of the second inequality is analogous.
\end{proof}

\subsection{Effective bounds for the zeta function of the symmetric group\label{subsec:Effective-bounds-for-the-zeta-function}}

The Witten zeta function of the symmetric group $S_{n}$ is defined
for a real parameter $s$ as
\begin{equation}
\zeta^{S_{n}}(s)\eqdf\sum_{\lambda\vdash n}\frac{1}{d_{\lambda}^{~s}}.\label{eq:zeta-function-def}
\end{equation}
This function, and various closely related functions, play a major
role in this paper. One main reason for its appearance is due to a
formula going back to Hurwitz \cite{hurwitz1902ueber} that states
\begin{equation}
|\X_{g,n}|=|\Hom(\Gamma_{g},S_{n})|=|S_{n}|^{2g-1}\zeta^{S_{n}}(2g-2).\label{eq:mednykhs-formula}
\end{equation}
This is also sometimes called Mednykh's formula \cite{Mednyhk}. We
first give the following result due to Liebeck and Shalev \cite[Thm.~1.1]{LiebeckShalev}
and independently, Gamburd \cite[Prop. 4.2]{gamburd2006poisson}.
We refer the reader to $\S\S$\ref{subsec:Notation} for the definition
of notations (e.g.~$O,\ll$) that we use in this $\S$\ref{sec:rep-theory-symmetric}.
\begin{thm}
\cite{LiebeckShalev,gamburd2006poisson}\label{thm:zeta approximation}
For any $s>0$, as $n\to\infty$
\[
\zeta^{S_{n}}(s)=2+O\left(n^{-s}\right).
\]
\end{thm}

This has the following corollary when combined with (\ref{eq:mednykhs-formula}).
\begin{cor}
\label{cor:size-of-X_n}For any $g\in\N$ with $g\geq2$, we have
\[
\frac{|\X_{g,n}|}{(n!)^{2g-1}}=2+O(n^{-2}).
\]
\end{cor}

As well as the previous results, we also need to know how well $\zeta^{S_{n}}(2g-2)$
is approximated by restricting the summation in (\ref{eq:zeta-function-def})
to $\lambda$ with a bounded number of boxes either outside the first
row or outside the first column. We let $\Lambda(n,b)$ denote the
collection of $\lambda\vdash n$ such that $\lambda_{1}\leq n-b$
and $\check{\lambda}_{1}\le n-b$. In other words, $\Lambda(n,b)$
is the collection of YDs $\lambda\vdash n$ with both $b_{\lambda}\geq b$
and $\check{b}_{\lambda}\geq b$. A version of the next proposition,
when $b$ is fixed and $n\to\infty$, is due independently to Liebeck
and Shalev \cite[Prop. 2.5]{LiebeckShalev} and Gamburd \cite[Prop. 4.2]{gamburd2006poisson}.
Here, we need a version that holds uniformly over $b$ that is not
too large compared to $n$. 
\begin{prop}
\label{prop:effective-Liebeck-Shalev}Fix $s>0$. There exists a constant
$\kappa=\kappa(s)>1$ such that when $b^{2}\le\frac{n}{3}$, 
\begin{equation}
\sum_{\lambda\in\Lambda(n,b)}\frac{1}{d_{\lambda}^{~s}}\ll_{s}\left(\frac{\kappa b^{2s}}{\left(n-b^{2}\right)^{s}}\right)^{b}.\label{eq:liebeck-shalev-gamburd-effectiv}
\end{equation}
\end{prop}

\begin{proof}
Here we follow Liebeck and Shalev \cite[proof of Prop. 2.5]{LiebeckShalev}
and make the proof uniform over $b$. Let $\Lambda_{0}(n,b)$ denote
the collection of $\lambda\vdash n$ with $\check{\lambda}_{1}\leq\lambda_{1}\leq n-b$.
Since $d_{\lambda}=d_{\check{\lambda}}$, 
\[
\sum_{\lambda\in\Lambda(n,b)}\frac{1}{d_{\lambda}^{~s}}\leq2\sum_{\lambda\in\Lambda_{0}(n,b)}\frac{1}{d_{\lambda}^{~s}},
\]
so it suffices to prove a bound for $\sum_{\lambda\in\Lambda_{0}(n,b)}\frac{1}{d_{\lambda}^{~s}}$.
Let $\Lambda_{1}(n,b)$ denote the elements $\lambda$ of $\Lambda_{0}(n,b)$
with $\lambda_{1}\geq\frac{2n}{3}$. We write 
\[
\sum_{\lambda\in\Lambda_{0}(n,b)}\frac{1}{d_{\lambda}^{~s}}=\Sigma_{1}+\Sigma_{2}
\]
 where
\[
\Sigma_{1}\eqdf\sum_{\lambda\in\Lambda_{1}(n,b)}\frac{1}{d_{\lambda}^{~s}},\quad\Sigma_{2}\eqdf\sum_{\lambda\in\Lambda_{0}(n,b)-\Lambda_{1}(n,b)}\frac{1}{d_{\lambda}^{~s}}.
\]

\noindent \textbf{Bound for $\Sigma_{1}$}. By \cite[Lem.~2.1]{LiebeckShalev}
if $\lambda\in\Lambda_{1}(n,b)$ then since $\lambda_{1}\geq\frac{n}{2}$,
$d_{\lambda}\geq\binom{\lambda_{1}}{n-\lambda_{1}}.$ Indeed, for
completeness, following \cite[Proof of Lem.~2.1]{LiebeckShalev} we
can find many tableaux of shape $\lambda$ as follows. Put the numbers
$1,\ldots,n-\lambda_{1}$ in the left most entries of the first row
of $\lambda$. Then for any of the $\binom{\lambda_{1}}{n-\lambda_{1}}$
choices of size $n-\lambda_{1}$ subsets of $[n-\lambda_{1}+1,n]$,
there is obviously a tableau of shape $\lambda$ with those numbers
outside the first row.

\noindent Let $p(m)$ denote the number of $\mu\vdash m$. The number
of $\lambda\in\Lambda_{1}(n,b)$ with a valid fixed value of $\lambda_{1}$
is $p(n-\lambda_{1})$ since $\lambda_{1}\geq\frac{n}{2}$ and hence
any YD with $n-\lambda_{1}$ boxes can be added below the fixed first
row of $\lambda_{1}$ boxes to form $\lambda$. Therefore
\[
\Sigma_{1}\le\sum_{\lambda_{1}=\lceil\frac{2n}{3}\rceil}^{n-b}\frac{p\left(n-\lambda_{1}\right)}{\binom{\lambda_{1}}{n-\lambda_{1}}^{s}}=\sum_{\ell=b}^{\lfloor\frac{n}{3}\rfloor}\frac{p(\ell)}{\binom{n-\ell}{\ell}^{s}}.
\]
We now split the sum into two ranges to estimate $\Sigma_{1}\leq\Sigma'_{1}+\Sigma''_{1}$
where 
\[
\Sigma'_{1}=\sum_{\ell=b}^{b^{2}}\frac{p(\ell)}{\binom{n-\ell}{\ell}^{s}},\quad\Sigma''_{1}=\sum_{\ell=b^{2}+1}^{\lfloor\frac{n}{3}\rfloor}\frac{p(\ell)}{\binom{n-\ell}{\ell}^{s}}.
\]
First we deal with $\Sigma'_{1}$. We have $p(\ell)\leq c_{1}^{\sqrt{\ell}}$
for some $c_{1}>1$ \cite[Thm. 14.5]{ApostolNT}. As $\ell\le n-\ell$,
\[
\binom{n-\ell}{\ell}\geq\frac{(n-\ell)^{\ell}}{\ell^{\ell}}.
\]
This gives 
\begin{align}
\Sigma'_{1} & \leq\sum_{\ell=b}^{b^{2}}c_{1}^{\sqrt{\ell}}\left(\frac{\ell}{n-\ell}\right)^{s\ell}\le c_{1}^{~b}\sum_{\ell=b}^{b^{2}}\left(\frac{b^{2}}{n-b^{2}}\right)^{s\ell}\ll_{s}c_{1}^{~b}\left(\frac{b^{2}}{n-b^{2}}\right)^{sb},\label{eq:sigma-bound-1}
\end{align}
where the last inequality used that $\frac{b^{2}}{(n-b^{2})}\leq\frac{1}{2}$
as we assume $b^{2}\leq\frac{n}{3}$.

To deal with $\Sigma''_{1}$ we make the following claim.

\noindent \emph{Claim. }There is $n_{00}>0$ such that when $n\geq n_{00}$
and $\ell\leq\frac{n}{3}$
\begin{equation}
\binom{n-\ell}{\ell}\geq\left(\frac{2n}{3}\right)^{\sqrt{\ell}}.\label{eq:claim-inequality}
\end{equation}

\noindent \emph{Proof of claim. }Observe that when $\ell\leq\frac{n}{3}$
\begin{align*}
\binom{n-\ell}{\ell} & \geq\frac{(n-\ell)^{\ell}}{\ell^{\ell}}=(n-\ell)^{\sqrt{\ell}}\left(n-\ell\right)^{\ell-\sqrt{\ell}}\ell^{-\ell}\\
 & \geq\left(\frac{2n}{3}\right)^{\sqrt{\ell}}\left(2\ell\right)^{\ell-\sqrt{\ell}}\ell^{-\ell}=\left(\frac{2n}{3}\right)^{\sqrt{\ell}}\left(\frac{2^{\sqrt{\ell}-1}}{\ell}\right)^{\sqrt{\ell}}.
\end{align*}
We have $2^{\sqrt{\ell}-1}\geq\ell$ when $\ell\geq49$ which proves
the claim in this case. On the other hand, it is easy to see that
there is a $n_{00}>0$ such that (\ref{eq:claim-inequality}) holds
when $n\geq n_{00}$ and $1\leq\ell<49$. \emph{This proves the claim.}

The claim gives
\[
\Sigma''_{1}\leq\sum_{\ell=b^{2}+1}^{\lfloor\frac{n}{3}\rfloor}\left(\frac{c_{2}}{n^{s}}\right)^{\sqrt{\ell}}
\]
for some $c_{2}=c_{2}(s)>1$ when $n\geq n_{00}$. Let $n_{0}=n_{0}(s)\geq n_{00}$
be such that when $n\geq n_{0}$, $\frac{c_{2}}{n^{s}}<e^{-1}$. Let
$q=q\left(n\right)\eqdf\frac{c_{2}}{n^{s}}.$ Then when $n\geq n_{0}$,
$\log(q)\leq-1$ and
\[
\Sigma''_{1}\leq\int_{b^{2}}^{\infty}q^{\sqrt{x}}dx=\frac{2q^{b}}{\log q}\left(\frac{1}{\log q}-b\right).
\]
We obtain
\begin{equation}
\Sigma''_{1}\leq2(b+1)q^{b}\leq\frac{2(b+1)c_{2}^{~b}}{n^{sb}}.\label{eq:sigma-bound-2}
\end{equation}
Together with (\ref{eq:sigma-bound-1}) this yields:
\begin{equation}
\Sigma_{1}\ll_{s}c_{1}^{~b}\left(\frac{b^{2}}{n-b^{2}}\right)^{sb}+\frac{2(b+1)c_{2}^{~b}}{n^{sb}}\ll_{s}\left(\frac{\kappa b^{2s}}{\left(n-b^{2}\right)^{s}}\right)^{b}\label{eq:sigma_1_bound}
\end{equation}
with $\kappa=\kappa\left(s\right)=\max\left(c_{1},c_{2}\right)$.

\noindent \textbf{Bound for $\Sigma_{2}$. }If $\lambda\in\Lambda_{0}(n,b)-\Lambda_{1}(n,b)$
then $\check{\lambda}_{1}\leq\lambda_{1}<\frac{2n}{3}$ and \cite[Prop. 2.4]{LiebeckShalev}
gives the existence of an absolute $c_{0}>1$ such that 
\[
d_{\lambda}\geq c_{0}^{~n}.
\]
Thus for large enough $n$ and $b^{2}\leq\frac{n}{3}$
\begin{equation}
\Sigma_{2}\leq\sum_{\lambda\in\Lambda_{0}(n,b)-\Lambda_{1}(n,b)}c_{0}^{-ns}\leq p(n)c_{0}^{-ns}\leq c_{1}^{\sqrt{n}}c_{0}^{-ns}\ll_{s}n^{-bs}.\label{eq:sigma-bound-3}
\end{equation}
Putting (\ref{eq:sigma_1_bound}) and (\ref{eq:sigma-bound-3}) together
proves the proposition.
\end{proof}

\section{Estimates for the probabilities of tiled surfaces\label{sec:Estimates-for-the-probabilities-of-tiled-surfaces}}

Before reading this $\S$\ref{sec:Estimates-for-the-probabilities-of-tiled-surfaces},
we recommend the reader to have read $\S\S$\ref{subsec:Overview-of-the-paper}
for context and motivation.

\subsection{Prior results\label{subsec:Prior-results}}

The aim of this $\S\S$\ref{subsec:Prior-results} is to introduce
an already known formula (Theorem \ref{thm:E_n-emb-exact-expression})
for the quantities $\E_{n}^{\emb}(Y)$ that are essential to this
paper, and to give some known first estimates for the quantities appearing
therein (Lemma \ref{lem:product-of-matrix-coefs-bounded-using-D}).
To better understand their source and logic, the reader is advised
to look at \cite[\S 5]{MPasympcover}. 

We continue to assume $g=2$. Throughout this entire $\S$\ref{sec:Estimates-for-the-probabilities-of-tiled-surfaces}
we will assume that $Y$ is a fixed compact tiled surface. We let
$\v=\v(Y)$, $\e=\e(Y)$, $\f=\f(Y)$ denote the number of vertices,
edges, and octagons of $Y$, respectively. Throughout this section,
$f$ will stand for one of the letters $a,b,c,d$. For each letter
$f\in\{a,b,c,d\}$, let $\e_{f}$ denote the number of $f$-labeled
edges of $Y$.

In \cite[\S\S 5.3]{MPasympcover} we constructed permutations
\begin{align*}
\sigma_{f}^{+},\sigma_{f}^{-},\tau_{f}^{+},\tau_{f}^{-}\in S'_{\v}\subset S_{n}
\end{align*}
for each $f\in\{a,b,c,d\}$ satisfying certain five properties named
\textbf{P1},\textbf{ P2},\textbf{ P3},\textbf{ P4},\textbf{ }and \textbf{P5
}that are essential to the development of the theory, but not illuminating
to state precisely here. We henceforth view these permutations as
fixed, given $Y$. 

Recall from $\S\S$\ref{subsec:rep-theory-background} that for YDs
$\mu\subset\lambda$ we say $\mu\subset_{k}\lambda$ if $\lambda$
has $k$ more boxes than $\mu$. Also recall from $\S\S$\ref{subsec:rep-theory-background}
that $\sqcup$ denotes concatenation of Young tableaux, and for a
SYD $\lambda/\nu$, if $T$ is a (standard) tableau of shape $\lambda/\nu$,
$w_{T}$ denotes a Gelfand-Tsetlin basis vector in $V^{\lambda/\nu}$
associated to $T$. In the same situation, we write $\langle\bullet,\bullet\rangle$
for the inner product in the unitary representation $V^{\lambda/\nu}$.
In the prequel paper \cite[Thm. 5.10]{MPasympcover} the following
theorem was proved.
\begin{thm}
\label{thm:E_n-emb-exact-expression}For $n\geq\v$ we have
\begin{equation}
\E_{n}^{\emb}(Y)=\frac{\left(n!\right)^{3}}{\left|\X_{n}\right|}\cdot\frac{\left(n\right)_{\v}\left(n\right)_{\f}}{\prod_{f}\left(n\right)_{\e_{f}}}\cdot\Xi_{n}(Y)\label{eq:prob-to-xi-relation}
\end{equation}
where
\begin{equation}
\Xi_{n}(Y)\eqdf\sum_{\substack{\lambda,\nu:\\
\nu\subset_{\v-\f}\lambda\vdash n-\f
}
}d_{\lambda}d_{\nu}\sum_{\substack{\mu_{a},\mu_{b},\mu_{c},\mu_{d}\\
\forall f,\,\nu\subset\mu_{f}\subset_{\e_{f}-\f}\lambda
}
}\frac{1}{d_{\mu_{a}}d_{\mu_{b}}d_{\mu_{c}}d_{\mu_{d}}}\Upsilon_{n}\left(\left\{ \sigma_{f}^{\pm},\tau_{f}^{\pm}\right\} ,\nu,\left\{ \mu_{f}\right\} ,\lambda\right),\label{eq:Xi-def}
\end{equation}
\begin{equation}
\Upsilon_{n}\left(\left\{ \sigma_{f}^{\pm},\tau_{f}^{\pm}\right\} ,\nu,\left\{ \mu_{f}\right\} ,\lambda\right)\eqdf\sum_{\begin{gathered}r_{f}^{+},r_{f}^{-}\in\Tab\left(\mu_{f}/\nu\right)\\
s_{f},t_{f}\in\Tab\left(\lambda/\mu_{f}\right)
\end{gathered}
}\M\left(\left\{ \sigma_{f}^{\pm},\tau_{f}^{\pm},r_{f}^{\pm},s_{f},t_{f}\right\} \right)\label{eq:Upsilon-def}
\end{equation}
and $\M(\{\sigma_{f}^{\pm},\tau_{f}^{\pm},r_{f}^{\pm},s_{f},t_{f}\})$
is the following product of matrix coefficients:
\begin{eqnarray}
\M\left(\left\{ \sigma_{f}^{\pm},\tau_{f}^{\pm},r_{f}^{\pm},s_{f},t_{f}\right\} \right) & \eqdf & \left\langle \sigma_{b}^{-}\left(\sigma_{a}^{+}\right)^{-1}w_{r_{a}^{+}\sqcup s_{a}},w_{r_{b}^{-}\sqcup s{}_{b}}\right\rangle \left\langle \tau_{a}^{+}\left(\sigma_{b}^{+}\right)^{-1}w_{r_{b}^{+}\sqcup s_{b}},w_{r_{a}^{+}\sqcup t_{a}}\right\rangle \cdot\nonumber \\
 &  & \left\langle \tau_{b}^{+}\left(\tau_{a}^{-}\right)^{-1}w_{r_{a}^{-}\sqcup t{}_{a}},w_{r_{b}^{+}\sqcup t_{b}}\right\rangle \left\langle \sigma_{c}^{-}\left(\tau_{b}^{-}\right)^{-1}w_{r_{b}^{-}\sqcup t{}_{b}},w_{r_{c}^{-}\sqcup s{}_{c}}\right\rangle \cdot\nonumber \\
 &  & \left\langle \sigma_{d}^{-}\left(\sigma_{c}^{+}\right)^{-1}w_{r_{c}^{+}\sqcup s_{c}},w_{r_{d}^{-}\sqcup s{}_{d}}\right\rangle \left\langle \tau_{c}^{+}\left(\sigma_{d}^{+}\right)^{-1}w_{r_{d}^{+}\sqcup s_{d}},w_{r_{c}^{+}\sqcup t_{c}}\right\rangle \cdot\nonumber \\
 &  & \left\langle \tau_{d}^{+}\left(\tau_{c}^{-}\right)^{-1}w_{r_{c}^{-}\sqcup t{}_{c}},w_{r_{d}^{+}\sqcup t_{d}}\right\rangle \left\langle \sigma_{a}^{-}\left(\tau_{d}^{-}\right)^{-1}w_{r_{d}^{-}\sqcup t{}_{d}},w_{r_{a}^{-}\sqcup s{}_{a}}\right\rangle .\label{eq:M-def}
\end{eqnarray}
\end{thm}

Note that $\frac{\left(n!\right)^{3}}{\left|\X_{n}\right|}\stackrel{n\to\infty}{\to}2$
by (\ref{eq:mednykhs-formula}) and Theorem \ref{thm:zeta approximation},
and that $\frac{\left(n\right)_{\v}\left(n\right)_{\f}}{\prod_{f}\left(n\right)_{\e_{f}}}=n^{\chi\left(Y\right)}\left(1+O\left(n^{-1}\right)\right)$,
so the more mysterious term in (\ref{eq:prob-to-xi-relation}) is
$\Xi_{n}\left(Y\right)$. In light of Theorem \ref{thm:E_n-emb-exact-expression},
we will repeatedly discuss $\nu,\{\mu_{f}\},\lambda$ satisfying
\begin{equation}
\nu\subset_{\v-\e_{f}}\mu_{f}\subset_{\e_{f}-\f}\lambda\vdash n-\f\,\,\,~~~~~~~~\forall f\in\{a,b,c,d\}\label{eq:nu-mu-lambda-setup}
\end{equation}
and $\{r_{f}^{\pm},s_{f},t_{f}\}$ satisfying 
\begin{equation}
r_{f}^{+},r_{f}^{-}\in\Tab(\mu_{f}/\nu),\quad s_{f},t_{f}\in\Tab(\lambda/\mu_{f})\,\,\,~~~~~~~~\forall f\in\{a,b,c,d\}.\label{eq:r-s-t-setup}
\end{equation}
To give good estimates for $\Xi_{n}(Y)$, we need an effective bound
for the quantities $\M(\{\sigma_{f}^{\pm},\tau_{f}^{\pm},r_{f}^{\pm},s_{f},t_{f}\})$
that was obtained in \cite{MPasympcover}. Before giving this bound,
we recall some notation. For $T\in\Tab(\lambda/\nu)$, we write $\tp(T)$
for the set of elements in the top row of $T$ (the row of length
$\lambda_{1}-\nu_{1}$ which may be empty). For any two sets $A,B$
in $[n]$, we define $d(A,B)=|A\backslash B|$. Given $\{r_{f}^{\pm},s_{f},t_{f}\}$
as in (\ref{eq:r-s-t-setup}), we define
\begin{align}
 & D_{\tp}\left(\left\{ \sigma_{f}^{\pm},\tau_{f}^{\pm},r_{f}^{\pm},s_{f},t_{f}\right\} \right)\eqdf\label{eq:Dtopdef}\\
 & ~~~d\left(\sigma_{b}^{-}\left(\sigma_{a}^{+}\right)^{-1}\tp(r_{a}^{+}\sqcup s_{a}),\tp(r_{b}^{-}\sqcup s{}_{b})\right)+d\left(\tau_{a}^{+}\left(\sigma_{b}^{+}\right)^{-1}\tp(r_{b}^{+}\sqcup s_{b}),\tp(r_{a}^{+}\sqcup t_{a})\right)+\nonumber \\
 & ~~~d\left(\tau_{b}^{+}\left(\tau_{a}^{-}\right)^{-1}\tp(r_{a}^{-}\sqcup t{}_{a}),\tp(r_{b}^{+}\sqcup t_{b})\right)+d\left(\sigma_{c}^{-}\left(\tau_{b}^{-}\right)^{-1}\tp(r_{b}^{-}\sqcup t{}_{b}),\tp(r_{c}^{-}\sqcup s{}_{c})\right)+\nonumber \\
 & ~~~d\left(\sigma_{d}^{-}\left(\sigma_{c}^{+}\right)^{-1}\tp(r_{c}^{+}\sqcup s_{c}),\tp(r_{d}^{-}\sqcup s{}_{d})\right)+d\left(\tau_{c}^{+}\left(\sigma_{d}^{+}\right)^{-1}\tp(r_{d}^{+}\sqcup s_{d}),\tp(r_{c}^{+}\sqcup t_{c})\right)+\nonumber \\
 & ~~~d\left(\tau_{d}^{+}\left(\tau_{c}^{-}\right)^{-1}\tp(r_{c}^{-}\sqcup t{}_{c}),\tp(r_{d}^{+}\sqcup t_{d})\right)+d\left(\sigma_{a}^{-}\left(\tau_{d}^{-}\right)^{-1}\tp(r_{d}^{-}\sqcup t{}_{d}),\tp(r_{a}^{-}\sqcup s{}_{a})\right).\nonumber 
\end{align}

\begin{lem}
\cite[Lem.~5.14]{MPasympcover}\label{lem:product-of-matrix-coefs-bounded-using-D}
Let $\nu,\{\mu_{f}\},\lambda$ be as in (\ref{eq:nu-mu-lambda-setup})
and $\{r_{f}^{\pm},s_{f},t_{f}\}$ be as in (\ref{eq:r-s-t-setup}).
If $\lambda_{1}+\nu_{1}>n-\f+(\v-\f)^{2}$, then 
\[
\left|\M\left(\left\{ \sigma_{f}^{\pm},\tau_{f}^{\pm},r_{f}^{\pm},s_{f},t_{f}\right\} \right)\right|\leq\left(\frac{(\v-\f)^{2}}{\lambda{}_{1}+\nu_{1}-(n-\f)}\right)^{D_{\tp}\left(\left\{ \sigma_{f}^{\pm},\tau_{f}^{\pm},r_{f}^{\pm},s_{f},t_{f}\right\} \right)}.
\]
\end{lem}

The condition $\lambda_{1}+\nu_{1}>n-\f+(\v-\f)^{2}$ corresponds
to the bound given by Lemma \ref{lem:product-of-matrix-coefs-bounded-using-D}
being non-trivial, and we will be applying Lemma \ref{lem:product-of-matrix-coefs-bounded-using-D}
when both $\lambda$ and $\nu$ have $O(n^{1/4})$ boxes outside their
first rows and $\v,\f\ll n^{1/4}$. In particular, $\lambda_{1}+\nu_{1}$
is of order $2n$, while $\f$ and $\left(\v-\f\right)^{2}$ are of
much smaller order. Hence the condition will be met for sufficiently
large $n$.

Recall from $\S\S$\ref{subsec:Effective-bounds-for-dimensions} that
$b_{\nu}$ is the number of boxes of a Young diagram $\nu$ outside
the first row, and $\check{b}_{\nu}$ is the number of boxes outside
the first column. We have the following trivial upper bound for $D_{\tp}(\{\sigma_{f}^{\pm},\tau_{f}^{\pm},r_{f}^{\pm},s_{f},t_{f}\})$:
\begin{eqnarray}
D_{\tp}\left(\left\{ \sigma_{f}^{\pm},\tau_{f}^{\pm},r_{f}^{\pm},s_{f},t_{f}\right\} \right) & \le & 8\left(b_{\lambda}-b_{\nu}\right).\label{eq:dtop-upper-bound}
\end{eqnarray}
We recall the following estimate obtained in \cite[Prop. 5.22]{MPasympcover}.
\begin{prop}
\label{prop:geometric-D-bounds}Let $\varepsilon\geq0$. Suppose that
$\nu,\{\mu_{f}\},\lambda$ are as in (\ref{eq:nu-mu-lambda-setup})
and $\{r_{f}^{\pm},s_{f},t_{f}\}$ are as in (\ref{eq:r-s-t-setup}).
If $Y$ is $\varepsilon$-adapted then
\begin{align}
D_{\tp}\left(\left\{ \sigma_{f}^{\pm},\tau_{f}^{\pm},r_{f}^{\pm},s_{f},t_{f}\right\} \right) & \geq b_{\lambda}+3b_{\nu}-b_{\mu_{a}}-b_{\mu_{b}}-b_{\mu_{c}}-b_{\mu_{d}}+\varepsilon b_{\lambda/\nu}.\label{eq:d-to-b-bound}
\end{align}
\end{prop}

\begin{figure}
\begin{centering}
$\begin{array}{cccccc}
\nu= & \ydiagram{3,1} & \mu_{a}= & \ydiagram{4,2,1} & \lambda= & \ydiagram{6,2,1}\\
|\nu|= & n-\v=4~~~~~~~~~~~~ & |\mu_{a}|= & n-\e_{a}=7~~~~~~~~~~~~ & |\lambda|= & n-\f=9~~~~~~~~~~~~~~\\
b_{\nu}= & 1~~~~~~~~~~~~~~~~~~~~~~~ & b_{\mu_{a}}= & 3~~~~~~~~~~~~~~~~~~~~~~~ & b_{\lambda}= & 3~~~~~~~~~~~~~~~~~~~~~~~
\end{array}$
\par\end{centering}
\caption{An example of possible $\nu,\mu_{a},\lambda$ appearing in Theorem
\ref{thm:E_n-emb-exact-expression} (supposing e.g.~$n=10,\protect\v=6,\protect\e_{a}=3,\protect\f=1$)}
\end{figure}

\subsection{Partitioning $\Xi_{n}$ and preliminary estimates\label{subsec:PartitioningXiandpreliminaryestimates}}

In this $\S\S\ref{subsec:PartitioningXiandpreliminaryestimates}$
we show how the condition that $Y$ is $\varepsilon$-adapted leads
to bounds on $\Xi_{n}$. We continue to view $Y$ as fixed and hence
suppress dependence of quantities on $Y$. We write $\D=\D\left(Y\right)\eqdf\v-\f$.
Note that $\D\ge0$ by (\ref{eq:v-e-f-inequality}), with equality
if and only if $Y$ has no boundary. By Lemma \ref{lem:D-vs-d}, $\D\le\d\left(Y\right)$.
So $\D$ is another measure of the size of the boundary of $Y$, and
it plays an important role in some of our bounds below. We will use
the notation $\Xi_{n}^{P(\nu)}$ where $P$ is a proposition concerning
$\nu$ to mean
\[
\Xi_{n}^{P(\nu)}\eqdf\sum_{\substack{\nu\subset_{\v-\f}\lambda\vdash\,n-\f\\
P(\nu)\text{ holds true}
}
}d_{\lambda}d_{\nu}\sum_{\nu\subset\mu_{f}\subset_{\e_{f}-\f}\lambda}\frac{1}{d_{\mu_{a}}d_{\mu_{b}}d_{\mu_{c}}d_{\mu_{d}}}\Upsilon_{n}\left(\left\{ \sigma_{f}^{\pm},\tau_{f}^{\pm}\right\} ,\nu,\left\{ \mu_{f}\right\} ,\lambda\right).
\]
We will continue to use this notation, for various propositions $P$,
throughout the rest of the paper. We want to give bounds for various
$\Xi_{n}^{P(\nu)}$ under the condition that $Y$ is either boundary
reduced (namely, $0$-adapted) or, moreover, $\varepsilon$-adapted
for some $\varepsilon>0$. We will always assume $\v\leq n^{1/4}$
and so also $\D=\v-\f\le n^{1/4}$. Note that $b_{\nu}\le\D$ and
$\check{b}_{\nu}\le\D$ cannot hold simultaneously as $\v,\D\leq n^{1/4}$,
and as all but one box of $\nu\vdash n-\v$ is either outside the
first row or first column, one has the simple inequality $b_{\nu}+\check{b}_{\nu}+1\geq n-\v$.

Then for $n\gg1$ we have 
\[
\Xi_{n}=\Xi_{n}^{\nu=(n-\v)}+\Xi_{n}^{\nu=(1)^{n-\v}}+\Xi_{n}^{0<b_{\nu}\leq\D;\check{b}_{\nu}>0}+\Xi_{n}^{0<\check{b}_{\nu}\leq\D;b_{\nu}>0}+\Xi_{n}^{b_{\nu},\check{b}_{\nu}>\D}.
\]
Moreover by \cite[Lem.~5.9]{MPasympcover} we have 
\[
\Xi_{n}^{\nu=(n-\v)}=\Xi_{n}^{\nu=(1)^{n-\v}},\quad\Xi_{n}^{0<b_{\nu}\leq\D;\check{b}_{\nu}>0}=\Xi_{n}^{0<\check{b}_{\nu}\leq\D;b_{\nu}>0}
\]
hence 
\begin{equation}
\Xi_{n}=2\Xi_{n}^{\nu=(n-\v)}+2\Xi_{n}^{0<b_{\nu}\leq\D;\check{b}_{\nu}>0}+\Xi_{n}^{b_{\nu},\check{b}_{\nu}>\D}.\label{eq:Xin-splitting}
\end{equation}
This is according to three regimes for $b_{\nu}$ and $\check{b}_{\nu}$:
\begin{itemize}
\item The \emph{zero regime}: when $b_{\nu}$ or $\check{b}_{\nu}$ equal
$0$. The contribution from here is $2\Xi_{n}^{\nu=(n-\nu)}$.
\item The \emph{intermediate regime:} when $b_{\nu},\check{b}_{\nu}>0$
but one of them is at most $\D$. The contribution from this regime
is $2\Xi_{n}^{0<b_{\nu}\leq\D;\check{b}_{\nu}>0}$. 
\item The \emph{large regime:} when both $b_{\nu},\check{b}_{\nu}>\D$.
The contribution from this regime is $\Xi_{n}^{b_{\nu},\check{b}_{\nu}>\D}$.
\end{itemize}
The strategy for bounding these different contributions is to further
partition the tuples $(\nu,\left\{ \mu_{f}\right\} ,\lambda)$ according
to the data $b_{\lambda},\{b_{\mu_{f}}\},b_{\nu},\check{b}_{\lambda},\{\check{b}_{\mu_{f}}\},\check{b}_{\nu}$. 
\begin{defn}
\label{def:B}For $\overline{B}=\left(B_{\lambda},\{B_{\mu_{f}}\},B_{\nu},\check{B}_{\lambda},\{\check{B}_{\mu_{f}}\},\check{B}_{\nu}\right)$
we write\marginpar{${\scriptstyle \left(\nu,\left\{ \mu_{f}\right\} ,\lambda\right)\vdash\overline{B}}$}
\[
\left(\nu,\left\{ \mu_{f}\right\} ,\lambda\right)\vdash\overline{B}
\]
if (\ref{eq:nu-mu-lambda-setup}) holds, and $\nu$, $\left\{ \mu_{f}\right\} $
and $\lambda$ have the prescribed number of blocks outside the first
row and outside the first column, namely,
\[
b_{\lambda}=B_{\lambda},\check{b}_{\lambda}=\check{B}_{\lambda},\,b_{\nu}=B_{\nu},\check{b}_{\nu}=\check{B}_{\nu}~~~~~~\mathrm{and}~~~~~~\forall f\in\left\{ a,b,c,d\right\} ~~b_{\mu_{f}}=B_{\mu_{f}},\check{b}_{\mu_{f}}=\check{B}_{\mu_{f}}.
\]
We denote by $\B_{n}\left(Y\right)$ the collection of tuples $\overline{B}$
which admit at least one tuple of YDs $(\nu,\left\{ \mu_{f}\right\} ,\lambda)$.
Finally, we let
\begin{align}
\Xi_{n}^{\overline{B}}=\Xi_{n}^{\overline{B}}\left(Y\right)\eqdf & \sum_{\left(\nu,\{\mu_{f}\},\lambda\right)\vdash\overline{B}}\frac{d_{\lambda}d_{\nu}}{d_{\mu_{a}}d_{\mu_{b}}d_{\mu_{c}}d_{\mu_{d}}}\Upsilon_{n}\left(\left\{ \sigma_{f}^{\pm},\tau_{f}^{\pm}\right\} ,\nu,\left\{ \mu_{f}\right\} ,\lambda\right).\label{eq:xi-B-def}
\end{align}
\end{defn}

Note that $\Xi_{n}\left(Y\right)=\sum_{\overline{B}\in\B_{n}\left(Y\right)}\Xi_{n}^{\overline{B}}$.
Also, note that $\overline{B}\in\B_{n}\left(Y\right)$ imposes restrictions
on the possible values of $B_{\lambda},\{B_{\mu_{f}}\},B_{\nu},\check{B}_{\lambda},\{\check{B}_{\mu_{f}}\},\check{B}_{\nu}$.
For example, for every $f\in\{a,b,c,d\}$, $0\le B_{\mu_{f}}-B_{\nu}\le\v-\e_{f}$
and $0\le B_{\lambda}-B_{\mu_{f}}\le\e_{f}-\f$, and likewise for
the $\check{B}$'s. In addition, $B_{\nu}+\check{B}_{\nu}+1\ge n-\nu$,
and so on. 

We first give a general estimate for the quotient of dimensions in
the summands in (\ref{eq:xi-B-def}).
\begin{lem}
\label{lem:generic-summand-bound}Suppose that $\v\leq n^{1/4}$ and
that $(\nu,\{\mu_{f}\},\lambda)$ satisfy (\ref{eq:nu-mu-lambda-setup}).
If $b_{\nu}\le\D$ then
\begin{equation}
\frac{d_{\lambda}d_{\nu}}{d_{\mu_{a}}d_{\mu_{b}}d_{\mu_{c}}d_{\mu_{d}}}\ll\frac{1}{d_{\nu}^{~2}}b_{\lambda}^{5b_{\lambda}}n^{\left(b_{\lambda}+3b_{\nu}-\sum_{f}b_{\mu_{f}}\right)}.\label{eq:generic-summand-ineq}
\end{equation}
\end{lem}

\begin{proof}
By Lemma \ref{lem:dim-ratio-bound}, 
\[
\frac{d_{\nu}}{d_{\mu_{f}}}\le\frac{b_{\mu_{f}}^{~b_{\mu_{f}}}\left(n-\v\right)^{b_{\nu}}}{(n-\e_{f}-b_{\mu_{f}})^{b_{\mu_{f}}}}\le\frac{b_{\lambda}^{~b_{\lambda}}n^{b_{\nu}}}{\left(n-2n^{1/4}\right)^{b_{\mu_{f}}}},
\]
where the second inequality is based on that $\e_{f}+b_{\mu_{f}}\le\e_{f}+\left(b_{\nu}+\v-\e_{f}\right)=b_{\nu}+\v\le2n^{1/4}$.
The hypotheses of Lemma \ref{lem:dim-ratio-bound} are met here since
\[
2b_{\mu_{f}}-|\nu|\leq2(b_{\nu}+\v-\e_{f})-(n-\v)\leq5\v-n\leq5n^{\frac{1}{4}}-n\leq0
\]
for $n\gg1$. Similarly, since $2b_{\lambda}-|\nu|\leq2(b_{\nu}+\D)-(n-\v)\leq5n^{\frac{1}{4}}-n\leq0$
for $n\gg1$, Lemma \ref{lem:dim-ratio-bound} gives $\frac{d_{\lambda}}{d_{\nu}}\le\frac{b_{\nu}^{b_{\nu}}(n-\f)^{b_{\lambda}}}{(n-\v-b_{\nu})^{b_{\nu}}}\le\frac{b_{\lambda}^{~b_{\lambda}}n^{b_{\lambda}}}{\left(n-2n^{1/4}\right)^{b_{\nu}}}.$
Altogether,
\begin{eqnarray*}
\frac{d_{\lambda}d_{\nu}^{~3}}{d_{\mu_{a}}d_{\mu_{b}}d_{\mu_{c}}d_{\mu_{d}}} & \le & \frac{b_{\lambda}^{~5b_{\lambda}}n^{\left(b_{\lambda}+4b_{\nu}\right)}}{\left(n-2n^{1/4}\right)^{b_{\nu}+\sum_{f}b_{\mu_{f}}}}=b_{\lambda}^{~5b_{\lambda}}n^{\left(b_{\lambda}+3b_{\nu}-\sum_{f}b_{\mu_{f}}\right)}\left(\frac{1}{1-2n^{-3/4}}\right)^{b_{\nu}+\sum_{f}b_{\mu_{f}}}\\
 & \le & b_{\lambda}^{~5b_{\lambda}}n^{\left(b_{\lambda}+3b_{\nu}-\sum_{f}b_{\mu_{f}}\right)}\cdot\left(\frac{1}{1-2n^{-3/4}}\right)^{9n^{1/4}}.
\end{eqnarray*}
As $\left(\frac{1}{1-2n^{-3/4}}\right)^{9n^{1/4}}\stackrel{n\to\infty}{\to}1$,
the right hand side of the last inequality is at most $2b_{\lambda}^{~5b_{\lambda}}n^{\left(b_{\lambda}+3b_{\nu}-\sum_{f}b_{\mu_{f}}\right)}$
for large enough $n$. 
\end{proof}
We next give bounds for the individual $\Xi_{n}^{\overline{B}}$.
\begin{lem}
\label{lem:xi-B-bounds-small-B_nu}There is $\kappa>1$ such that
if $Y$ is $\varepsilon$-adapted for $\varepsilon\geq0$, $\v\leq n^{1/4}$
and $B_{\nu}\le\D$, then
\begin{align*}
\left|\Xi_{n}^{\overline{B}}\right| & \ll B_{\lambda}^{~10B_{\lambda}}\left(\D^{24}n^{-\varepsilon}\right)^{B_{\lambda}-B_{\nu}}\left(\frac{\kappa\D^{4}}{\left(n-\v-\D^{2}\right)^{2}}\right)^{B_{\nu}}.
\end{align*}
\end{lem}

\begin{proof}
By assumption, $B_{\nu}\le\D\leq\v\leq n^{\frac{1}{4}}$. So for every
$\left(\nu,\{\mu_{f}\},\lambda\right)\vdash\overline{B}$, 
\[
\lambda_{1}+\nu_{1}-\left(n-\f\right)=\left(n-\f-B_{\lambda}\right)+\left(n-\v-B_{\nu}\right)-\left(n-\f\right)\ge n-\v-B_{\nu}-\left(B_{\nu}+\D\right)\ge n-4\v,
\]
and Lemma \ref{lem:product-of-matrix-coefs-bounded-using-D} gives
that whenever $\left\{ r_{f}^{\pm},s_{f},t_{f}\right\} $ satisfy
(\ref{eq:r-s-t-setup}),
\begin{equation}
\left|\M\left(\left\{ \sigma_{f}^{\pm},\tau_{f}^{\pm},r_{f}^{\pm},s_{f},t_{f}\right\} \right)\right|\le\left(\frac{\D^{2}}{n-4\v}\right)^{D_{\tp}\left(\left\{ \sigma_{f}^{\pm},\tau_{f}^{\pm},r_{f}^{\pm},s_{f},t_{f}\right\} \right)}.\label{eq:half-result-for-M}
\end{equation}
Proposition \ref{prop:geometric-D-bounds} gives
\[
B_{\lambda}+3B_{\nu}-\sum_{f}B_{\mu_{f}}\le D_{\tp}\left(\left\{ \sigma_{f}^{\pm},\tau_{f}^{\pm},r_{f}^{\pm},s_{f},t_{f}\right\} \right)-\varepsilon\left(B_{\lambda}-B_{\nu}\right),
\]
so by Lemma \ref{lem:generic-summand-bound}
\[
\frac{d_{\lambda}d_{\nu}^{3}}{d_{\mu_{a}}d_{\mu_{b}}d_{\mu_{c}}d_{\mu_{d}}}\left|\M\left(\left\{ \sigma_{f}^{\pm},\tau_{f}^{\pm},r_{f}^{\pm},s_{f},t_{f}\right\} \right)\right|\ll B_{\lambda}^{~5B_{\lambda}}n^{-\varepsilon\left(B_{\lambda}-B_{\nu}\right)}\left(\frac{n\D^{2}}{n-4\v}\right)^{D_{\tp}\left(\left\{ \sigma_{f}^{\pm},\tau_{f}^{\pm},r_{f}^{\pm},s_{f},t_{f}\right\} \right)}.
\]
Now using the trivial upper bound $D_{\tp}\left(\left\{ \sigma_{f}^{\pm},\tau_{f}^{\pm},r_{f}^{\pm},s_{f},t_{f}\right\} \right)\leq8(B_{\lambda}-B_{\nu})$
in (\ref{eq:dtop-upper-bound}) and $B_{\lambda}-B_{\nu}\le\v-\f\leq\v\le n^{\frac{1}{4}}$,
we obtain that for large enough $n$,
\[
\left(\frac{n\D^{2}}{n-4\v}\right)^{D_{\tp}\left(\left\{ \sigma_{f}^{\pm},\tau_{f}^{\pm},r_{f}^{\pm},s_{f},t_{f}\right\} \right)}\leq\D^{16\left(B_{\lambda}-B_{\nu}\right)}\left(\frac{1}{1-4n^{-3/4}}\right)^{8n^{1/4}}\leq2\D^{16(B_{\lambda}-B_{\nu})}.
\]
Therefore, 
\begin{align*}
\frac{d_{\lambda}d_{\nu}^{3}}{d_{\mu_{a}}d_{\mu_{b}}d_{\mu_{c}}d_{\mu_{d}}}\left|\M\left(\left\{ \sigma_{f}^{\pm},\tau_{f}^{\pm},r_{f}^{\pm},s_{f},t_{f}\right\} \right)\right| & \ll B_{\lambda}^{~5B_{\lambda}}\left(\D^{16}n^{-\varepsilon}\right)^{B_{\lambda}-B_{\nu}}.
\end{align*}
From this we obtain 
\begin{eqnarray*}
\left|\Xi_{n}^{\overline{B}}\right| & \ll & B_{\lambda}^{~5B_{\lambda}}\left(\D^{16}n^{-\varepsilon}\right)^{B_{\lambda}-B_{\nu}}\sum_{\left(\nu,\{\mu_{f}\},\lambda\right)\vdash\overline{B}}\frac{1}{d_{\nu}^{2}}\sum_{\begin{gathered}r_{f}^{+},r_{f}^{-}\in\Tab\left(\mu_{f}/\nu\right)\\
s_{f},t_{f}\in\Tab\left(\lambda/\mu_{f}\right)
\end{gathered}
}1\\
 & \le & B_{\lambda}^{~5B_{\lambda}}\left(\D^{24}n^{-\varepsilon}\right)^{B_{\lambda}-B_{\nu}}\sum_{\left(\nu,\{\mu_{f}\},\lambda\right)\vdash\overline{B}}\frac{1}{d_{\nu}^{2}}
\end{eqnarray*}
since there are at most $(\D)_{(B_{\lambda}-B_{\nu})}\leq\D^{(B_{\lambda}-B_{\nu})}$
choices of $r_{f}^{+}\sqcup s_{f}$ or of $r_{f}^{-}\sqcup t_{f}$
for all $f$, by Lemma \ref{lem:dimension-bound-of-skew-module}.
For fixed $\nu$ above, there are at most $B_{\lambda}^{5B_{\lambda}}$
choices of $\{\mu_{f}\}$ and $\lambda$ such that $(\nu,\{\mu_{f}\},\lambda)\vdash\overline{B}$.
For example, the boxes outside the first row of $\lambda$ uniquely
determine $\lambda$ and form a YD of size $B_{\lambda}$; there are
at most $B_{\lambda}!\leq B_{\lambda}^{B_{\lambda}}$ of these. Hence
\begin{align*}
\left|\Xi_{n}^{\overline{B}}\right| & \ll B_{\lambda}^{~10B_{\lambda}}\left(\D^{24}n^{-\varepsilon}\right)^{B_{\lambda}-B_{\nu}}\sum_{\nu\vdash n-\v\colon b_{\nu}=B_{\nu}}\frac{1}{d_{\nu}^{2}}.
\end{align*}

Note that above, we have $\nu_{1}=n-\v-B_{\nu}\geq n-2n^{\frac{1}{4}}$,
so $\check{b}_{\nu}\geq n-2n^{\frac{1}{4}}-1\geq n^{\frac{1}{4}}\geq B_{\nu}$
for $n\gg1$, and in this case $\nu\in\Lambda(n-\v,B_{\nu})$. Moreover,
for $n\gg1$, $B_{\nu}^{2}\leq n^{\frac{1}{2}}\leq\frac{n-n^{\frac{1}{4}}}{3}\leq\frac{n-\v}{3}$
and so we can finally apply Proposition \ref{prop:effective-Liebeck-Shalev}
to obtain for the same $\kappa=\kappa\left(2\right)>1$ from Proposition
\ref{prop:effective-Liebeck-Shalev} that
\begin{eqnarray*}
\left|\Xi_{n}^{\overline{B}}\right| & \ll & B_{\lambda}^{10B_{\lambda}}\left(\D^{24}n^{-\varepsilon}\right)^{B_{\lambda}-B_{\nu}}\left(\frac{\kappa B_{\nu}^{~4}}{\left(n-\v-B_{\nu}^{~2}\right)^{2}}\right)^{B_{\nu}}\\
 & \le & B_{\lambda}^{~10B_{\lambda}}\left(\D^{24}n^{-\varepsilon}\right)^{B_{\lambda}-B_{\nu}}\left(\frac{\kappa\D^{4}}{\left(n-\v-\D^{2}\right)^{2}}\right)^{B_{\nu}}.
\end{eqnarray*}
\end{proof}
Since Lemma \ref{lem:xi-B-bounds-small-B_nu} is only useful for $B_{\nu}$
or $\check{B}_{\nu}$ small compared to $n$ we have to supplement
it with the following weaker bound.
\begin{lem}
\label{lem:Xi-B-bounds-generic}If $Y$ is any tiled surface and $\overline{B}\in\B_{n}\left(Y\right)$
then
\begin{align*}
\left|\Xi_{n}^{\overline{B}}\right| & \leq\left(\D!\right)^{8}\sum_{\left(\nu,\{\mu_{f}\},\lambda\right)\vdash\overline{B}}\frac{d_{\lambda}}{d_{\nu}^{3}}.
\end{align*}
\end{lem}

\begin{proof}
Since $\M(\{\sigma_{f}^{\pm},\tau_{f}^{\pm},r_{f}^{\pm}s_{f},t_{f}\})$
is a product of matrix coefficients of unit vectors in unitary representations,
we obtain $|\M(\{\sigma_{f}^{\pm},\tau_{f}^{\pm},r_{f}^{\pm}s_{f},t_{f}\})|\leq1$.
Therefore, with assumptions as in the lemma, and arguing similarly
as in the proof of Lemma \ref{lem:xi-B-bounds-small-B_nu}, we obtain
\begin{align*}
\left|\Xi_{n}^{\overline{B}}\right| & \leq\sum_{\left(\nu,\{\mu_{f}\},\lambda\right)\vdash\overline{B}}\frac{d_{\lambda}d_{\nu}}{d_{\mu_{a}}d_{\mu_{b}}d_{\mu_{c}}d_{\mu_{d}}}\sum_{\begin{gathered}r_{f}^{+},r_{f}^{-}\in\Tab\left(\mu_{f}/\nu\right)\\
s_{f},t_{f}\in\Tab\left(\lambda/\mu_{f}\right)
\end{gathered}
}1\\
 & \stackrel{\left(*\right)}{\le}\left(\D!\right)^{8}\sum_{\left(\nu,\{\mu_{f}\},\lambda\right)\vdash\overline{B}}\frac{d_{\lambda}d_{\nu}}{d_{\mu_{a}}d_{\mu_{b}}d_{\mu_{c}}d_{\mu_{d}}}\\
 & \leq(\D!)^{8}\sum_{\left(\nu,\{\mu_{f}\},\lambda\right)\vdash\overline{B}}\frac{d_{\lambda}}{d_{\nu}^{3}},
\end{align*}
where in $\left(*\right)$ we used the fact there are at most $\left|\lambda/\nu\right|!=\left(\v-\f\right)!$
choices of $r_{f}^{+}\sqcup s_{f}$ and of $r_{f}^{-}\sqcup t_{f}$.
\end{proof}

\subsection{The zero regime of $b_{\nu}$\label{subsec:The-zero-regime}}

We only need analytic estimates for $\Xi_{n}^{\nu=(n-\v)}$ when $Y$
is boundary reduced (so $0$-adapted); when $Y$ is $\varepsilon$-adapted
for $\varepsilon>0$ we will take a different, more algebraic approach
to $\Xi_{n}^{\nu=(n-\v)}$ in $\S\S$\ref{subsec:A-new-expression-for_xi-dnu=00003D0}. 
\begin{lem}
\label{lem:zero-regime-BR}If $Y$ is boundary reduced and $\v\leq n^{1/4}$
then
\[
\left|\Xi_{n}^{\nu=(n-\v)}\right|\ll(\D+1)^{9}\D^{34\D}.
\]
\end{lem}

\begin{proof}
If $\nu=(n-\v)$ then $B_{\nu}=0$. Inserting the bounds from Lemma
\ref{lem:xi-B-bounds-small-B_nu} with $\varepsilon=0$ (since $Y$
is boundary reduced, see Lemma \ref{lem:BR and SBR are 0 and ge 0 adapted})
and $B_{\nu}=0$ gives 
\begin{align*}
\left|\Xi_{n}^{\nu=(n-\v)}\right|\ll\sum_{\overline{B}\in\B_{n}\left(Y\right)\colon B_{\nu}=0}B_{\lambda}^{~10B_{\lambda}}\D^{24B_{\lambda}}.
\end{align*}
Because $\overline{B}\in\B_{n}\left(Y\right)$, there exist some $\left(\nu,\left\{ \mu_{f}\right\} ,\lambda\right)\vdash\overline{B}$
and satisfying (\ref{eq:nu-mu-lambda-setup}). We then have since
$\nu\subset_{\v-\f}\lambda$, and $b_{\nu}=0$, $B_{\lambda}=b_{\lambda}\leq b_{\nu}+\v-\f=\v-\f=\D$.
In $\B_{n}\left(Y\right)$, the set of $\overline{B}'s$ with $B_{\nu}=0$
and a fixed value of $B_{\lambda}$ is of size at most $(\D+1)^{9}$.
Indeed, there are at most $B_{\lambda}+1\le\D+1$ options for $B_{\mu_{f}}$
for each $f$. Since $n-\v-1=\check{B}_{\nu}\leq\check{B}_{\mu_{f}}\leq\check{B}_{\lambda}\leq n-\f-1$,
there are at most $\v-\f+1=\D+1$ possible values of each of $\check{B}_{\mu_{f}}$
and $\check{B}_{\lambda}$. In total then there are at most $(\D+1)^{9}$
choices. Hence 
\begin{align*}
\left|\Xi_{n}^{\nu=(n-\v)}\right| & \ll(\D+1)^{9}\sum_{B_{\lambda}=0}^{\D}\left(B_{\lambda}^{~10}\D^{24}\right)^{B_{\lambda}}\le(\D+1)^{9}\sum_{B_{\lambda}=0}^{\D}(\D^{34})^{B_{\lambda}}\ll(\D+1)^{9}\D^{34\D}.
\end{align*}
\end{proof}

\subsection{The intermediate regime of $b_{\nu}$\label{subsec:The-intermediate-regime}}
\begin{lem}
\label{lem:intermediate-regime}Assume that $\v\le n^{1/4}.$
\begin{enumerate}
\item \label{enu:intermediate-BR}If $Y$ is boundary reduced with $\D\leq n^{1/10}$
then 
\begin{align}
\left|\Xi_{n}^{0<b_{\nu}\leq\D;\check{b}_{\nu}>0}\right|\ll\frac{\left(\D^{34}2^{10}\right)^{\D+1}}{\left(n-\v-\D^{2}\right)^{2}}.\label{eq:intermediate-br}
\end{align}
\item \label{enu:intermediate_region_dnu-e>0}For any $\varepsilon\in(0,1)$,
there is $\eta=\eta(\varepsilon)\in(0,\frac{1}{100})$ such that if
$Y$ is $\varepsilon$-adapted, with $\D\leq n^{\eta}$ then 
\begin{align}
\left|\Xi_{n}^{0<b_{\nu}\leq\D;\check{b}_{\nu}>0}\right|\ll_{\varepsilon}\frac{1}{n}.\label{eq:intermediate-epsilon>0}
\end{align}
\end{enumerate}
\end{lem}

\begin{proof}
When $\D=0$, the inequality $0<b_{\nu}\le\D$ cannot hold, and so
$\Xi_{n}^{0<b_{\nu}\leq\D;\check{b}_{\nu}>0}=0$ by definition, and
both statements hold. So assume $\D\geq1$. We can also assume that
$\D\le n^{1/10}$. 

For any $\varepsilon\ge0$, the bounds from Lemma \ref{lem:xi-B-bounds-small-B_nu}
give 
\begin{align*}
\left|\Xi_{n}^{0<b_{\nu}\leq\D;\check{b}_{\nu}>0}\right|\ll\sum_{\substack{\overline{B}\in{\cal B}_{n}\left(Y\right)\colon\\
0<B_{\nu}\leq\D;\check{B}_{\nu}>0
}
}B_{\lambda}^{~10B_{\lambda}}\left(\D^{24}n^{-\varepsilon}\right)^{B_{\lambda}-B_{\nu}}\left(\frac{\kappa\D^{4}}{\left(n-\v-\D^{2}\right)^{2}}\right)^{B_{\nu}}.
\end{align*}
Arguing similarly as in the proof of Lemma \ref{lem:zero-regime-BR},
the number of $\overline{B}$'s in the sum above with a fixed value
of $B_{\nu}$ and $B_{\lambda}$ is $\ll\D^{10}$. Also note that
$B_{\lambda}\le B_{\nu}+\D\le2\D$. We obtain 
\begin{align*}
\left|\Xi_{n}^{0<b_{\nu}\leq\D;\check{b}_{\nu}>0}\right| & \ll\D^{10}\sum_{\substack{0<B_{\nu}\leq\D\\
B_{\nu}\le B_{\lambda}\leq B_{_{\nu}}+\D
}
}B_{\lambda}^{~10B_{\lambda}}\left(\D^{24}n^{-\varepsilon}\right)^{B_{\lambda}-B_{\nu}}\left(\frac{\kappa\D^{4}}{\left(n-\v-\D^{2}\right)^{2}}\right)^{B_{\nu}}\\
 & \leq\D^{10}\sum_{B_{\nu}=1}^{\D}\left(\frac{\kappa(2\D)^{10}\D^{4}}{\left(n-\v-\D^{2}\right)^{2}}\right)^{B_{\nu}}\sum_{B_{\lambda}=B_{\nu}}^{B_{\nu}+\D}\left(\D^{24}B_{\lambda}^{~10}n^{-\varepsilon}\right)^{B_{\lambda}-B_{\nu}}.
\end{align*}
As $B_{\lambda}\le2\D$, we bound the second summation by $\sum_{t=0}^{\D}\left(\D^{34}2^{10}n^{-\varepsilon}\right)^{t}$.
By our assumption that $\D\le n^{1/10}$ and $\v\le n^{1/4}$, we
have $\frac{\kappa(2\D)^{10}\D^{4}}{\left(n-\v-\D^{2}\right)^{2}}\le\frac{1}{2}$
for large enough $n$. Hence
\begin{eqnarray}
\left|\Xi_{n}^{0<b_{\nu}\leq\D;\check{b}_{\nu}>0}\right| & \ll & \frac{\D^{10}\cdot\kappa(2\D)^{10}\D^{4}}{\left(n-\v-\D^{2}\right)^{2}}\sum_{t=0}^{\D}\left(\D^{34}2^{10}n^{-\varepsilon}\right)^{t}\ll\frac{\D^{24}}{\left(n-\v-\D^{2}\right)^{2}}\sum_{t=0}^{\D}\left(\D^{34}2^{10}n^{-\varepsilon}\right)^{t}.~~~~~~~~~~\phantom{}\label{eq:intermediate}
\end{eqnarray}
If $Y$ is boundary reduced, it is $0$-adapted (Lemma \ref{lem:BR and SBR are 0 and ge 0 adapted}),
so (\ref{eq:intermediate}) yields
\begin{eqnarray*}
\left|\Xi_{n}^{0<b_{\nu}\leq\D;\check{b}_{\nu}>0}\right| & \ll & \frac{\D^{24}}{\left(n-\v-\D^{2}\right)^{2}}\cdot\left(\D^{34}2^{10}\right)^{\D}\le\frac{\left(\D^{34}2^{10}\right)^{\D+1}}{\left(n-\v-\D^{2}\right)^{2}}
\end{eqnarray*}
proving the first statement. 

For the second statement, given $\varepsilon>0$, let $\eta=\frac{\varepsilon}{100}$
and assume $1\leq\D\leq n^{\eta}$. The choice of $\eta$ implies
that for $n\gg_{\varepsilon}1$, $\D^{34}2^{10}n^{-\varepsilon}\le\frac{1}{2}$,
so (\ref{eq:intermediate}) gives
\begin{eqnarray*}
\left|\Xi_{n}^{0<b_{\nu}\leq\D;\check{b}_{\nu}>0}\right| & \ll_{\varepsilon} & \frac{\D^{24}}{\left(n-\v-\D^{2}\right)^{2}}\ll_{\varepsilon}\frac{1}{n}.
\end{eqnarray*}
\vspace{-20bp}
\end{proof}

\subsection{The large regime of $b_{\nu},\check{b}_{\nu}$\label{subsec:The-large-regime}}

In the large regime of $b_{\nu}$ and $\check{b}_{\nu}$ we use the
same estimate for any type of tiled surface.
\begin{lem}
\label{lem:large-regime}If $\v\leq n^{1/4}$ and $\D\leq n^{1/24}$
then 
\[
\left|\Xi_{n}^{b_{\nu},\check{b}_{\nu}>\D}\right|\ll\frac{(\D+1)^{4}}{\left(n-\v-\D^{2}\right)^{2}}.
\]
\end{lem}

\begin{proof}
Using the bound from Lemma \ref{lem:Xi-B-bounds-generic} gives
\begin{align*}
\left|\Xi_{n}^{b_{\nu},\check{b}_{\nu}>\D}\right| & \leq\sum_{\overline{B}\in\B_{n}\left(Y\right)\colon~B_{\nu},\check{B}_{\nu}>\D}(\D!)^{8}\sum_{\left(\nu,\{\mu_{f}\},\lambda\right)\vdash\overline{B}}\frac{d_{\lambda}}{d_{\nu}^{3}}\\
 & \leq\left(\D!\right)^{8}\sum_{\nu\vdash n-\v,b_{\nu}>\D,\check{b}_{\nu}>\D}d_{\nu}^{-3}\sum_{\nu\subset_{\v-\f}\lambda}d_{\lambda}\sum_{\nu\subset\mu_{f}\subset_{\e_{f}-\f}\lambda}1\\
 & \leq\left(\D!\right)^{12}\sum_{\nu\vdash n-\v,b_{\nu}>\D,\check{b}_{\nu}>\D}d_{\nu}^{-3}\sum_{\nu\subset_{\v-\f}\lambda}d_{\lambda}~\le~\D^{12\D}\frac{(n-\f)!}{(n-\v)!}\sum_{\nu\vdash n-\v,b_{\nu}>\D,\check{b}_{\nu}>\D}d_{\nu}^{-2}\\
 & \ll\D^{12\D}n^{\D}\left(\frac{\kappa\left(\D+1\right)^{4}}{\left(n-\v-\left(\D+1\right)^{2}\right)^{2}}\right)^{\D+1}=\left(\frac{\kappa n\D^{12}\left(\D+1\right)^{4}}{\left(n-\v-\left(\D+1\right)^{2}\right)^{2}}\right)^{\D}\frac{\kappa\left(\D+1\right)^{4}}{\left(n-\v-\left(\D+1\right)^{2}\right)^{2}}.
\end{align*}
The second-last inequality used Lemma \ref{lem:induced-rep-dimension}
and the final inequality used Proposition \ref{prop:effective-Liebeck-Shalev}.
Since we assume $\D\leq n^{1/24}$ and $\v\le n^{1/4}$ we obtain
the stated result.
\end{proof}

\subsection{Assembly of analytic estimates for $\Xi_{n}$}

Now we combine the estimates obtained in $\S\S$\ref{subsec:The-zero-regime},
\ref{subsec:The-intermediate-regime}, \ref{subsec:The-large-regime}.
First we give the culmination of our previous estimates when $Y$
is boundary reduced.
\begin{prop}
\label{prop:Xi-bound-final-BR}There is $A_{0}>0$ such that if $Y$
is boundary reduced, $\v\leq n^{1/4}$, and $\D\leq n^{1/24}$, then
\[
|\Xi_{n}|\ll(A_{0}\D)^{A_{0}\D}.
\]
\end{prop}

\begin{proof}
With assumptions as in the proposition, splitting $\Xi_{n}$ as in
(\ref{eq:Xin-splitting}) and using Lemmas \ref{lem:zero-regime-BR},
\ref{lem:intermediate-regime}(\ref{enu:intermediate-BR}), and \ref{lem:large-regime}
gives
\[
\left|\Xi_{n}\right|\ll(\D+1)^{9}\D^{34\D}+\frac{\left(\D^{34}2^{10}\right)^{\D+1}}{\left(n-\v-\D^{2}\right)^{2}}+\frac{(\D+1)^{4}}{\left(n-\v-\D^{2}\right)^{2}}.
\]
If $\D=0$ this gives $|\Xi_{n}|\ll1$ which proves the result. If
$1\leq\D\leq n^{1/24}$ we obtain $|\Xi_{n}|\ll(A_{0}\D)^{A_{0}\D}$
as required.
\end{proof}
Next we show that if $Y$ is $\varepsilon$-adapted, then $\D$ can
be as large as a fractional power of $n$ while $\Xi_{n}$ is still
very well approximated by $2\Xi_{n}^{\nu=(n-\v)}$.
\begin{prop}
\label{prop:Xi-bound-final-e-adapted}For any $\varepsilon\in(0,1)$,
there is $\eta=\eta(\varepsilon)\in(0,\frac{1}{100})$ such that if
$Y$ is $\varepsilon$-adapted with $\D\leq n^{\eta}$ and $\v\leq n^{1/4}$,
then
\[
\left|\Xi_{n}-2\Xi_{n}^{\nu=(n-\v)}\right|\ll_{\varepsilon}\frac{1}{n}.
\]
\end{prop}

\begin{proof}
Lemmas \ref{lem:intermediate-regime}(\ref{enu:intermediate_region_dnu-e>0})
and \ref{lem:large-regime} yield that given $\varepsilon\in(0,1)$,
there is $\eta=\eta(\varepsilon)\in(0,\frac{1}{100})$, such that
if $\D\leq n^{\eta}$, $\v\leq n^{1/4}$ and $Y$ is $\varepsilon$-adapted,
then
\begin{align*}
\left|\Xi_{n}-2\Xi_{n}^{\nu=(n-\v)}\right| & =\left|2\Xi_{n}^{0<b_{\nu}\leq\D;\check{b}_{\nu}>0}+\Xi_{n}^{b_{\nu},\check{b}_{\nu}>\D}\right|\\
 & \ll_{\varepsilon}\frac{1}{n}+\frac{(\D+1)^{4}}{\left(n-\v-\D^{2}\right)^{2}}\ll\frac{1}{n}.
\end{align*}
\end{proof}
\begin{rem*}
For general $g$, the condition $\eta(\epsilon)<\frac{1}{100}$ of
Proposition \ref{prop:Xi-bound-final-e-adapted} should be replaced
by $\eta(\epsilon)<\frac{1}{Cg}$ for some universal $C\geq100$.
\end{rem*}

\subsection{A new expression for $\Xi_{n}^{\nu=(n-\protect\v)}$\label{subsec:A-new-expression-for_xi-dnu=00003D0}}

We continue to fix a compact tiled surface $Y$. The goal of this
section is to give a formula for $\Xi_{n}^{\nu=(n-\v)}$ that is more
precise than is possible to obtain with the methods of the previous
section. This will be done by refining the methods of \cite[\S 5]{MPasympcover}. 

We will assume throughout that $n\geq\v$. We fix a bijective map
$\J:Y^{(0)}\to[\v]$, and as in \cite[\S 5]{MPasympcover} for each
$n\in\N$ we modify $\J$ by letting 
\begin{equation}
\J_{n}:Y^{(0)}\to[n-\v+1,n],\quad\J_{n}(v)\eqdf\J(v)+n-\v.\label{eq:shifted-framing}
\end{equation}
We use the map $\J_{n}$ to identify the vertex set of $Y$ with $[n-\v+1,n${]}.
Let $\V_{f}^{-}=\V_{f}^{-}(Y)\subset[n-\v+1,n]$ be the subset of
vertices of $Y$ with outgoing $f$-labeled edges, and $\V_{f}^{+}\subset[n-\v+1,n]$
those vertices of $Y$ with incoming $f$-labeled edges. Note that
$\e_{f}=|\V_{f}^{-}|=|\V_{f}^{+}|$. Recall that $S'_{\v}\le S_{n}$
is the subgroup of permutations fixing $\left[n-\v\right]$ element-wise.
For each $f\in\{a,b,c,d\}$ we fix $g_{f}^{0}\in S'_{\v}$ such that
for every pair of vertices $i,j$ of $Y$ in $[n-\v+1,n]$ with a
directed $f$-labeled edge from $i$ to $j$, we have $g_{f}^{0}(i)=j$.
Note that $g_{f}^{0}(\V_{f}^{-})=\V_{f}^{+}$. We let $g^{0}\eqdf(g_{a}^{0},g_{b}^{0},g_{c}^{0},g_{d}^{0})\in S_{n}^{4}$.
For each $f\in\{a,b,c,d\}$ let $G_{f}$ be the subgroup of $S_{n}$
fixing pointwise $\V_{f}^{-}$. Let $G\eqdf G_{a}\times G_{b}\times G_{c}\times G_{d}\leq S_{n}^{4}$.

Our formula for $\Xi_{n}^{\nu=(n-\v)}$ will involve the size of the
set
\begin{align}
\X_{n}^{*}(Y,\J)\eqdf\left\{ \left(\alpha_{a},\alpha_{b},\alpha_{c},\alpha_{d}\right)\in g^{0}G\,\middle|\,W\left(\alpha_{a},\alpha_{b},\alpha_{c},\alpha_{d}\right)\in S_{n-\v}\right\} \label{eq:X_n^*}
\end{align}
where\footnote{The reason we use this word instead of the relator $[g_{a},g_{b}][g_{c},g_{d}]$
of $\Gamma_{2}$ is the same as in \cite{MPasympcover}: the one-to-one
correspondence between $\X_{n}$ and degree-$n$ covers of a genus
$2$ surface uses the version of the symmetric group where permutations
are multiplied as functions acting from the \emph{right}, whereas
in this section we want to multiply permutations as functions on $[n]$
acting from the \emph{left}.} $W(g_{a},g_{b},g_{c},g_{d})\eqdf g_{d}^{-1}g_{c}^{-1}g_{d}g_{c}g_{b}^{-1}g_{a}^{-1}g_{b}g_{a}$.
Note that a similar set, denoted $\X_{n}(Y,\J)$ in \cite[\S\S 5.2]{MPasympcover},
is the set in which the condition is that $W\left(\alpha_{a},\alpha_{b},\alpha_{c},\alpha_{d}\right)=1$
rather than the identity only when restricted to $\left[n-\v+1,n\right]$,
as in (\ref{eq:X_n^*}). This smaller set $\X_{n}(Y,\J)$ counts the
number of covers $\phi\in\Hom\left(\Gamma_{2},S_{n}\right)$ in which
$\left(Y,\J\right)$ embeds.

The main result of this $\S\S$\ref{subsec:A-new-expression-for_xi-dnu=00003D0}
is the following.
\begin{prop}
\label{prop:Xi_n-refined}With notations as above,
\begin{align*}
\Xi_{n}^{\nu=(n-\v)}=\,\frac{\left(n\right)_{\v}\left|\X_{n}^{*}(Y,\J)\right|}{\left(n\right)_{\f}\prod_{f\in a,b,c,d}(n-\e_{f})!}.
\end{align*}
\end{prop}

Recall that $(n)_{q}$ is the Pochhammer symbol as defined in $\S\S$\ref{subsec:Notation}.
\emph{In the rest of the paper, whenever we write an integral over
a group, it is performed with respect to the uniform measure on the
relevant group. }Let 
\begin{align*}
I & \eqdf\int_{h_{f}\in G_{f}}\int_{\pi\in S_{n-\v}}{\bf 1}\left\{ W\left(g_{a}^{0}h_{a},g_{b}^{0}h_{b},g_{c}^{0}h_{c},g_{d}^{0}h_{d}\right)\pi=1\right\} .
\end{align*}
 The following lemma is immediate as a result of relating sums to
normalized integrals\@.
\begin{lem}
\label{cor:double-coset-integral}We have $\left|\X_{n}^{*}(Y,\J)\right|=\left|S_{n-\v}\right|\cdot\left|G\right|\cdot I$.
\end{lem}

For a Young diagram $\lambda$ of size $m$, we write $\chi_{\lambda}$
for the trace of the irreducible representation of $S_{m}$ on $V^{\lambda}$.
\begin{cor}
\label{cor:size-xn*}We have 
\[
|\X_{n}^{*}(Y,\J)|=\frac{\prod_{f\in a,b,c,d}(n-\e_{f})!}{(n)_{\v}}\sum_{\lambda\vdash n}d_{\lambda}\Theta_{\lambda}^{(n-\v)}(Y,\J)
\]
where 
\begin{equation}
\Theta_{\lambda}^{(n-\v)}(Y,\J)\eqdf\int_{h_{f}\in G_{f}}\int_{\pi\in S_{n-\v}}\chi_{\lambda}\left(W\left(g_{a}^{0}h_{a},g_{b}^{0}h_{b},g_{c}^{0}h_{c},g_{d}^{0}h_{d}\right)\pi\right).\label{eq:Theta-nu}
\end{equation}
\end{cor}

\begin{proof}
Using Schur orthogonality, write 
\begin{align*}
\mathbf{1}\{g=1\} & =\frac{1}{n!}\sum_{\lambda\vdash n}d_{\lambda}\chi_{\lambda}(g),
\end{align*}
hence
\[
I=\frac{1}{n!}\sum_{\lambda\vdash n}d_{\lambda}\Theta_{\lambda}^{(n-\v)}\left(Y,\J\right).
\]
We have $|G|=\prod_{f\in\left\{ a,b,c,d\right\} }(n-\e_{f})!$, hence
by Lemma \ref{cor:double-coset-integral}
\begin{align*}
\left|\X_{n}^{*}\left(Y,\J\right)\right| & =(n-\v)!\prod_{f\in\left\{ a,b,c,d\right\} }\left(n-\e_{f}\right)!\cdot\frac{1}{n!}\sum_{\lambda\vdash n}d_{\lambda}\Theta_{\lambda}^{(n-\v)}\left(Y,\J\right)\\
 & =\frac{\prod_{f\in a,b,c,d}(n-\e_{f})!}{(n)_{\v}}\sum_{\lambda\vdash n}d_{\lambda}\Theta_{\lambda}^{(n-\v)}\left(Y,\J\right).
\end{align*}
\end{proof}
Consider the vector space
\[
W^{\lambda}\eqdf V^{\lambda}\otimes\check{V}^{\lambda}\otimes V^{\lambda}\otimes\check{V}^{\lambda}\otimes V^{\lambda}\otimes\check{V}^{\lambda}\otimes V^{\lambda}\otimes\check{V}^{\lambda}
\]
as a unitary representation of $S_{n}^{8}$. This is a departure from
\cite[\S 5]{MPasympcover} where $W^{\lambda}$ was thought of as
a representation of $S_{n}^{4}$; we take a more flexible setup here.
The reader may find it useful to see \cite[\S\S 5.4]{MPasympcover}
for extra background on representation theory. The inner product on
$V^{\lambda}$ gives an isomorphism $V^{\lambda}\cong\check{V}^{\lambda}$,
$v\mapsto\check{v}$. Let $B_{\lambda}\in\End(W^{\lambda})$ be defined
as in \cite[eq. (5.9)]{MPasympcover} by the formula
\begin{align}
\left\langle B_{\lambda}\left(v_{1}\otimes\check{v}_{2}\otimes v_{3}\otimes\check{v}_{4}\otimes v_{5}\otimes\check{v}_{6}\otimes v_{7}\otimes\check{v}_{8}\right),w_{1}\otimes\check{w}_{2}\otimes w_{3}\otimes\check{w}_{4}\otimes w_{5}\otimes\check{w}_{6}\otimes w_{7}\otimes\check{w}_{8}\right\rangle  & \eqdf\nonumber \\
\langle v_{1},w_{3}\rangle\langle v_{3},v_{2}\rangle\langle w_{2},v_{4}\rangle\langle w_{4},w_{5}\rangle\langle v_{5},w_{7}\rangle\langle v_{7},v_{6}\rangle\langle w_{6},v_{8}\rangle\langle w_{8},w_{1}\rangle.\label{eq:B-lambda-def}
\end{align}

We note the following, extending \cite[Lem.~5.4]{MPasympcover}.
\begin{lem}
\label{lem:B-lambda-property}For any $(g_{1},g_{2},g_{3},g_{4},g_{5},g_{6},g_{7},g_{8})\in S_{n}^{8}$,
we have 
\[
\mathrm{tr}_{W^{\lambda}}(B_{\lambda}\circ(g_{1},g_{2},g_{3},g_{4},g_{5},g_{6},g_{7},g_{8}))=\chi_{\lambda}(g_{8}^{-1}g_{6}^{-1}g_{7}g_{5}g_{4}^{-1}g_{2}^{-1}g_{3}g_{1}).
\]
\end{lem}

\begin{proof}
The proof is a direct calculation directly generalizing \cite[Lem.~5.4]{MPasympcover}.
\end{proof}
Let $Q$ be the orthogonal projection in $W^{\lambda}$ onto the vectors
that are invariant by $G$ acting on $W^{\lambda}$ by the map
\[
(g_{a},g_{b},g_{c},g_{d})\in G\mapsto(g_{a},g_{a},g_{b},g_{b},g_{c},g_{c},g_{d},g_{d})\in S_{n}^{8}.
\]
This projection appeared also in \cite[\S\S 5.4]{MPasympcover}.
\begin{lem}
\label{lem:theta-nu-as-trace-of-B}We have $\Theta_{\lambda}^{(n-\v)}(Y,\J)=\mathrm{tr}_{W^{\lambda}}(\mathfrak{p}B_{\lambda}g^{0}Q)$
where $\mathfrak{p}$ denotes the operator
\[
\mathfrak{p}\eqdf\int_{\pi\in S_{n-\v}}\left(\pi,1,1,1,1,1,1,1\right)\in\End\left(W^{\lambda}\right).
\]
\end{lem}

\begin{rem}
Note that $\mathfrak{p}$ is the projection in $\End(W^{\lambda})$
onto the $\triv$-isotypic subspace for the action of $S_{n-\v}$
on the first factor of $W^{\lambda}$ (while being the identity on
the remaining seven factors). This is a self-adjoint operator.
\end{rem}

\begin{proof}
Recall the definition of $\Theta_{\lambda}^{(n-\v)}(Y,\J)$ in (\ref{eq:Theta-nu}).
Using Lemma \ref{lem:B-lambda-property}, for every set of fixed values
of the $h_{f}$ and $\pi$, we have 
\begin{align*}
 & \chi_{\lambda}\left(W\left(g_{a}^{0}h_{a},g_{b}^{0}h_{b},g_{c}^{0}h_{c},g_{d}^{0}h_{d}\right)\pi\right)=\\
 & ~~~~~~~\mathrm{tr}_{W^{\lambda}}\left(B_{\lambda}\circ\left(g_{a}^{0}h_{a}\pi,g_{a}^{0}h_{a},g_{b}^{0}h_{b},g_{b}^{0}h_{b},g_{c}^{0}h_{c},g_{c}^{0}h_{c},g_{d}^{0}h_{d},g_{d}^{0}h_{d}\right)\right)
\end{align*}
Therefore, 
\begin{align*}
\Theta_{\lambda}^{(n-\v)}(Y,\J) & =\mathrm{tr}_{W^{\lambda}}(B_{\lambda}g^{0}Q\mathfrak{p})=\mathrm{tr}_{W^{\lambda}}(\mathfrak{p}B_{\lambda}g^{0}Q).
\end{align*}
\end{proof}
Using Lemma \ref{lem:theta-nu-as-trace-of-B}, we now find a new expression
for $\Theta_{\lambda}^{(n-\v)}(Y,\J)$ by calculating $\mathrm{tr}_{W^{\lambda}}(\mathfrak{p}B_{\lambda}g^{0}Q)$.
\begin{prop}
\label{prop:theta-nu-expression}We have
\begin{eqnarray}
\Theta_{\lambda}^{(n-\v)}\left(Y,\J\right) & = & \sum_{(n-\v)\subset\mu_{f}\subset_{\e_{f}-\f}\lambda'\subset_{\f}\lambda}\frac{d_{\lambda/\lambda'}}{d_{\mu_{a}}d_{\mu_{b}}d_{\mu_{c}}d_{\mu_{d}}}\Upsilon_{n}\left(\left\{ \sigma_{f}^{\pm},\tau_{f}^{\pm}\right\} ,(n-\v),\left\{ \mu_{f}\right\} ,\lambda'\right).\label{eq:theta-nu-expression}
\end{eqnarray}
\end{prop}

\begin{proof}
This calculation is very similar to the proof of \cite[Prop. 5.8]{MPasympcover}
where $\mathrm{tr}_{W^{\lambda}}(B_{\lambda}g^{0}Q)$ was calculated.
The only difference here is the presence of the additional operator
$\mathfrak{p}$. Therefore we will not give all the details. The proof
follows \cite[proof of Prop. 5.8]{MPasympcover} using properties
\textbf{P1}-\textbf{P4 }of $\sigma_{f}^{\pm},\tau_{f}^{\pm}$. One
also uses that $\mathfrak{p}$ is a self-adjoint projection. The role
that $\text{\ensuremath{\mathfrak{p}}}$ plays in the proof is that
instead of obtaining a summation over all $\nu\subset_{\v}\lambda$,
the projection $\mathfrak{p}$ forces only the relevant $\nu=(n-\v)$
to appear.

Indeed, the calculation leading to \cite[eq. (5.17)]{MPasympcover}
is replaced by 
\begin{align*}
 & \left\langle \mathfrak{p}B_{\lambda}\left[\EE_{\mu_{a},S_{a},T_{a}}^{\lambda,a,+}\otimes\EE_{\mu_{b},S_{b},T_{b}}^{\lambda,b,+}\otimes\EE_{\mu_{c},S_{c},T_{c}}^{\lambda,c,+}\otimes\EE_{\mu_{d},S_{d},T_{d}}^{\lambda,d,+}\right],\EE_{\mu{}_{a},S_{a},T_{a}}^{\lambda,a,-}\otimes\EE_{\mu{}_{b},S_{b},T_{b}}^{\lambda,b,-}\otimes\EE_{\mu{}_{c},S_{c},T_{c}}^{\lambda,c,-}\otimes\EE_{\mu{}_{d},S_{d},T_{d}}^{\lambda,d,-}\right\rangle \\
= & \left\langle B_{\lambda}\left[\EE_{\mu_{a},S_{a},T_{a}}^{\lambda,a,+}\otimes\EE_{\mu_{b},S_{b},T_{b}}^{\lambda,b,+}\otimes\EE_{\mu_{c},S_{c},T_{c}}^{\lambda,c,+}\otimes\EE_{\mu_{d},S_{d},T_{d}}^{\lambda,d,+}\right],\mathfrak{p}\left(\EE_{\mu{}_{a},S_{a},T_{a}}^{\lambda,a,-}\otimes\EE_{\mu{}_{b},S_{b},T_{b}}^{\lambda,b,-}\otimes\EE_{\mu{}_{c},S_{c},T_{c}}^{\lambda,c,-}\otimes\EE_{\mu{}_{d},S_{d},T_{d}}^{\lambda,d,-}\right)\right\rangle \\
= & \frac{1}{d_{\mu_{a}}d_{\mu_{b}}d_{\mu_{c}}d_{\mu_{d}}}\sum_{R_{f}^{\pm}\in\Tab\left(\mu_{f}\right)}\left\langle v_{R_{a}^{+}\sqcup S_{a}}^{\sigma_{a}^{+}},v_{R_{b}^{-}\sqcup S_{b}}^{\sigma_{b}^{-}}\right\rangle \left\langle v_{R_{b}^{+}\sqcup S_{b}}^{\sigma_{b}^{+}},v_{R_{a}^{+}\sqcup T_{a}}^{\tau_{a}^{+}}\right\rangle \left\langle v_{R_{a}^{-}\sqcup T_{a}}^{\tau_{a}^{-}},v_{R_{b}^{+}\sqcup T_{b}}^{\tau_{b}^{+}}\right\rangle \cdot\\
 & ~~~~~\left\langle v_{R_{b}^{-}\sqcup T_{b}}^{\tau_{b}^{-}},v_{R_{c}^{-}\sqcup S_{c}}^{\sigma_{c}^{-}}\right\rangle \left\langle v_{R_{c}^{+}\sqcup S_{c}}^{\sigma_{c}^{+}},v_{R_{d}^{-}\sqcup S_{d}}^{\sigma_{d}^{-}}\right\rangle \left\langle v_{R_{d}^{+}\sqcup S_{d}}^{\sigma_{d}^{+}},v_{R_{c}^{+}\sqcup T_{c}}^{\tau_{c}^{+}}\right\rangle \left\langle v_{R_{c}^{-}\sqcup T_{c}}^{\tau_{c}^{-}},v_{R_{d}^{+}\sqcup T_{d}}^{\tau_{d}^{+}}\right\rangle \left\langle v_{R_{d}^{-}\sqcup T_{d}}^{\tau_{d}^{-}},\mathfrak{p}_{0}v_{R_{a}^{-}\sqcup S_{a}}^{\sigma_{a}^{-}}\right\rangle 
\end{align*}
where $\mathfrak{p}_{0}$ is orthogonal projection to the $S_{n-\v}$-invariant
vectors in $V^{\lambda}$. Then the same discussion as precedes \cite[eq. (5.17)]{MPasympcover}
applies now to show that the above is zero unless there is $\nu\vdash n-\v$
such that $\nu\subset\mu_{f}$ for all $f\in\{a,b,c,d\}$, and all
$R_{f}^{+}\lvert_{\le n-\v}$, $R_{f}^{-}\lvert_{\le n-\v}$ are equal
and of shape $\nu$, except now, the presence of $\mathfrak{\mathfrak{p}}_{0}$
forces $\nu=(n-\v)$. Then the rest of the proof is the same.
\end{proof}

\begin{proof}[Proof of Proposition \ref{prop:Xi_n-refined}]
Combining Corollary \ref{cor:size-xn*} and Proposition \ref{prop:theta-nu-expression}
we obtain
\begin{eqnarray*}
 &  & \left|\X_{n}^{*}(Y,\J)\right|\\
 & = & \frac{\prod_{f\in\left\{ a,b,c,d\right\} }(n-\e_{f})!}{(n)_{\v}}\sum_{\lambda\vdash n}d_{\lambda}\sum_{(n-\v)\subset\mu_{f}\subset_{\e_{f}-\f}\lambda'\subset_{\f}\lambda}\frac{d_{\lambda/\lambda'}}{d_{\mu_{a}}d_{\mu_{b}}d_{\mu_{c}}d_{\mu_{d}}}\Upsilon_{n}\left(\left\{ \sigma_{f}^{\pm},\tau_{f}^{\pm}\right\} ,(n-\v),\left\{ \mu_{f}\right\} ,\lambda'\right)\\
 & = & \frac{\prod_{f\in a,b,c,d}\left(n-\e_{f}\right)!(n)_{\f}}{(n)_{\v}}\sum_{(n-\v)\subset\mu_{f}\subset_{\e_{f}-\f}\lambda'\vdash n-\f}\frac{d_{\lambda'}}{d_{\mu_{a}}d_{\mu_{b}}d_{\mu_{c}}d_{\mu_{d}}}\Upsilon_{n}\left(\left\{ \sigma_{f}^{\pm},\tau_{f}^{\pm}\right\} ,(n-\v),\left\{ \mu_{f}\right\} ,\lambda'\right)\\
 & = & \frac{\prod_{f\in a,b,c,d}\left(n-\e_{f}\right)!(n)_{\f}}{(n)_{\v}}\Xi_{n}^{\nu=(n-\v)},
\end{eqnarray*}
where the second equality used Lemma \ref{lem:induced-rep-dimension}
and the third used $d_{(n-\v)}=1$. This gives the result.
\end{proof}

\subsection{Understanding $\left|\protect\X_{n}^{*}(Y,\protect\J)\right|$\label{subsec:UnderstandingX*_n}}

Recall the definition of $\X_{n}^{*}(Y,\J)$ in (\ref{eq:X_n^*}).
Because these 4-tuples of permutations generally do not correspond
to covers of the surface $\Sigma_{2}$, they are better analyzed as
$n$-degree covers of the bouquet of four loops, namely, as graphs
on $n$ vertices labeled by $\left[n\right]$ with directed edges
labeled by $a,b,c,d$, and exactly one incoming $f$-edge and one
outgoing $f$-edge in every vertex and every $f\in\left\{ a,b,c,d\right\} $.
Equivalently, these graphs are the Schreier graphs depicting the action
of $S_{n}$ on $\left[n\right]$ with respect to the four permutations
$\alpha_{a},\alpha_{b},\alpha_{c},\alpha_{d}$.

Such a Schreier graph $\G$ corresponds to some 4-tuple $\left(\alpha_{a},\alpha_{b},\alpha_{c},\alpha_{d}\right)\in\X_{n}^{*}(Y,\J)$
if and only if the following two conditions are satisfied. The assumption
that $\left(\alpha_{a},\alpha_{b},\alpha_{c},\alpha_{d}\right)\in g^{0}G$
means that $Y^{\left(1\right)}$, the 1-skeleton of $Y$, is embedded
in $\G$, in an embedding that extends $\J_{n}$ on the vertices.
The condition that $W\left(\alpha_{a},\alpha_{b},\alpha_{c},\alpha_{d}\right)\in S_{n-\v}$,
means that at every vertex of $\G$ with label in $[n-\v+1,n]$, there
is a closed path of length 8 that spells out the word $[a,b][c,d]$.

In Lemma \ref{lem:x_n^*-rational} below we show that the number of
such graphs (equal to $\left|\X_{n}^{*}\left(Y,\J\right)\right|$)
is rational in $n$. To this end, we apply techniques based on Stallings
core graphs, in a similar fashion to the techniques applied in \cite{puder2014primitive,PP15}.

Construct a finite graph $\hat{Y}$ as follows. Start with $Y^{\left(1\right)}$,
the $1$-skeleton of $Y$. At every vertex attach a closed cycle of
length $8$ spelling out $\left[a,b\right]\left[c,d\right]$. Then
fold the resulting graph, in the sense of Stallings\footnote{Folding a graph with directed and labeled edges means that as long
as there is a vertex with two incoming edges with the same label,
or two outgoing edges with the same label, these two edges are merged,
and so are their other endpoints. It is well known that this process
has a unique outcome \cite[\S 3]{stallings1983topology}.}, to obtain $\hat{Y}$. In other words, at each vertex $v$ of $Y^{\left(1\right)}$,
if there is a closed path at $v$ spelling $\left[a,b\right]\left[c,d\right]$,
do nothing. Otherwise, find the largest prefix of $\left[a,b\right]\left[c,d\right]$
that can be read on a path $p$ starting at $v$ and the largest suffix
of $\left[a,b\right]\left[c,d\right]$ that can be read on a path
$s$ terminating at $v$. Because $Y$ is a tiled surface, $\left|p\right|+\left|s\right|<8$.
Attach a path of length $8-\left|p\right|-\left|s\right|$ between
the endpoint of $p$ and the beginning of $s$ which spells out the
missing part of the word $\left[a,b\right]\left[c,d\right]$. In this
description, no folding is required. Note, in particular, that $Y^{\left(1\right)}$
is embedded in $\hat{Y}$. 

\begin{figure}
\begin{centering}
\includegraphics[viewport=0bp 12.00608bp 451bp 158bp]{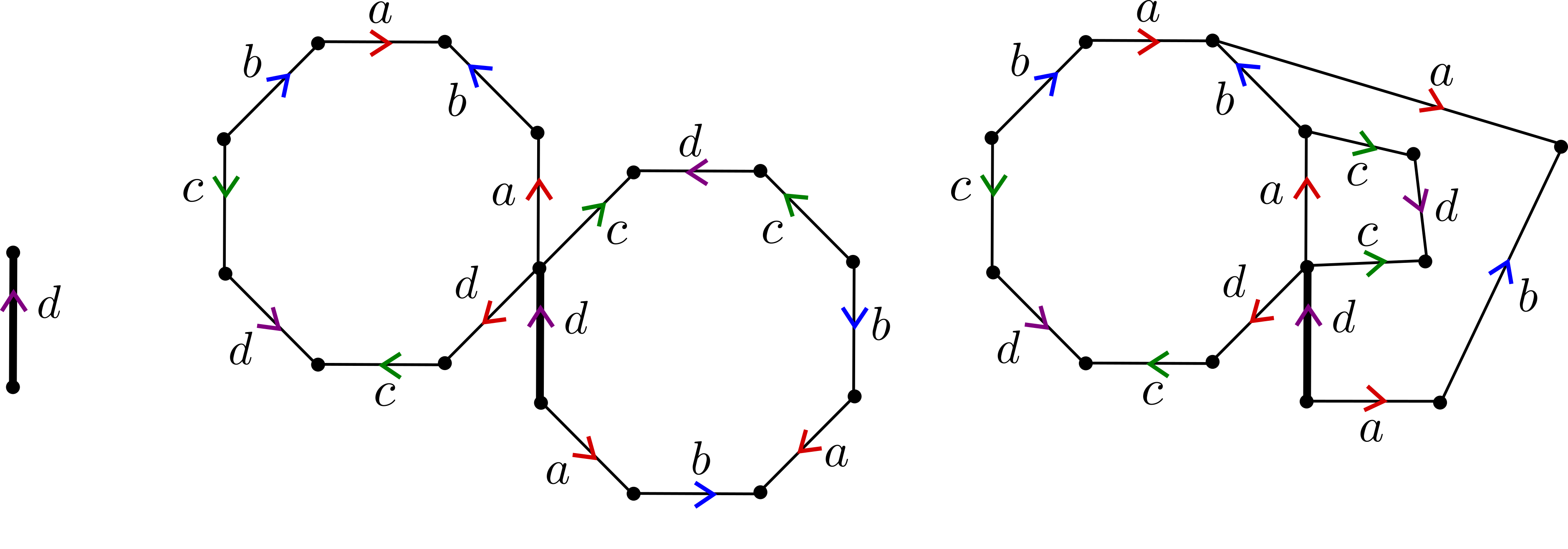}
\par\end{centering}
\caption{\label{fig:Yhat}On the left is a tiled surface $Y$ consisting of
two vertices and a single $d$-edge between them. The middle part
shows $\hat{Y}$: where we \textquotedblleft grow\textquotedblright{}
an octagon from every vertex of $Y$, and in which $Y^{\left(1\right)}$
is embedded. The figure on the right shows another element of ${\cal Q}\left(Y\right)$:
a folded quotient of $\hat{Y}$ where $Y^{\left(1\right)}$ is still
embedded.}
\end{figure}

By the discussion above, the Schreier graphs $\G$ corresponding to
$\X_{n}^{*}\left(Y,\J\right)$ are the graphs in which there is an
embedding of $Y^{\left(1\right)}$ which \emph{extends to a morphism
of directed edge-labeled graphs} of $\hat{Y}$. We group these $\G$
according to the image of $\hat{Y}$. So denote by ${\cal Q}\left(Y\right)$
the possible images of $\hat{Y}$ in the graphs $\G$: these are precisely
the folded quotients of $\hat{Y}$ (edges can only be merged with
equally-labeled other edges) which restrict to a bijection on $Y^{\left(1\right)}$.
In particular, $\hat{Y}\in{\cal Q}\left(Y\right)$. We illustrate
these concepts in Figure \ref{fig:Yhat}. As $\hat{Y}$ is a finite
graph, the set ${\cal Q}\left(Y\right)$ is finite.
\begin{lem}
\label{lem:x_n^*-rational}For every $n\ge8\v\left(Y\right)$,
\begin{equation}
\left|\X_{n}^{*}(Y,\J)\right|=\frac{\left(n!\right)^{4}}{(n)_{\v(Y)}}\sum_{H\in{\cal Q}\left(Y\right)}\frac{\left(n\right)_{\v\left(H\right)}}{\prod_{f\in\left\{ a,b,c,d\right\} }\left(n\right)_{\e_{f}(H)}}.\label{eq:rational expression for X*}
\end{equation}
\end{lem}

\begin{proof}
By the discussion above it is enough to show that for every $H\in{\cal Q}\left(Y\right)$
and $n\ge8\v\left(Y\right)$, the number of Schreier graphs $\G$
on $n$ vertices where the image of $\hat{Y}$ is $H$, is precisely
\[
\frac{(n!)^{4}}{\left(n\right)_{\v\left(Y\right)}}\cdot\frac{\left(n\right)_{\v\left(H\right)}}{\prod_{f\in\left\{ a,b,c,d\right\} }\left(n\right)_{\e_{f}(H)}}.
\]
First, note that $\v\left(H\right)\le\v\left(\hat{Y}\right)\le8\v\left(Y\right)$,
so under the assumption that $n\ge8\v\left(Y\right)$, $H$ can indeed
be embedded in Schreier graphs on $n$ vertices. The number of possible
labelings of the vertices of $H$, which must extend the labeling
of the vertices of $Y^{\left(1\right)}$, is 
\[
\left(n-\v\left(Y\right)\right)\left(n-\v\left(Y\right)-1\right)\cdots\left(n-\v\left(H\right)+1\right)=\frac{\left(n\right)_{\v\left(H\right)}}{\left(n\right)_{\v\left(Y\right)}}.
\]
There are exactly $\e_{a}\left(H\right)$ constraints on the permutation
$\alpha_{a}$ for it to agree with the data in the vertex-labeled
$H$, so there are $\left(n-\e_{a}\left(H\right)\right)!=\frac{n!}{\left(n\right)_{\e_{a}\left(H\right)}}$
such permutations. The same logic applied to the other letters gives
the required result. 
\end{proof}
Combining Lemma \ref{lem:x_n^*-rational} with Proposition \ref{prop:Xi_n-refined}
gives the following corollary. 
\begin{cor}
\label{cor:rational-form-Xi-dnu}For $n\geq8\v(Y)$ we have 
\begin{eqnarray*}
\Xi_{n}^{\nu=(n-\v)}(Y) & = & \frac{\prod_{f\in\left\{ a,b,c,d\right\} }(n)_{\e_{f}(Y)}}{(n)_{\f(Y)}}\sum_{H\in{\cal Q}\left(Y\right)}\frac{\left(n\right)_{\v(H)}}{\prod_{f\in\left\{ a,b,c,d\right\} }(n)_{\e_{f}(H)}}.
\end{eqnarray*}
In particular, if $Y$ is fixed and $n\to\infty,$ we have

\begin{equation}
\Xi_{n}^{\nu=(n-\v)}(Y)=\sum_{H\in{\cal Q}\left(Y\right)}n^{\e\left(Y\right)-\f\left(Y\right)+\chi\left(H\right)}\left(1+O_{Y}\left(\frac{1}{n}\right)\right).\label{eq:EC of H in expression for Chi-dv=00003D1}
\end{equation}
\end{cor}

\begin{proof}
The first statement follows directly from Lemma \ref{lem:x_n^*-rational}
and Proposition \ref{prop:Xi_n-refined}. To obtain the second statement
from the first, we use that all Pochammer symbols $(n)_{q}$ appearing
therein have $q$ bounded depending on $Y$ and hence $(n)_{q}=n^{q}+O_{Y}(n^{q-1})$.
\end{proof}
Note that in the construction of $\hat{Y}$ from $Y^{\left(1\right)}$,
we add a ``handle'' (a sequence of edges) to the graph for every
vertex of $Y$ that does not admit a closed cycle spelling $\left[a,b\right]\left[c,d\right]$.
Hence the Euler characteristic of $\hat{Y}$ is equal to that of $Y^{\left(1\right)}$
minus the number of such vertices in $Y$. If $Y$ has an octagon
attached along every closed cycle spelling $\left[a,b\right]\left[c,d\right]$,
there are $\v\left(Y\right)-\f\left(Y\right)$ such vertices, so 
\begin{equation}
\chi\left(\hat{Y}\right)=\chi\left(Y^{\left(1\right)}\right)-\left(\v\left(Y\right)-\f\left(Y\right)\right)=\f\left(Y\right)-\e\left(Y\right).\label{eq:EC of H hat}
\end{equation}
In particular, this is the case when $Y$ is (strongly) boundary reduced.
This is important because of the role of $\chi\left(H\right)$ in
(\ref{eq:EC of H in expression for Chi-dv=00003D1}) for $H\in{\cal Q}(Y)$.
It turns out that when $Y$ is strongly boundary reduced, $\hat{Y}$
has Euler characteristic strictly larger than all other graphs in
${\cal Q}\left(Y\right)$:
\begin{lem}
\label{lem:x-max-f-e}If $Y$ is strongly boundary reduced, then for
every $H\in{\cal Q}(Y)\setminus\{\hat{Y}\}$,
\[
\chi\left(H\right)<\chi\left(\hat{Y}\right).
\]
\end{lem}

\begin{proof}
We use \cite[Prop.~5.26]{MPasympcover} that states that if $Y$ is
strongly boundary reduced, then as $n\to\infty$, 
\begin{equation}
\Xi_{n}(Y)=2+O_{Y}\left(n^{-1}\right).\label{eq:chi-max-temp1}
\end{equation}
When $Y$ is fixed and $n\to\infty$, it follows from Lemmas \ref{lem:intermediate-regime}(\ref{enu:intermediate-BR})
and \ref{lem:large-regime} that 
\begin{equation}
\Xi_{n}(Y)=2\Xi_{n}^{(n-\v)}(Y)+O_{Y}\left(n^{-2}\right).\label{eq:xmax-temp2}
\end{equation}
Combining (\ref{eq:chi-max-temp1}) and (\ref{eq:xmax-temp2}) gives
\begin{equation}
\Xi_{n}^{(n-\v)}(Y)=1+O_{Y}\left(n^{-1}\right).\label{eq:dnu=00003D0-asymptotic}
\end{equation}
Comparing (\ref{eq:dnu=00003D0-asymptotic}) with (\ref{eq:EC of H in expression for Chi-dv=00003D1})
shows that there is exactly one $H\in{\cal Q}(Y$) with $\chi\left(H\right)=\f(Y)-\e(Y)$,
and all remaining graphs in ${\cal Q}(Y)$ have strictly smaller Euler
characteristic. Finally, (\ref{eq:EC of H hat}) shows this $H$ must
be $\hat{Y}$ itself.
\end{proof}

\subsection{Bounds on $\protect\E_{n}^{\protect\emb}(Y)$ for $\varepsilon$-adapted
$Y$}

In this section we give the final implications of the previous sections
for $\E_{n}^{\emb}(Y)$ for $\varepsilon$-adapted $Y$. Recall the
definition of ${\cal Q}\left(Y\right)$ from $\S\S$\ref{subsec:UnderstandingX*_n}.
We will need the following easy bound for Pochhammer symbols.
\begin{lem}
\label{lem:Pochhammer-bounds}Let $n\in\N$ and $q\in\N\cup\{0\}$
with $q\leq\frac{1}{2}n$. Then 
\[
n^{q}\left(1-\frac{q^{2}}{n}\right)\leq n^{q}\exp\left(\frac{-q^{2}}{n}\right)\leq(n)_{q}\leq n^{q}.
\]
\end{lem}

\begin{proof}
The first inequality is based on $1-x\le e^{-x}$. The second one
is based on writing $\left(n\right)_{q}=n^{q}\left(1-\frac{1}{n}\right)\cdots\left(1-\frac{q-1}{n}\right)$
and using $e^{-2x}\le1-x$ which holds for $x\in\left[0,\frac{1}{2}\right]$.
The third inequality is obvious.
\end{proof}
\begin{prop}
\label{prop:e-adapted-E_n_asympt}Let $\varepsilon\in(0,1)$ and $\eta=\eta(\varepsilon)\in(0,\frac{1}{100})$
be the parameter provided by Proposition \ref{prop:Xi-bound-final-e-adapted}
for this $\varepsilon$. Let $n\in\N$ and $M=M\left(n\right)\ge1$.
Let $Y$ be $\varepsilon$-adapted with $\D(Y)\leq n^{\eta}$ and
$\v(Y),\e(Y),\f(Y)\leq M\leq n^{1/4}$. Then 
\begin{equation}
\frac{\E_{n}^{\emb}(Y)}{n^{\chi(Y)}}=\left(1+O_{\varepsilon}\left(\frac{M^{2}}{n}\right)\right)\left(1+\sum_{H\in{\cal Q}\left(Y\right)\backslash\{\hat{Y}\}}n^{\chi(H)+\e(Y)-\f(Y)}\right).\label{eq:E_n-asmpy}
\end{equation}
\end{prop}

\begin{proof}
Assume all parameters are as in the statement of the proposition.
By Theorem \ref{thm:E_n-emb-exact-expression} and Proposition \ref{prop:Xi-bound-final-e-adapted}
we have
\[
\frac{\E_{n}^{\emb}(Y)}{n^{\chi(Y)}}=\frac{(n!)^{3}}{|\X_{n}|}\cdot\frac{(n)_{\v(Y)}(n)_{\f(Y)}}{\prod_{f}(n)_{\e_{f}(Y)}n^{\chi(Y)}}\left[2\Xi_{n}^{\nu=(n-\v)}\left(Y\right)+O_{\varepsilon}\left(\frac{1}{n}\right)\right].
\]
By Lemma \ref{lem:Pochhammer-bounds}, $\frac{(n)_{\v(Y)}(n)_{\f(Y)}}{\prod_{f}(n)_{\e_{f}(Y)}n^{\chi(Y)}}=1+O\left(\frac{M^{2}}{n}\right)$.
By Corollary \ref{cor:size-of-X_n}, $\frac{\left(n!\right)^{3}}{\left|\X_{n}\right|}=\frac{1}{2}+O\left(\frac{1}{n^{2}}\right)$.
With Corollary \ref{cor:rational-form-Xi-dnu}, this gives
\begin{eqnarray}
\frac{\E_{n}^{\emb}(Y)}{n^{\chi(Y)}} & = & \left[\frac{1}{2}+O\left(\frac{M^{2}}{n}\right)\right]\left[2\frac{\prod_{f}\left(n\right)_{\e_{f}\left(Y\right)}}{\left(n\right)_{\f\left(Y\right)}}\sum_{H\in{\cal Q}\left(Y\right)}\frac{\left(n\right)_{\v\left(H\right)}}{\prod_{f}\left(n\right)_{\e_{f}\left(H\right)}}\right]+O_{\varepsilon}\left(\frac{1}{n}\right)\nonumber \\
 & \stackrel{\text{Lem. \ref{lem:Pochhammer-bounds}}}{=} & \left[1+O\left(\frac{M^{2}}{n}\right)\right]\sum_{H\in{\cal Q}\left(Y\right)}n^{\e(Y)-\f(Y)+\chi(H)}+O_{\varepsilon}\left(\frac{1}{n}\right),\label{eq:n-vs-m-temp}
\end{eqnarray}
where the use of Lemma \ref{lem:Pochhammer-bounds} is justified since
for every $H\in{\cal Q}\left(Y\right)$, $\v\left(H\right)\le\v(\hat{Y})\le8\v\left(Y\right)\le8M$,
and $\e\left(H\right)\le\e(\hat{Y})\le\e\left(Y\right)+8\v\left(Y\right)\le9M$.
In the summation in (\ref{eq:n-vs-m-temp}), the top power of $n$
is realized by $\hat{Y}$ and is equal to zero (by (\ref{eq:EC of H hat})
and Lemma \ref{lem:x-max-f-e}), so we obtain
\begin{align*}
\frac{\E_{n}^{\emb}(Y)}{n^{\chi(Y)}} & =\left[1+O\left(\frac{M^{2}}{n}\right)\right]\left(1+\sum_{H\in{\cal Q}\left(Y\right)\backslash\{\hat{Y}\}}n^{\chi(H)+\e(Y)-\f(Y)}\right)+O_{\varepsilon}\left(\frac{1}{n}\right),
\end{align*}
which yields (\ref{eq:E_n-asmpy}).
\end{proof}
The drawback of Proposition \ref{prop:e-adapted-E_n_asympt} is that
we do not know how to directly estimate the sum over $H\in{\cal Q}\left(Y\right)\backslash\{\hat{Y}\}$
that appears therein. Because we can not directly deal with this sum,
we instead use Proposition \ref{prop:e-adapted-E_n_asympt} to deduce
in the remaining results of this section that for $\varepsilon$-adapted
$Y$ we can control $\E_{n}^{\emb}(Y)$ using $\E_{m}^{\emb}(Y)$
with $m$ much smaller than $n$.
\begin{cor}
\label{cor:Em-lower-bound}Let $\varepsilon\in(0,1)$, and $\eta=\eta(\varepsilon)\in(0,\frac{1}{100})$
be the parameter provided by Proposition \ref{prop:Xi-bound-final-e-adapted}
for this $\varepsilon$. Let $m\in\N$. Let $Y$ be $\varepsilon$-adapted
with $\D(Y)\leq m^{\eta}$ and $\v(Y),\e(Y),\f(Y)\leq m^{1/4}$. Then
\[
\frac{\E_{m}^{\emb}(Y)}{m^{\chi\left(Y\right)}}\gg_{\varepsilon}1+\sum_{H\in{\cal Q}\left(Y\right)\backslash\{\hat{Y}\}}m^{\chi(H)+\e(Y)-\f(Y)}.
\]
In particular, $\E_{m}^{\emb}(Y)\gg_{\epsilon}m^{\chi(Y)}.$ 
\end{cor}

\begin{rem}
\label{rem:extra-remark}As a direct consequence of Corollary \ref{cor:Em-lower-bound},
under the same conditions, if $\chi(Y)<0$ and $n\geq m$ we obtain
\[
\frac{m}{n}\E_{m}^{\emb}(Y)\gg\frac{m^{\chi(Y)+1}}{n}\gg n^{\chi(Y)}.
\]
\end{rem}

While Corollary \ref{cor:Em-lower-bound} is a direct consequence
of Proposition \ref{prop:e-adapted-E_n_asympt}, we can get more information
by combining Proposition \ref{prop:e-adapted-E_n_asympt} with what
we already know about $\mathcal{Q}(Y)$. 
\begin{prop}
\label{prop:main-term-asymp}Let $\varepsilon\in(0,1)$, $\eta$ be
as in Proposition \ref{prop:Xi-bound-final-e-adapted} and $K>1$.
Let $n\in\N$ and $m=m\left(n\right)\in\N$ with $m<n$ and $m\stackrel{n\to\infty}{\to}\infty$.
Let $Y$ be $\varepsilon$-adapted and suppose that $\v(Y),\e(Y),\f(Y)\leq(K\log n)^{2}\le m^{1/4}$
and that $\D(Y)\leq K\log n\le m^{\eta}$. Then 
\begin{equation}
\frac{\E_{n}^{\emb}(Y)}{n^{\chi(Y)}}=1+O_{\varepsilon,K}\left(\frac{(\log n)^{4}}{n}\right)+O_{\varepsilon,K}\left(\frac{m}{n}\frac{\E_{m}^{\emb}(Y)}{m^{\chi(Y)}}\right).\label{eq:main-term-asymp}
\end{equation}
\end{prop}

\begin{proof}
With assumptions as in the proposition, Proposition \ref{prop:e-adapted-E_n_asympt}
gives
\begin{align*}
\frac{\E_{n}^{\emb}(Y)}{n^{\chi(Y)}} & =\left(1+O_{\varepsilon,K}\left(\frac{(\log n)^{4}}{n}\right)\right)\left(1+\sum_{H\in{\cal Q}\left(Y\right)\backslash\{\hat{Y}\}}n^{\chi(H)+\e(Y)-\f(Y)}\right)\\
 & =1+O_{\varepsilon,K}\left(\frac{(\log n)^{4}}{n}\right)+O_{\varepsilon,K}\left(\sum_{H\in{\cal Q}\left(Y\right)\backslash\{\hat{Y}\}}n^{\chi(H)+\e(Y)-\f(Y)}\right).
\end{align*}
Finally, because for every $H\in{\cal Q}\left(Y\right)\setminus\{\hat{Y}\}$
we have $\chi\left(H\right)+\e\left(Y\right)-\f\left(Y\right)\le-1$
and $m<n$,
\begin{eqnarray*}
\sum_{H\in{\cal Q}\left(Y\right)\backslash\{\hat{Y}\}}n^{\chi(H)+\e(Y)-\f(Y)} & = & \sum_{H\in{\cal Q}\left(Y\right)\backslash\{\hat{Y}\}}\left(\frac{n}{m}\right)^{\chi(H)+\e(Y)-\f(Y)}m^{\chi(H)+\e(Y)-\f(Y)}\\
 & \le & \frac{m}{n}\sum_{H\in{\cal Q}\left(Y\right)\backslash\{\hat{Y}\}}m^{\chi(H)+\e(Y)-\f(Y)}\stackrel{\text{Cor. \ref{cor:Em-lower-bound}}}{\ll_{\varepsilon}}\frac{m}{n}\frac{\E_{m}^{\emb}(Y)}{m^{\chi(Y)}},
\end{eqnarray*}
concluding the proof of the proposition.
\end{proof}

\section{Proof of Theorem \ref{thm:effective-error}\label{sec:Proof-of-Effective-error-theorem}}

The reader is suggested to have read the overview in $\S\S$\ref{subsec:Overview-of-the-paper}
before attempting to read this section of the paper.

\subsection{Setup}

We remind the reader that $g=2$. We are given $c>0$, and an element
$\gamma\in\Gamma$ of cyclic word length $\ell_{w}(\gamma)\leq c\log n$.
We assume that $\gamma$ is not a proper power of another element
of $\Gamma$. We remind the reader that $\mathcal{C}_{\gamma}$ is
an annular tiled surface associated to $\gamma$ as in Example \ref{exa:the-loop-of-a-word}.
By Lemma \ref{lem:fix-gamma-fix-cycle},
\[
\E_{n}\left[\fix_{\gamma}\right]=\E_{n}\left(\mathcal{C}_{\gamma}\right),
\]
where $\E_{n}(\CC_{\gamma})$ is the expected number of morphisms
from $\CC_{\gamma}$ to the random surface $X_{\phi}$. Let $\varepsilon=\frac{1}{32}$
(for general $g$, $\varepsilon=\frac{1}{16g}$) and let $\mathcal{R}_{\varepsilon}(\CC_{\gamma})$
be the finite resolution of $\mathcal{C_{\gamma}}$ provided by Definition
\ref{def:epsilon-resolution} and Theorem \ref{thm:certification of resolution}.
Each element of this resolution is a morphism $h:\CC_{\gamma}\to W_{h}$
where $W_{h}$ is a tiled surface. By Lemma \ref{lem:resolution-sum-of-expectations}
we have for any $n\geq1$
\begin{equation}
\E_{n}\left[\fix_{\gamma}\right]=\sum_{h\in\mathcal{R_{\varepsilon}}(\CC_{\gamma})}\mathbb{E}_{n}^{\emb}\left(W_{h}\right),\label{eq:main-proof-sum-of-expectations}
\end{equation}
where $\mathbb{E}_{n}^{\emb}\left(W_{h}\right)$ is the expected number
of embeddings of $W_{h}$ into the random tiled surface $X_{\phi}$.
Associated to each $W_{h}$ here, $\v(W_{h}),\e(W_{h}),$ and $\f(W_{h})$
are the number of vertices, edges, and faces of $W_{h}$. Also associated
to $W_{h}$ are $\d(W_{h})$, the number of edges in the boundary
of $W_{h}$, $\chi(W_{h})$, the topological Euler characteristic
of $W_{h}$, and $\D(W_{h})=\v(W_{h})-\f(W_{h})$.

By Corollary \ref{cor:facts-about-the-resolution}, there is a constant
$K=K(c)>0$, such that for each $h\in\mathcal{R}_{\varepsilon}(\CC_{\gamma})$,
and for $n\geq3$, we have 
\begin{align*}
\d(W_{h}) & \leq K\log n,\\
\f(W_{h}) & \leq K(\log n)^{2}.
\end{align*}
By Lemma \ref{lem:D-vs-d} we have $\v(W_{h})\leq\d(W_{h})+\f(W_{h})$,
so $\D\left(W_{h}\right)\le\d\left(W_{h}\right)$ and 
\[
\D(W_{h})\leq K\log n.
\]
We also have $\e(W_{h})\leq4\v(W_{h})$ by (\ref{eq:v-e-f-inequality}).
Hence by increasing $K$ if necessary we can also ensure
\[
\v(W_{h}),\e(W_{h})\leq K(\log n)^{2}.
\]

\subsection{Part I: The contribution from non-$\varepsilon$-adapted surfaces\label{subsec:Part-I:non-e-adapted}}

Our first goal is to control the contribution to $\E_{n}[\fix_{\gamma}]$
in (\ref{eq:main-proof-sum-of-expectations}) from non-$\varepsilon$-adapted
surfaces. Let $\mathcal{R_{\varepsilon}}^{(\text{non-}\varepsilon\text{-ad})}\left(\CC_{\gamma}\right)$
denote the set of morphisms $h:\mathcal{C}_{\gamma}\to W_{h}$ in
$\mathcal{R_{\varepsilon}}(\CC_{\gamma})$ such that $W_{h}$ is not
$\varepsilon$-adapted. In particular, such $W_{h}$ is boundary reduced
and $\f\left(W_{h}\right)>\d\left(W_{h}\right)$. 
\begin{prop}
\label{prop:contributions-from-non-adapted}There is a constant $A>0$
such that for any $c>0$, if $\ell_{w}(\gamma)\leq c\log n$, then
\[
\sum_{h\in\mathcal{R_{\varepsilon}}^{(\text{non-}\varepsilon\text{-ad})}\left(\CC_{\gamma}\right)}\mathbb{E}_{n}^{\emb}\left(W_{h}\right)\ll_{c}\frac{(\log n)^{A}}{n}.
\]
\end{prop}

\begin{proof}
We first do some counting. Let us count $h\in\mathcal{R_{\varepsilon}}^{(\text{non-}\varepsilon\text{-ad})}\left(\CC_{\gamma}\right)$
by their value of $\D(W_{h})$ and $\f(W_{h})$. By Corollary \ref{cor:facts-about-the-resolution}
every $h\in\mathcal{R_{\varepsilon}}^{(\text{non-}\varepsilon\text{-ad})}\left(\CC_{\gamma}\right)$
has 
\begin{equation}
\chi(W_{h})<-\f(W_{h})<-\d(W_{h}).\label{eq:chi < -f}
\end{equation}
Combining (\ref{eq:chi < -f}) with Lemma \ref{lem:D-vs-d} yields
\begin{equation}
0\leq\D(W_{h})\le\d\left(W_{h}\right)<\f(W_{h}).\label{eq:f>0}
\end{equation}
Notice that (\ref{eq:f>0}) implies $\f(W_{h})\geq1$. First we bound
the number of possible $W_{h}$ with $\D(W_{h})=\D_{0}$ and $\f(W_{h})=\f_{0}$
for fixed $\D_{0}<\f_{0}$. Note that in this case $\v(W_{h})=\v_{0}\eqdf\D_{0}+\f_{0}$.
We may over-count the number of $W_{h}$ with $\v_{0}$ vertices by
counting the number of $W_{h}$ together with a labeling of their
vertices by $[\v_{0}]$. We first construct the one-skeleton of such
a tiled surface: there are at most $\v_{0}^{~\v_{0}}$ choices for
the $a$-labeled edges, and also for the $b$-labeled edges etc. Because
$W_{h}$ are all boundary reduced, there is an octagon attached to
any closed $\left[a,b\right]\left[c,d\right]$ path, so the one-skeleton
completely determines the entire tiled surface. Hence there are at
most $\v_{0}^{~4\v_{0}}$ choices for $W_{h}$ with $\v(W_{h})=\v_{0}$. 

We also have to estimate how many ways there are to map $\mathcal{C}_{\gamma}$
into such a $W_{h}$. Fixing arbitrarily a vertex $v$ of $\mathcal{C}_{\gamma}$,
any morphism $\mathcal{C}_{\gamma}\to W_{h}$ is uniquely determined
by where $v$ goes; hence there are at most $\v_{0}$ morphisms and
so in total there are at most 
\begin{align*}
\v_{0}^{4\v_{0}+1} & \leq\v_{0}^{5\v_{0}}=(\D_{0}+\f_{0})^{5(\D_{0}+\f_{0})}\leq(2\f_{0})^{10\f_{0}}
\end{align*}
elements $h\in\mathcal{R_{\varepsilon}}^{(\text{non-}\varepsilon\text{-ad})}\left(\CC_{\gamma}\right)$
with $\D(W_{h})=\D_{0}$ and $\f(W_{h})=\f_{0}$. Hence there are
at most $K\log n\cdot(2\f_{0})^{10\f_{0}}$ elements $h\in\mathcal{R_{\varepsilon}}^{(\text{non-}\varepsilon\text{-ad})}\left(\CC_{\gamma}\right)$
with $\f(W_{h})=\f_{0}$.

We are going to use Theorem \ref{thm:E_n-emb-exact-expression} that
relates $\E_{n}^{\emb}\left(W_{h}\right)$ to a certain quantity $\Xi_{n}(W_{h})$.
By Proposition \ref{prop:Xi-bound-final-BR} there is $A_{0}>1$ such
that for $h\in\mathcal{R_{\varepsilon}}^{(\text{non-}\varepsilon\text{-ad})}\left(\CC_{\gamma}\right)$
\begin{equation}
\left|\Xi_{n}(W_{h})\right|\ll_{K}\left(A_{0}\D(W_{h})\right)^{A_{0}\D(W_{h})}\leq\left(A_{0}\f(W_{h})\right)^{A_{0}\f(W_{h})},\label{eq:ineq for Chi_n(W_h)}
\end{equation}
so by Theorem \ref{thm:E_n-emb-exact-expression}, Corollary \ref{cor:size-of-X_n},
and Lemma \ref{lem:Pochhammer-bounds} we get
\begin{eqnarray}
\E_{n}^{\emb}\left(W_{h}\right) & \stackrel{\mathrm{Thm~\ref{thm:E_n-emb-exact-expression}}}{=} & \frac{n!^{3}}{|\X_{n}|}\frac{(n)_{\v(W_{h})}(n)_{\f(W_{h})}}{\prod_{f}(n)_{\e_{f}(W_{h})}}\Xi_{n}\left(W_{h}\right)\stackrel{\mathrm{Cor.~\ref{cor:size-of-X_n}}}{\ll}\frac{(n)_{\v(W_{h})}(n)_{\f(W_{h})}}{\prod_{f}(n)_{\e_{f}(W_{h})}}\Xi_{n}\left(W_{h}\right)\nonumber \\
 & \stackrel{\mathrm{Lemma~\ref{lem:Pochhammer-bounds}}}{\ll_{K}} & n^{\chi(W_{h})}\Xi_{n}\left(W_{h}\right)\stackrel{\eqref{eq:ineq for Chi_n(W_h)}}{\ll_{K}}n^{\chi(W_{h})}\left(A_{0}\f(W_{h})\right)^{A_{0}\f(W_{h})}.\label{eq:for-overview}
\end{eqnarray}
Therefore, for every $1\le f_{0}\le K\left(\log n\right)^{2}$, 
\begin{eqnarray*}
\sum_{\substack{h\in\mathcal{R_{\varepsilon}}^{(\text{non-}\varepsilon\text{-ad})}\left(\CC_{\gamma}\right)\\
\f(W_{h})=\f_{0}
}
}\mathbb{E}_{n}^{\emb}\left(W_{h}\right) & \ll_{K} & (A_{0}\f_{0})^{A_{0}\f_{0}}\sum_{\substack{h\in\mathcal{R_{\varepsilon}}^{(\text{non-}\varepsilon\text{-ad})}\left(\CC_{\gamma}\right)\\
\f(W_{h})=\f_{0}
}
}n^{\chi(W_{h})}\stackrel{\eqref{eq:chi < -f}}{\le}(A_{0}\f_{0})^{A_{0}\f_{0}}\sum_{\substack{h\in\mathcal{R_{\varepsilon}}^{(\text{non-}\varepsilon\text{-ad})}\left(\CC_{\gamma}\right)\\
\f(W_{h})=\f_{0}
}
}n^{-\f_{0}}\\
 & \le & K\log n\left(\frac{\left(A_{0}\f_{0}\right)^{A_{0}}\left(2\f_{0}\right)^{10}}{n}\right)^{\f_{0}}\le K\log n\cdot\left(\frac{A_{0}^{A_{0}}2^{10}\left(K(\log n)^{2}\right)^{A_{0}+10}}{n}\right)^{\f_{0}}.
\end{eqnarray*}
So
\begin{eqnarray*}
\sum_{h\in\mathcal{R_{\varepsilon}}^{(\text{non-}\varepsilon\text{-ad})}\left(\CC_{\gamma}\right)}\mathbb{E}_{n}^{\emb}\left(W_{h}\right) & = & \sum_{\f_{0}=1}^{K(\log n)^{2}}\sum_{\substack{h\in\mathcal{R_{\varepsilon}}^{(\text{non-}\varepsilon\text{-ad})}\left(\CC_{\gamma}\right)\\
\f(W_{h})=\f_{0}
}
}\mathbb{E}_{n}^{\emb}\left(W_{h}\right)\\
 & \ll_{K} & K\log n\cdot\sum_{\f_{0}=1}^{K(\log n)^{2}}\left(\frac{A_{0}^{A_{0}}2^{10}\left(K(\log n)^{2}\right)^{A_{0}+10}}{n}\right)^{\f_{0}}\ll_{K}\frac{(\log n)^{2A_{0}+21}}{n},
\end{eqnarray*}
where the last inequality is based on that $\frac{A_{0}^{A_{0}}2^{10}\left(K(\log n)^{2}\right)^{A_{0}+10}}{n}\le\frac{1}{2}$
for $n\gg_{K}1$.
\end{proof}

\subsection{Part II: The contribution from $\varepsilon$-adapted surfaces\label{subsec:Part-II:-Counting}}

Write $\mathcal{R_{\varepsilon}}^{(\varepsilon\text{-ad})}(\CC_{\gamma})\subset\mathcal{R}_{\varepsilon}(\CC_{\gamma})$
for the collection of morphisms $h:\CC_{\gamma}\to W_{h}$ in $\mathcal{R}_{\varepsilon}(\CC_{\gamma})$
such that $W_{h}$ is $\varepsilon$-adapted. In light of Proposition
\ref{prop:contributions-from-non-adapted} it remains to deal with
the contributions to $\E_{n}[\fix_{\gamma}]$ from $\mathcal{R_{\varepsilon}}^{(\varepsilon\text{-ad})}(\CC_{\gamma})$.
Indeed we have by Proposition \ref{prop:contributions-from-non-adapted}
and (\ref{eq:main-proof-sum-of-expectations})

\begin{equation}
\E_{n}[\fix_{\gamma}]=\sum_{h\in\mathcal{R_{\varepsilon}}^{(\varepsilon\text{-ad})}(\CC_{\gamma})}\mathbb{E}_{n}^{\emb}\left(W_{h}\right)+O_{c}\left(\frac{(\log n)^{A}}{n}\right).\label{eq:final-proof-temp}
\end{equation}

Recall that if $W_{h}$ is $\varepsilon$-adapted, it is, in particular,
strongly boundary reduced, and so by \cite[\S\S 1.6]{MPasympcover},
$\mathbb{E}_{n}^{\emb}\left(W_{h}\right)=n^{\chi\left(W_{h}\right)}\left[1+O\left(n^{-1}\right)\right]$
as $n\to\infty$. By Theorem \ref{thm:asymptotic-non-effective},
$\E_{n}[\fix_{\gamma}]=1+O\left(n^{-1}\right)$. Comparing this with
(\ref{eq:final-proof-temp}), we conclude that there is exactly one
$h_{0}\in\mathcal{R}_{\varepsilon}(\CC_{\gamma})$ with $\chi\left(W_{h_{0}}\right)=0$.
This $h_{0}$ also satisfies that $W_{h_{0}}$ is $\varepsilon$-adapted\footnote{It can be shown that $h_{0}$ is the result of the OvB algorithm when
applied to the embedding $\CC_{\gamma}\hookrightarrow\left\langle \gamma\right\rangle \backslash\ts$
with $\ts$ the universal cover of $\Sigma_{2}$ -- see \cite[\S 2]{MPasympcover}.}. 

Still, we are missing some information about $\mathcal{R_{\varepsilon}}^{(\varepsilon\text{-ad})}(\CC_{\gamma})$
that we will need: for example, the ability to count how many $h:\CC_{\gamma}\to W_{h}$
there are in $\mathcal{R_{\varepsilon}}^{(\varepsilon\text{-ad})}(\CC_{\gamma})$
with different orders of contributions (i.e.~$n^{\chi(W_{h})}$)
to (\ref{eq:final-proof-temp}). We are going to use a trick to get
around this missing information. 

Let $\eta\in(0,\frac{1}{100})$ be the parameter provided by Proposition
\ref{prop:Xi-bound-final-e-adapted} for the current $\varepsilon=\frac{1}{32}$
(the reason for choosing $\eta$ like this now is just so that we
can momentarily apply Corollary \ref{cor:Em-lower-bound} and Proposition
\ref{prop:main-term-asymp}). Let $m$ be an auxiliary parameter given
by 
\[
m=\left\lceil \left(K\log n\right)^{1/\eta}\right\rceil 
\]
so that when $n\gg_{c}1$, for all $h\in\mathcal{R_{\varepsilon}}^{(\varepsilon\text{-ad})}(\CC_{\gamma})$,
$\D(W_{h})\leq K\log n\leq m^{\eta}$ and $\v(W_{h}),\e(W_{h}),\f(W_{h})\leq K(\log n)^{2}\leq m^{1/4}$.
Moreover, $(\log n)^{100}\ll_{c}m\ll_{c}(\log n)^{\frac{1}{\eta}}$.
To exploit the fact that each $\E_{n}^{\emb}(W_{h})$ is controlled
by $\E_{m}^{\emb}(W_{h})$ (Corollary \ref{cor:Em-lower-bound} and
Proposition \ref{prop:main-term-asymp}), we will at two points use
the inequality
\begin{align}
m & \geq\E_{m}[\fix_{\gamma}]\stackrel{\eqref{eq:main-proof-sum-of-expectations}}{=}\sum_{h\in\mathcal{R_{\varepsilon}}(\CC_{\gamma})}\mathbb{E}_{m}^{\emb}\left(W_{h}\right)\geq\sum_{h\in\mathcal{R_{\varepsilon}}^{(\varepsilon\text{-ad})}(\CC_{\gamma})}\mathbb{E}_{m}^{\emb}\left(W_{h}\right).\label{eq:m-domination}
\end{align}
We begin with
\begin{align}
\sum_{h\in\mathcal{R_{\varepsilon}}^{(\varepsilon\text{-ad})}(\CC_{\gamma})}\mathbb{E}_{n}^{\emb}\left(W_{h}\right) & \stackrel{\text{Prop. \ref{prop:main-term-asymp} }}{=}\sum_{h\in\mathcal{R_{\varepsilon}}^{(\varepsilon\text{-ad})}(\CC_{\gamma})}n^{\chi(W_{h})}\left[1+O_{c}\left(\frac{(\log n)^{4}}{n}\right)+O_{c}\left(\frac{m}{n}\frac{\E_{m}^{\emb}(W_{h})}{m^{\chi\left(W_{h}\right)}}\right)\right]\nonumber \\
 & =\sum_{h\in\mathcal{R_{\varepsilon}}^{(\varepsilon\text{-ad})}(\CC_{\gamma})}n^{\chi(W_{h})}\left[1+O_{c}\left(\frac{(\log n)^{4}}{n}\right)\right]+O_{c}\left(\frac{m}{n}\sum_{h\in\mathcal{R_{\varepsilon}}^{(\varepsilon\text{-ad})}(\CC_{\gamma})}\E_{m}^{\emb}(W_{h})\right)\nonumber \\
 & \stackrel{\eqref{eq:m-domination}}{=}\sum_{h\in\mathcal{R_{\varepsilon}}^{(\varepsilon\text{-ad})}(\CC_{\gamma})}n^{\chi(W_{h})}\left(1+O_{c}\left(\frac{(\log n)^{4}}{n}\right)\right)+O_{c}\left(\frac{m^{2}}{n}\right).\label{eq:first-run}
\end{align}
The middle estimate above used that $\chi(W_{h})\leq0$ for all $h\in\mathcal{R_{\varepsilon}}^{(\varepsilon\text{-ad})}(\CC_{\gamma})$,
and so $\left(\frac{n}{m}\right)^{\chi\left(W_{h}\right)}\le1$. The
contribution to (\ref{eq:first-run}) from $h_{0}$ is $1+O_{c}\left(\frac{(\log n)^{4}}{n}\right)$.
So we obtain
\begin{equation}
\sum_{h\in\mathcal{R_{\varepsilon}}^{(\varepsilon\text{-ad})}(\CC_{\gamma})}\mathbb{E}_{n}^{\emb}\left(W_{h}\right)=1+O_{c}\left(\frac{(\log n)^{4}}{n}\right)+O\left(\frac{m^{2}}{n}\right)+O\left(\sum_{\substack{h\in\mathcal{R_{\varepsilon}}^{(\varepsilon\text{-ad})}(\CC_{\gamma})\\
\chi(W_{h})<0
}
}n^{\chi(W_{h})}\right).\label{eq:intermediate-final-proof}
\end{equation}
To deal with the last error term, we relate it to the expectations
at level $m$. Indeed,
\begin{eqnarray*}
\left|\sum_{\substack{h\in\mathcal{R_{\varepsilon}}^{(\varepsilon\text{-ad})}(\CC_{\gamma})\\
\chi(W_{h})<0
}
}n^{\chi(W_{h})}\right| & = & \sum_{\substack{h\in\mathcal{R_{\varepsilon}}^{(\varepsilon\text{-ad})}(\CC_{\gamma})\\
\chi(W_{h})<0
}
}\left(\frac{n}{m}\right)^{\chi(W_{h})}m^{\chi(W_{h})}\leq\frac{m}{n}\sum_{\substack{h\in\mathcal{R_{\varepsilon}}^{(\varepsilon\text{-ad})}(\CC_{\gamma})}
}m^{\chi(W_{h})}\\
 & \stackrel{\text{Cor. \ref{cor:Em-lower-bound}}}{\ll} & \frac{m}{n}\sum_{\substack{h\in\mathcal{R_{\varepsilon}}^{(\varepsilon\text{-ad})}(\CC_{\gamma})}
}\E_{m}^{\emb}\left(W_{h}\right)\stackrel{\eqref{eq:m-domination}}{\leq}\frac{m^{2}}{n}.
\end{eqnarray*}
Incorporating this estimate into (\ref{eq:intermediate-final-proof})
gives
\begin{align*}
\sum_{h\in\mathcal{R_{\varepsilon}}^{(\varepsilon\text{-ad})}(\CC_{\gamma})}\mathbb{E}_{n}^{\emb}\left(W_{h}\right) & =1+O_{c}\left(\frac{(\log n)^{4}}{n}\right)+O\left(\frac{m^{2}}{n}\right)=1+O_{c}\left(\frac{(\log n)^{A}}{n}\right),
\end{align*}
where $A=\frac{2}{\eta}$. Combining this with (\ref{eq:final-proof-temp})
and increasing $A$ if necessary we obtain
\[
\E_{n}[\fix_{\gamma}]=1+O_{c}\left(\frac{(\log n)^{A}}{n}\right)
\]
as required. \uline{This concludes the proof of Theorem \mbox{\ref{thm:effective-error}}}.
$\square$
\begin{rem}
The arguments above show that 
\[
\sum_{\substack{h\in\mathcal{R_{\varepsilon}}^{(\varepsilon\text{-ad})}(\CC_{\gamma})}
}m^{\chi(W_{h})}\ll\sum_{\substack{h\in\mathcal{R_{\varepsilon}}^{(\varepsilon\text{-ad})}(\CC_{\gamma})}
}\E_{m}^{\emb}\left(W_{h}\right)\ll m,
\]
hence the number of elements of $h\in\mathcal{R_{\varepsilon}}^{(\varepsilon\text{-ad})}(\CC_{\gamma})$
with $\chi(W_{h})=\chi$ is $\ll m^{1-\chi}$. In general, given arbitrary
$\gamma\in\Gamma$, and $\epsilon>0$ we obtain by the same argument
that for some $\eta=\eta(\epsilon)>0$ we have
\[
\#\{\,h\in\mathcal{R_{\varepsilon}}^{(\varepsilon\text{-ad})}(\CC_{\gamma})\,:\,\chi(W_{h})=\chi\,\}\ll_{\epsilon}\left(\ell(\gamma)^{\frac{1}{\eta}}\right)^{1-\chi}.
\]
We mention this side-effect of our proof in case it is of independent
interest.
\end{rem}

{\small{}\bibliographystyle{amsalpha}
\bibliography{surface}

\providecommand{\bysame}{\leavevmode\hbox to3em{\hrulefill}\thinspace}
\providecommand{\MR}{\relax\ifhmode\unskip\space\fi MR }
\providecommand{\MRhref}[2]{%
  \href{http://www.ams.org/mathscinet-getitem?mr=#1}{#2}
}
\providecommand{\href}[2]{#2}
\begin{thebibliography}{{Hid}21}

\bibitem[Alo86]{Alon}
N.~Alon, \emph{Eigenvalues and expanders}, Combinatorica \textbf{6} (1986),
  no.~2, 83--96, Theory of computing (Singer Island, Fla., 1984). \MR{875835}

\bibitem[AM85]{AlonMilman}
N.~Alon and V.~D. Milman, \emph{{$\lambda_1,$} isoperimetric inequalities for
  graphs, and superconcentrators}, J. Combin. Theory Ser. B \textbf{38} (1985),
  no.~1, 73--88. \MR{782626}

\bibitem[Apo76]{ApostolNT}
T.~M. Apostol, \emph{Introduction to analytic number theory}, Springer-Verlag,
  New York-Heidelberg, 1976, Undergraduate Texts in Mathematics. \MR{0434929}

\bibitem[BBD88]{BBD}
P.~Buser, M.~Burger, and J.~Dodziuk, \emph{Riemann surfaces of large genus and
  large {$\lambda_1$}}, Geometry and analysis on manifolds ({K}atata/{K}yoto,
  1987), Lecture Notes in Math., vol. 1339, Springer, Berlin, 1988, pp.~54--63.
  \MR{961472}

\bibitem[BC19]{BordenaveCollins}
C.~Bordenave and B.~Collins, \emph{Eigenvalues of random lifts and polynomials
  of random permutation matrices}, Ann. of Math. (2) \textbf{190} (2019),
  no.~3, 811--875. \MR{4024563}

\bibitem[BCP21]{budzinski2021diameter}
T.~Budzinski, N.~Curien, and B.~Petri, \emph{The diameter of random
  {B}ely{\u\i} surfaces}, Algebraic \& Geometric Topology \textbf{21} (2021),
  no.~6, 2929--2957.

\bibitem[Bea84]{Beardon}
A.~F. Beardon, \emph{A primer on {R}iemann surfaces}, London Mathematical
  Society Lecture Note Series, vol.~78, Cambridge University Press, Cambridge,
  1984. \MR{808581}

\bibitem[BH99]{BridsonHaefliger}
M.~R. Bridson and A.~Haefliger, \emph{Metric spaces of non-positive curvature},
  Grundlehren der Mathematischen Wissenschaften [Fundamental Principles of
  Mathematical Sciences], vol. 319, Springer-Verlag, Berlin, 1999. \MR{1744486}

\bibitem[BK93]{Barzdin1993}
Ya.~M. Barzdin and A.N. Kolmogorov, \emph{On the realization of networks in
  three-dimensional space}, pp.~194--202, Springer Netherlands, Dordrecht,
  1993.

\bibitem[BM04]{BrooksMakover}
R.~Brooks and E.~Makover, \emph{Random construction of {R}iemann surfaces}, J.
  Differential Geom. \textbf{68} (2004), no.~1, 121--157. \MR{2152911}

\bibitem[Bol88]{Bollobas}
B.~Bollob\'{a}s, \emph{The isoperimetric number of random regular graphs},
  European J. Combin. \textbf{9} (1988), no.~3, 241--244. \MR{947025}

\bibitem[Bor16]{Borthwick}
D.~Borthwick, \emph{Spectral theory of infinite-area hyperbolic surfaces},
  second ed., Progress in Mathematics, vol. 318, Birkh\"{a}user/Springer,
  [Cham], 2016. \MR{3497464}

\bibitem[Bor20]{bordenave2015new}
C.~Bordenave, \emph{A new proof of {F}riedman's second eigenvalue theorem and
  its extension to random lifts}, Ann. Sci. Ec. Norm. Sup{\'e}r. \textbf{53}
  (2020), no.~6, 1393--1439, available at arxiv:1502.04482.

\bibitem[BS87a]{BirmanSeries}
J.~S. Birman and C.~Series, \emph{Dehn's algorithm revisited, with applications
  to simple curves on surfaces}, Combinatorial group theory and topology
  ({A}lta, {U}tah, 1984), Ann. of Math. Stud., vol. 111, Princeton Univ. Press,
  Princeton, NJ, 1987, pp.~451--478. \MR{895628}

\bibitem[BS87b]{BroderShamir}
A.~Broder and E.~Shamir, \emph{On the second eigenvalue of random regular
  graphs.}, The 28th Annual Symposium on Foundations of Computer Science, 1987,
  pp.~286--294.

\bibitem[Bus10]{Buser}
P.~Buser, \emph{Geometry and spectra of compact {R}iemann surfaces}, Modern
  Birkh\"{a}user Classics, 2010, Reprint of the 1992 edition. \MR{2742784}

\bibitem[FP22]{friedman2020note}
J.~Friedman and D.~Puder, \emph{A note on the trace method for random regular
  graphs}, Israel Journal of Mathematics (2022+), to appear, arXiv:2006.13605.

\bibitem[Fri03]{FriedmanRelExp}
J.~Friedman, \emph{Relative expanders or weakly relatively {R}amanujan graphs},
  Duke Math. J. \textbf{118} (2003), no.~1, 19--35. \MR{1978881}

\bibitem[Fri08]{Friedman}
\bysame, \emph{A proof of {A}lon's second eigenvalue conjecture and related
  problems}, Mem. Amer. Math. Soc. \textbf{195} (2008), no.~910, viii+100.
  \MR{2437174}

\bibitem[Gam06]{gamburd2006poisson}
A.~Gamburd, \emph{Poisson--{D}irichlet distribution for random {B}elyi
  surfaces}, The Annals of Probability \textbf{34} (2006), no.~5, 1827--1848.

\bibitem[GJ78]{GelbartJacquet}
S.~Gelbart and H.~Jacquet, \emph{A relation between automorphic representations
  of {${\rm GL}(2)$} and {${\rm GL}(3)$}}, Ann. Sci. \'{E}cole Norm. Sup. (4)
  \textbf{11} (1978), no.~4, 471--542. \MR{533066}

\bibitem[GK19]{GolubevKamber}
K.~Golubev and A.~Kamber, \emph{Cutoff on hyperbolic surfaces}, Geom. Dedicata
  \textbf{203} (2019), 225--255. \MR{4027593}

\bibitem[Hej76]{Hejhal1}
D.~A. Hejhal, \emph{The {S}elberg trace formula for {${\rm PSL}(2,{\bf R})$}.
  {V}ol. {I}}, Lecture Notes in Mathematics, Vol. 548, Springer-Verlag,
  Berlin-New York, 1976. \MR{0439755}

\bibitem[Hej83]{Hejhal2}
\bysame, \emph{The {S}elberg trace formula for {${\rm PSL}(2,\,{\bf R})$}.
  {V}ol. 2}, Lecture Notes in Mathematics, vol. 1001, Springer-Verlag, Berlin,
  1983. \MR{711197}

\bibitem[{Hid}21]{HideWP}
W.~{Hide}, \emph{{Spectral gap for Weil-Petersson random surfaces with cusps}},
  preprint, arXiv:2107.14555 (2021).

\bibitem[HM21]{HideMagee}
W.~{Hide} and M.~{M}agee, \emph{{Near optimal spectral gaps for hyperbolic
  surfaces}}, preprint, arXiv:2107.05292 (2021).

\bibitem[Hub74]{Huber}
H.~Huber, \emph{\"{U}ber den ersten {E}igenwert des {L}aplace-{O}perators auf
  kompakten {R}iemannschen {F}l\"{a}chen}, Comment. Math. Helv. \textbf{49}
  (1974), 251--259. \MR{365408}

\bibitem[Hum18]{Humphries}
P.~Humphries, \emph{Density theorems for exceptional eigenvalues for congruence
  subgroups}, Algebra Number Theory \textbf{12} (2018), no.~7, 1581--1610.
  \MR{3871503}

\bibitem[Hur02]{hurwitz1902ueber}
A.~Hurwitz, \emph{Ueber die anzahl der {R}iemann'schen fl{\"a}chen mit
  gegebenen verzweigungspunkten}, Mathematische Annalen \textbf{55} (1902),
  no.~1, 53--66.

\bibitem[Hux86]{Huxley}
M.~N. Huxley, \emph{Exceptional eigenvalues and congruence subgroups}, The
  {S}elberg trace formula and related topics ({B}runswick, {M}aine, 1984),
  Contemp. Math., vol.~53, Amer. Math. Soc., Providence, RI, 1986,
  pp.~341--349. \MR{853564}

\bibitem[Iwa89]{Iwaniec89}
H.~Iwaniec, \emph{Selberg's lower bound of the first eigenvalue for congruence
  groups}, Number theory, trace formulas and discrete groups ({O}slo, 1987),
  Academic Press, Boston, MA, 1989, pp.~371--375. \MR{993327}

\bibitem[Iwa96]{Iwaniec96}
\bysame, \emph{The lowest eigenvalue for congruence groups}, Topics in
  geometry, Progr. Nonlinear Differential Equations Appl., vol.~20,
  Birkh\"{a}user Boston, Boston, MA, 1996, pp.~203--212. \MR{1390315}

\bibitem[Iwa02]{Iwaniecbook}
\bysame, \emph{Spectral methods of automorphic forms}, second ed., Graduate
  Studies in Mathematics, vol.~53, American Mathematical Society, Providence,
  RI; Revista Matem\'{a}tica Iberoamericana, Madrid, 2002. \MR{1942691}

\bibitem[Jen84]{Jenni}
F.~Jenni, \emph{\"{U}ber den ersten {E}igenwert des {L}aplace-{O}perators auf
  ausgew\"{a}hlten {B}eispielen kompakter {R}iemannscher {F}l\"{a}chen},
  Comment. Math. Helv. \textbf{59} (1984), no.~2, 193--203. \MR{749104}

\bibitem[JL70]{JacquetLanglands}
H.~Jacquet and R.~P. Langlands, \emph{Automorphic forms on {${\rm GL}(2)$}},
  Lecture Notes in Mathematics, Vol. 114, Springer-Verlag, Berlin-New York,
  1970. \MR{0401654}

\bibitem[Kim03]{KIM}
H.~H. Kim, \emph{Functoriality for the exterior square of {${\rm GL}_4$} and
  the symmetric fourth of {${\rm GL}_2$}}, J. Amer. Math. Soc. \textbf{16}
  (2003), no.~1, 139--183, With appendix 1 by D. Ramakrishnan and appendix 2 by
  H. H. Kim and P. Sarnak. \MR{1937203}

\bibitem[KS02]{KimShahidi}
H.~H. Kim and F.~Shahidi, \emph{Functorial products for {${\rm GL}_2\times{\rm
  GL}_3$} and the symmetric cube for {${\rm GL}_2$}}, Ann. of Math. (2)
  \textbf{155} (2002), no.~3, 837--893, With an appendix by Colin J. Bushnell
  and Guy Henniart. \MR{1923967}

\bibitem[LP81]{LP}
P.~D. Lax and R.~S. Phillips, \emph{The asymptotic distribution of lattice
  points in {E}uclidean and non-{E}uclidean spaces}, Functional analysis and
  approximation ({O}berwolfach, 1980), Internat. Ser. Numer. Math., vol.~60,
  Birkh\"{a}user, 1981, pp.~373--383. \MR{650290}

\bibitem[LPS88]{LPS}
A.~Lubotzky, R.~Phillips, and P.~Sarnak, \emph{Ramanujan graphs}, Combinatorica
  \textbf{8} (1988), no.~3, 261--277. \MR{963118}

\bibitem[LRS95]{LRS}
W.~Luo, Z.~Rudnick, and P.~Sarnak, \emph{On {S}elberg's eigenvalue conjecture},
  Geom. Funct. Anal. \textbf{5} (1995), no.~2, 387--401. \MR{1334872}

\bibitem[LS04]{LiebeckShalev}
M.~W. Liebeck and A.~Shalev, \emph{Fuchsian groups, coverings of {R}iemann
  surfaces, subgroup growth, random quotients and random walks}, J. Algebra
  \textbf{276} (2004), no.~2, 552--601. \MR{2058457}

\bibitem[LW21]{lipnowski2021optimal}
M.~Lipnowski and A.~Wright, \emph{Towards optimal spectral gaps in large
  genus}, 2021, Preprint, arXiv:2103.07496.

\bibitem[Med78]{Mednyhk}
A.~D. Mednyhk, \emph{Determination of the number of nonequivalent coverings
  over a compact {R}iemann surface}, Dokl. Akad. Nauk SSSR \textbf{239} (1978),
  no.~2, 269--271. \MR{490616}

\bibitem[Mir13]{MirzakhaniRandom}
M.~Mirzakhani, \emph{Growth of {W}eil-{P}etersson volumes and random hyperbolic
  surfaces of large genus}, J. Differential Geom. \textbf{94} (2013), no.~2,
  267--300. \MR{3080483}

\bibitem[MN20]{MageeNaud}
M.~Magee and F.~Naud, \emph{Explicit spectral gaps for random covers of
  {R}iemann surfaces}, Publ. Math. Inst. Hautes \'{E}tudes Sci. \textbf{132}
  (2020), 137--179. \MR{4179833}

\bibitem[MN21]{MageeNaud2}
\bysame, \emph{Extension of {A}lon's and {F}riedman's conjectures to {S}chottky
  surfaces}, 2021, Preprint, arXiv:2106.02555.

\bibitem[Mon22]{Monk}
L.~Monk, \emph{Benjamini-schramm convergence and spectrum of random hyperbolic
  surfaces of high genus}, Analysis \& PDE (2022), to appear, arXiv:2002.00869.

\bibitem[MP20]{MPasympcover}
M.~Magee and D.~Puder, \emph{The asymptotic statistics of random covering
  surfaces}, Preprint, arXiv:2003.05892, v5, 2020.

\bibitem[MP21]{MPcore}
\bysame, \emph{Core surfaces}, Preprint, arXiv:2108.00717, 2021.

\bibitem[Nil91]{Nilli}
A.~Nilli, \emph{On the second eigenvalue of a graph}, Discrete Math.
  \textbf{91} (1991), no.~2, 207--210. \MR{1124768}

\bibitem[Pat76]{Patterson}
S.~J. Patterson, \emph{The limit set of a {F}uchsian group}, Acta Math.
  \textbf{136} (1976), no.~3-4, 241--273. \MR{0450547}

\bibitem[PP15]{PP15}
D.~Puder and O.~Parzanchevski, \emph{Measure preserving words are primitive},
  Journal of the American Mathematical Society \textbf{28} (2015), no.~1,
  63--97.

\bibitem[Pud14]{puder2014primitive}
D.~Puder, \emph{Primitive words, free factors and measure preservation}, Israel
  Journal of Mathematics \textbf{201} (2014), no.~1, 25--73.

\bibitem[Pud15]{PUDER}
\bysame, \emph{Expansion of random graphs: new proofs, new results}, Invent.
  Math. \textbf{201} (2015), no.~3, 845--908. \MR{3385636}

\bibitem[Sar87]{Sarnak}
P.~Sarnak, \emph{Statistical properties of eigenvalues of the {H}ecke
  operators}, Analytic number theory and {D}iophantine problems ({S}tillwater,
  {OK}, 1984), Progr. Math., vol.~70, Birkh\"{a}user Boston, Boston, MA, 1987,
  pp.~321--331. \MR{1018385}

\bibitem[Sel56]{Selberg}
A.~Selberg, \emph{Harmonic analysis and discontinuous groups in weakly
  symmetric {R}iemannian spaces with applications to {D}irichlet series}, J.
  Indian Math. Soc. (N.S.) \textbf{20} (1956), 47--87. \MR{88511}

\bibitem[Sel65]{SelbergFourier}
\bysame, \emph{On the estimation of {F}ourier coefficients of modular forms},
  Proc. {S}ympos. {P}ure {M}ath., {V}ol. {VIII}, Amer. Math. Soc., Providence,
  R.I., 1965, pp.~1--15. \MR{0182610}

\bibitem[Sta83]{stallings1983topology}
J.~R. Stallings, \emph{Topology of finite graphs}, Inventiones mathematicae
  \textbf{71} (1983), no.~3, 551--565.

\bibitem[VO04]{VershikOkounkov}
A.~M. Vershik and A.~Yu. Okun'kov, \emph{A new approach to representation
  theory of symmetric groups. {II}}, Zap. Nauchn. Sem. S.-Peterburg. Otdel.
  Mat. Inst. Steklov. (POMI) \textbf{307} (2004), no.~Teor. Predst. Din. Sist.
  Komb. i Algoritm. Metody. 10, 57--98, 281. \MR{2050688}

\bibitem[Wei48]{Weil}
A.~Weil, \emph{On some exponential sums}, Proc. Nat. Acad. Sci. U.S.A.
  \textbf{34} (1948), 204--207. \MR{27006}

\bibitem[Wri20]{wright2020tour}
A.~Wright, \emph{A tour through {M}irzakhani’s work on moduli spaces of
  {R}iemann surfaces}, Bulletin of the American Mathematical Society
  \textbf{57} (2020), no.~3, 359--408.

\bibitem[WX22]{wu2021random}
Y.~Wu and Y.~Xue, \emph{Random hyperbolic surfaces of large genus have first
  eigenvalues greater than $\frac{3}{16}-\epsilon$}, Geometric and Functional
  Analysis (2022), to appear, arXiv:2102.05581.

\end{thebibliography}
}{\small\par}

\noindent Michael Magee, Department of Mathematical Sciences, Durham
University, Lower Mountjoy, DH1 3LE Durham, United Kingdom

\noindent \texttt{michael.r.magee@durham.ac.uk}~\\

\noindent Frédéric Naud, IMJ-PRG - UMR7586, Université Pierre et Marie
Curie, 4 place Jussieu, 75252 Paris Cedex 05 \\
\texttt{frederic.naud@imj-prg.fr}~\\

\noindent Doron Puder, School of Mathematical Sciences, Tel Aviv University,
Tel Aviv, 6997801, Israel\\
\texttt{doronpuder@gmail.com}
\end{document}